\documentclass[leqno]{report}
\usepackage{amsthm}
\usepackage{amssymb}
\usepackage{amscd}
\usepackage{amsxtra}     
\usepackage{rac}                        
\usepackage{dbl12}
\usepackage{verbatim}
\usepackage{hyperref}
\usepackage{graphicx}

\theoremstyle{plain}
\newtheorem{theorem}{Theorem}
\newtheorem{proposition}[theorem]{Proposition}
\newtheorem{corollary}[theorem]{Corollary}
\newtheorem{lemma}[theorem]{Lemma}

\newtheorem{convention}[theorem]{Convention}

\theoremstyle{definition}

\newtheorem{example}[theorem]{Example}

\theoremstyle{remark}

\newtheorem*{blanktheorem}{Theorem}
\include{header}

\numberwithin{theorem}{chapter}        

\newcommand{\Em}{{\bf M}}

\newcommand{\op}[1]{#1^\vee}
\newcommand{\vir}{\text{vir}}
\newcommand{\jay}{\text{\bf{J}}}
\newcommand{\cc}{\mathfrak{c}}
\newcommand{\rrr}{\mathfrak{r}}
\newcommand{\sss}{\mathfrak{s}}
\newcommand{\redge}{\rho}
\newcommand{\sedge}{\sigma}
\newcommand{\alphatw}[2]{\widetilde{\alpha}_{#1}^{#2}}
\newcommand{\Atw}{\text{\bf{A}}}
\newcommand{\Htw}{\widetilde{H}}
\newcommand{\Ac}{\widetilde{\Atw}}

\newcommand{\Pro}{\text{\bf{P}}}
\newcommand{\pvac}{\Pro_\emptyset}
\newcommand{\Protw}{\widetilde{\Pro}}
\newcommand{\Qop}{\text{\bf{Q}}}
\newcommand{\kk}{\Bbbk}
\newcommand{\gl}{\mathfrak{gl}}
 \newcommand{\psibar}{\overline{\psi}}
\newcommand{\mbar}{\overline{\mathcal{M}}}

\newcommand{\End}{\text{End}}
\newcommand{\B}{\mathcal{B}}
\newcommand{\Q}{\mathbb{Q}}
\newcommand{\R}{\mathbb{R}}

\newcommand{\Mbar}{\overline{\mathcal{M}}}
\newcommand{\C}{\mathbb{C}}
\newcommand{\Z}{\mathbb{Z}}
\newcommand{\BZ}{\mathcal{B}\Z}
\newcommand{\proj}{\mathbb{P}}
\newcommand{\zero}{\mathbf{0}}
\newcommand{\finity}{\boldsymbol{\infty}}
\newcommand{\ev}{\mathrm{ev}}
\newcommand{\floor}[1]{{\left\lfloor#1\right\rfloor}}
\newcommand{\leftover}[1]{{\left\langle#1\right\rangle}}
\newcommand{\E}{\mathbb{E}}
\newcommand{\Aut}{\mathrm{Aut}}
\newcommand{\aut}{\mathrm{Aut}}

\newcommand{\infwedge}{\bigwedge^{\frac{\infty}{2}}}

\newcommand{\Hur}{\mathrm{Hur}}

\newcommand{\D}{ {\bf D}}

\newcommand{\vac}{\left|0\right\rangle}
\newcommand{\Gauss}[3]{
{}_2F_1\left(\begin{array}{c}#1,1 \\ #2\end{array};#3\right)}
\newcommand{\Gaussfull}[4]{
{}_2F_1\left(\begin{array}{c}#1,#2 \\ #3\end{array};#4\right)}

\begin{document}

\bibliographystyle{alpha}


\titlepage{Equivariant Gromov-Witten theory of one dimensional stacks}{Paul D. Johnson}{Doctor of Philosophy}
{Mathematics} {2009}
{ Professor Yongbin Ruan, Chair \\
  Professor Igor Dolgachev \\
  Professor Igor Kriz \\
  Professor Karen E. Smith \\
  Associate Professor Leopoldo A. Pando Zayas
   }


\initializefrontsections


\setcounter{page}{1}

\dedicationpage{for my parents}


\startacknowledgementspage

Without my advisor, Yongbin Ruan, this work would not have been possible.  From the beginning he listened to what I was interested in and guided me towards problems that would fit with those interests, even if this diverged from his own interests.  I've learned as much from countless casual chats with him than I have from any class.  Finally, he's been incredibly encouraging and patient as I've slowly written this thesis.  To him, my deepest thanks.

There are so many other people from the math departments at Wisconsin and Michigan that have helped me over the years I don't know where to begin.  Professors, staff, fellow graduate students, and undergraduates I've taught all have made a lasting impression on me.  I've learned a lot of math from them, but just as important was their help in navigating through a university and through life in general.  These two math departments have been more than an office to me - they've been a home, filled with a warm and wonderful family.

My collaborators Renzo Cavalieri, Hannah Markwig, Rahul Pandharipande, and Hsian-Hua Tseng showed me how research is done, and how to relax afterwards.   Not only did some of our work together play a role in my thesis, working with them kept me sane and happy through the whole process.

Finally, Natalia provided me with a quiet place where much of this was written, not to mention much of my joy the last three years.  You've given me so much, Nachan.


\tableofcontents

\startthechapters

\chapter{Introduction}
\label{intro}

In their trilogy \cite{OP1} \cite{OP2} \cite{OP3}, Okounkov and Pandharipande completely determine the Gromov-Witten theory of curves.  This thesis is the beginning of a program extending their results to stacky curves.  The logical starting point of the trilogy is the second paper,  \cite{OP2}, which presents an explicit description of the $\C^*$-equivariant Gromov-Witten theory of $\proj^1$ in terms of operator expectations for the infinite wedge.  Our main result is an analogous formula for the $\C^*$-equivariant Gromov-Witten theory of stacky toric $\proj^1$, and an exploration of some of the immediate consequences of this operator expression, in particular that it satisfies the 2-Toda hierarchy.  We first explain what this means in simple terms by recalling the basics of Gromov-Witten theory and the work of the Kyoto school on integrable hierarchies.  After that we give an overview of our methods, beginning with a summary of Okounkov and Pandharipande's methods, which we closely follow.

\section{Gromov-Witten Theory}

Gromov-Witten theory studies the enumerative geometry of curves in a space.  First, the moduli space of stable maps $\mbar_{g,n}(X,\beta)$, where $\beta\in H_2(X,\Z)$, is constructed.  This Deligne-Mumford stack parameterizes pairs consisting of a nodal genus $g$ nodal curves $\mathcal{C}$ with $n$ marked points $p_i$, along with a holomorphic map $f:\mathcal{C}\to X$ so that $f_*[\mathcal{C}]=\beta$, that satisfy certain stability conditions which guarantee that such maps have only finitely many automorphisms.  This stack is generally very singular, but the foundational result of Gromov-Witten theory is that nevertheless it has a virtual fundamental class - a homology class of dimension $(\dim X-3)(1-g)+\langle c_1(TX), \beta\rangle +n$.  Using this class, Gromov-Witten theory studies the intersection theory of $\mbar_{g,n}(X,\beta)$ as though it were a smooth orbifold.

Gromov-Witten theory focuses on the intersections of two types of cohomology classes.  For $1\leq i\leq n$ there is a line bundle $L_i$, whose fiber over $\mbar_{g,n}(X,\beta)$ is $T^*\mathcal{C}_{p_i}$, the cotangent space to $\mathcal{C}$ at the $i$th marked point.  The psi classes are the first chern classes of these line bundles: $\psi_i=c_1(L_i)$.  The second type of class are those classes pulled back from $X$.  For $1\leq i\leq n$ there are evaluation maps $\ev_i:\mbar_{g,n}(X,\beta)\to X$, which takes a stable map to the image of the $i$th marked point: $\ev_i([\mathcal{C}, f])=f(p_i)$.  Using these evaluation maps, any cohomology class $\alpha\in H^*(X)$ may be pulled back $\mbar_{g,n}(X,\beta)$.

A Gromov-Witten invariant is the integration of a product of $\psi$ classes and pulled back classes from $X$ against the virtual fundamental class.  In extremely nice situations, these invariants are enumerative: they count actual curves in $X$ meeting certain conditions imposed by the cohomology classes used.  For instance, if we denote by $N_d$ the number of degree $d$ rational curves through $3d-1$ points in $\proj^2$, then we have $$N_d=\int_{\mbar_{0,3d-1}(\proj^2, d)} \prod_{i=1}^{3d-1} \ev_i^*(pt).$$
In general, however, Gromov-Witten invariants will only give ``virtual'' counts of curves.  A $G$ action on $X$ induces a $G$ action on $\mbar_{g,n}(X,\beta)$.  The virtual fundamental class may be made equivariant, and equivariant cohomology classes may be pulled back and integrated.

With a few minor complications described in detail in the next chapter, the above story holds true when $X$ is not a smooth projective variety, but rather a Deligne-Mumford stack or orbifold. Recall that an orbifold $\mathcal{X}$ is a topological space $X$ with some extra structure: each point $x$ has a neighborhood $U_x$, an isotropy group $G_x$ (possibly trivial), and an isomorphism $U_x\cong \C^n/G_x$ that should satisfy some compatibility.  We say $\mathcal{X}$ is effective if the generic point has trivial isotropy group - then each $G_x$ acts on $\C^n$ without a kernel.  We will focus for now on the effective case, and return to the ineffective case later.  We define $\mathcal{C}_{r,s}$ to be $\proj^1$ as a topological space, but at $0$ to be isomorphic to $\C/\Z_r$, and at $\infty$ it is isomorphic to $\C/\Z_s$.  The orbifolds $\mathcal{C}_{r,s}$ are the only effective orbicurves to have a $\C^*$ action.  The main result of this thesis is an algebraic framework - to be described momentarily - that computes all the equivariant Gromov-Witten invariants of the orbifolds $\mathcal{C}_{r,s}$.

An important reason for studying Gromov-Witten invariants is their recursive structure.  The ``boundary'' strata of $\mbar_{g,n}(X,\beta)$ - those maps where the domain curve $\mathcal{C}$ is nodal - decompose naturally as products of simpler such moduli spaces.  This yields various recursions among Gromov-Witten invariants.  Famously, Kontsevich used a simple such recursion to calculate the $N_d$ mentioned above.  To express these recursions, it is convenient to package all the Gromov-Witten invariants of a space $X$ into a generating function: we introduce formal variables to keep track of the degree and genus of the map, as well as variables recording which cohomology class is pulled back at each marked point, and which power of $\psi_i$ is used.  The generating function $F_X$ is then a formal power series in these variables, where the coefficient of each monomial is the corresponding Gromov-Witten invariant.

Packaged in this way, differential operators $D$ with $D F_X=0$ give recursions among Gromov-Witten invariants, and nice recursions can be written in terms of differential operators of this form.  Furthermore, nice families of recursions give rise to commuting families of differential operators $D$, all of which annihilate $F_X$.  Maximal families of such commuting operators form an integrable hierarchy.  Saying that the Gromov-Witten theory of a space satisfies an integrable hierarchy corresponds to saying that there is a particularly nice recursive algorithm for computing all the Gromov-Witten invariants of a space from a certain base set.

It is of great interest when Gromov-Witten invariants satisfy integrable hierarchies.  A famous example of this is the Witten-Kontsevich theorem, which asserts that the Gromov-Witten theory of a point (i.e., the intersection of $\psi$ classes on the moduli space of curves $\mbar_{g,n}$) satisfies the KdV hierarchy.  Okounkov and Pandharipande used their operator formalism to show that the equivariant Gromov-Witten theory of $\proj^1$ is a $\tau$ function of an integrable hierarchy known as the 2-Toda hierarchy.  A $\tau$ function is simply a function of the form $\tau=e^F$, where $F$ is a solution of the hierarchy; this change is made because the equations of the hierarchy take a convenient form when written in terms of $\tau$ functions.

Our main result is that the Gromov-Witten theory of $\mathcal{C}_{r,s}$ also satisfies the 2-Toda hierarchy:

\begin{blanktheorem}
Let $\tau_{\mathcal{X}}$ be the generating function for the equivariant Gromov-Witten of $\mathcal{X}=\mathcal{C}_{r,s}$.  Then, after an explicit linear change of variables depending on $r$ and $s$, $\tau_{\mathcal{X}}$ is a $\tau$ function for the 2-Toda hierarchy.
\end{blanktheorem}

Using a different approach \cite{MTequivariant}, Milanov and Tseng have already obtained this result.  However, their method depends on extending Givental's formalism to the orbifold setting, which has not yet been completely carried out.  Furthermore, our operator formalism should later prove useful in investigating the Gromov-Witten theory of more complicated orbifold curves.
 
We now explain in more detail our operator formalism, and how it connects to the 2-Toda hierarchy.

\section{Integrable Hierarchies and the Kyoto School}

The operator formalism we use and its relationship to the 2-Toda hierarchy was developed by the Kyoto school (see \cite{MJD} for a gentle introduction), and is an infinite dimensional analog of the Pl\"ucker embedding of the Grassmannian in projective space.  To make the 2-Toda hierarchy more concrete, we explain this now.

Recall the Pl\"ucker embedding of $G(k,n)$ in $\proj^{{{n}\choose {k}}-1}$.  Let $V$ be an $n$-dimensional vector space with basis $e_1,\dots, e_n$. Given a $k$-dimensional subspace $U$ of $V$, and choosing a basis $u_1,\dots, u_k$ we form the vector $u_1\wedge\cdots\wedge u_k\in\bigwedge^k V$.  Choosing a different basis of $U$ only changes this vector by multiplying by a scalar, and so we have a well defined map from $G(k,n)\to \proj(\bigwedge^k V)$, which is in fact an embedding.

Not every vector in $\bigwedge^k V$ is indecomposable (i.e., of the form $u_1\wedge\cdots\wedge u_k$).  In fact, $\bigwedge^k V$ has a basis given by Pl\"ucker coordinates: vectors of the form $e_{i_1}\wedge e_{i_2}\wedge\cdots\wedge e_{i_k}$.  If we put together the $u_i$ as a $k\times n$ matrix, then expanding $u_1\wedge\cdots\wedge u_k$ out in Pl\"ucker coordinates corresponds to looking at the determinants of all $k\times k$ minors.  However, since these minors overlap their determinants are not independent, and satisfy certain quadratic relations known as the Pl\"ucker equations.  The Pl\"ucker equations give defining equations for $G(k,n)$ inside of $\proj(\wedge^k V)$.

Recall also that the group $GL(V)$ obviously takes a $k$ dimensional subspace to another $k$-dimensional subspace, and acts transitively on these subspaces; thus another way to view the Grassmannian inside $\proj(\wedge^k V)$ is as the orbit of the element $e_1\wedge\cdots\wedge e_k$ under the induced action of $GL(V)$ on $\bigwedge V$.  This is the viewpoint that we will generalize for integrable systems.

To generalize this story, we now make $V$ an infinite dimensional vector space, with basis $e_i, i\in \Z+1/2$ a half integer.  The infinite wedge $\infwedge V$ has a basis consisting of those vectors of the form $e_{i_1}\wedge e_{i_2}\wedge \cdots$ for all decreasing sequences $i_1,i_2,\dots $ of half integers such that $i_k+k+1/2$ is constant for $k$ sufficiently large.  The infinite wedge contains a distinguished element, the vacuum vector $\vac$, consisting of the wedge product of all $e_\ell$ with $\ell\in\Z+1/2, \ell<0$.

The lie algebra $\gl_{\infty}$ consists of those infinite matrices acting on $V$ that have only finitely many nonzero diagonals.  The lie algebra $\gl_{\infty}$ does not quite act on the infinite wedge: there is seem issue with the action of diagonal matrices, that is carefully described in Section \ref{fock space}.  The infinite wedge is, however, a projective representation of this lie algebra.  The lie group $GL(\infty)$ consists of products of exponentials of things in $\gl_\infty$.  An obvious infinite dimensional Grassmannian is then the orbit of $GL(\infty)$ on the vacuum, which parameterizes infinite dimensional subspaces of $V$ that only differ from a fixed subspace in a finite dimensional way.  

The infinite wedge and $\gl_{\infty}$ is the algebraic framework we will use to express the Gromov-Witten theory:  we will produce several explicit infinite dimensional matrices, and the Gromov-Witten invariants of $\mathcal{C}_{r,s}$ will be coefficients that describe their action on the infinite wedge.  Once we have done this, it will follow relatively easily from the work of the Kyoto that the Gromov-Witten invariants satisfy the 2-Toda hierarchy, as we now describe.

The simplest integrable hierarchy to describe from this point of view is the KP hierarchy: it describes this Grassmannian among all of the infinite wedge.  Similarly, the KdV hierarchy mentioned above is also the orbit of an infinite dimensional group.  The loop group of $SL_2$ has a natural embedding inside $GL(\infty)$, and the KdV hierarchy is the description of the vacuum vector under this orbit.  The 2-Toda hierarchy is slightly more complicated: rather than describing the orbit of the vacuum, it simultaneously describes the orbit of all vectors of the infinite wedge under $GL(\infty)$.  Roughly speaking, solutions $\tau_M$ to the 2-Toda hierarchy are parameterized by operators $M$ on the infinite wedge in the closure of $GL(\infty)$.  There are two sets of a variables; monomials in the first set of variables correspond to elements $v$ of the infinite wedge; in $\tau_M$ they appear multiplied by monomials that describes $Mv$, the action on $M$.

There is a little missing from this description; solutions to integrable hierarchies should be power series, while in the above description they are given by elements of the infinite wedge, or tensors of the infinite wedge with its dual in the case of the 2-Toda hierarchy.  This problem is taken care of by a construction originating from physics: the Boson-Fermion correspondence.  Another way to understand the infinite wedge is as a highest weight representation of a Fermionic Heisenberg algebra.  Wedging and contracting by a given basis vector $v$ give operators $\psi_v$ and $\psi^*_v$.  For distinct $v$ these anticommute, otherwise their anticommutator is 1.  Any Pl\"ucker coordinate in the infinite wedge may be obtained by applying some series of these operators to the vacuum vector.  Because of this, the infinite wedge is sometimes called fermionic fock space.

Similarly, we have that power series in an infinite set of variables is naturally a highest weight representation of a Bosonic Heisenberg algebra, given by the operators of multiplication by $x$ and the partial derivative $\frac{\partial}{\partial x}$.  For distinct variables these operations commute, while for the same variable their commutator is one.  Any monomial can be obtained from $1$ by applying some sequence of these operators.

It turns out that inside $\gl_{\infty}$ there are operators that behave just like the Bosonic operators.   The resulting representation of the Bosonic Heisenberg algebra on the infinite wedge turns out to be isomorphic to the representation of the Bosonic Heisenberg algebra on power series, and hence gives a nontrivial isomorphism between the infinite wedge and power series.  This isomorphism takes the basic Pl\"ucker coordinates in the infinite wedge to Schur functions, and is a useful way to encode the representation theory of the symmetric group, which is how we will connect to this formalism.

Thus, the Boson-Fermion correspondence gives us a way to associate power series to vectors in the infinite wedge.  The differential operators that define our integrable hierarchies are what result when we take the Pl\"ucker relations on the fermionic side and see what they become on the bosonic side.

\section{Ineffective Orbifolds and the decomposition conjecture}

The above discussion has described our results on the effective orbifolds $\mathcal{C}_{r,s}$; we also describe the Gromov-Witten theory of ineffective toric orbifolds.

If the generic point of an orbfold has isotropy group $K$, then $K$ will be the kernel of the action of each $G_x$.  We can quotient out by this copy of $K$ in each $G_x$ to obtain a new orbifold $\mathcal{X}_\text{rig}$, called the rigidification of $\mathcal{X}$, and we say that $\mathcal{X}$ is a $K$ gerbe over $\mathcal{X}_\text{rig}$.  An example of an ineffective orbifold curve is $\mbar_{1,1}$, the compactified moduli space of genus 1 curves with one marked point.  Since every genus 1 curve has an involution, the generic point of $\mbar_{1,1}$ has $\Z_2$ as its isotropy group.  $\mbar_{1,1}$ is a $\Z_2$ gerbe over $\mathcal{C}_{2,3}$.
 
These ineffective isotropy groups do not effect the topology of the orbifold, and should rather been seen as extra structure.  However, this extra structure has a nontrivial effect on maps into the orbifold: not every map into the rigidification lifts to a map into the gerbe.  

The handling of the ineffective case is first done in this paper, and is best understood in connection with a conjecture coming from physics.  The decomposition conjecture of \cite{HHPSA} suggests that a CFT (conformal field theory) arising from a $K$ gerbe $\mathcal{Y}$ over an effective orbifold $\mathcal{X}$ should, after an appropriate change of variables, decompose into CFTs on disjoint spaces $\mathcal{X}_i$.  There is an explicit construction of the $\mathcal{X}_i$ from $\mathcal{Y}$; for the gerbes considered in this paper, each $\mathcal{X}_i$ will be isomorphic to the underlying effective orbifold $\mathcal{X}=\mathcal{C}_{r,s}$, and the index $i$ runs over the set $K^*$ of irreducible representations of $K$.

 The simplest example of the decomposition conjecture is the Gromov-Witten theory of $\mathcal{B}G$, which can be viewed as a $G$ gerbe over a point.   Jarvis and Kimura \cite{JK} have shown that the Gromov-Witten theory of $\mathcal{B}G$ satisfies multiple commuting copies of the KdV hierarchy, one copy for each conjugacy class of $G$.  That is, after an appropriate change of variables the Gromov-Witten theory of $\mathcal{B}G$ is the disjoint union of the Gromov-Witten theory of a point, one for each conjugacy class of $G$.

Over a point, only the trivial gerbe is possible.  When the gerbe is nontrivial the decomposition conjecture is more complicated: the CFTs on the spaces $\mathcal{X}_i$ must be modified slightly by ``turning on discrete torsion''.   Physically, discrete torsion is essentially the orbifold version of a $B$-field.  Mathematically, it corresponds to twisting the Gromov-Witten theory by a flat $\C^*$ gerbe, as described in \cite{PRY}. Note that twisting GW theory by a flat $\C^*$ gerbe is quite different from taking the GW theory of the total space of a gerbe for a finite group.  For smooth $\mathcal{X}$, twisting by a flat $\C^*$ gerbe simply rescales the degree variable $q$.  In contrast, for orbifolds twisting the theory can change it drastically, but we show that in our case twisting is captured entirely in rescalings of the degree variable $q$ and the insertion variables.  Additionally, the conjecture allows for a physically meaningless rescaling of the genus variable $u$.

In the ineffective case, our operators act not on an infinite wedge, but on a related fock space that encapsulates the representation theory of the wreath products of $K$ with $S_n$.  This Fock space is essential a tensor product of $|K|$ copies of the infinite wedge, and our operators are well behaved with respect to this identification with a tensor product of infinite wedges.  As a result, this operator formalism leads to:
                                                                                                                                     
\begin{blanktheorem}
Let $\mathcal{Y}$ be a banded $K$ gerbe over $\mathcal{X}=\mathcal{C}_{r,s}$, and let $\tau_{\mathcal{Y}}, \tau_{\mathcal{X}}$ be the corresponding generating functions for equivariant Gromov-Witten theory.  Let $K^*$ be the set of irreducible representations of $K$.  Then after an explicit linear change of variables we have
$$\tau_{\mathcal{Y}}=\prod_{\gamma\in K^*} \tau_{\gamma}.$$
Here, $\tau_\gamma$ is the generating function for the equivariant Gromov-Witten theory of $\mathcal{X}$ twisted by the flat $\C^*$ gerbe prescribed by the decomposition conjecture.
\end{blanktheorem}
This is the first complete verification of the decomposition conjecture for nontrivial gerbes, although while this work was in progress a preprint appeared \cite{AJT1} announcing the general solution for all toric gerbes \cite{AJT2}.  Furthermore, put together with the previous theorem and the fact that the twisting merely amounts to a rescaling of variables, we get as an immediate corollary that $\tau_{\mathcal{Y}}$ satisfies $|K|$ commuting copies of the 2-Toda hierarchy.

In addition to the immediate consequences of the operator formalism we have already discussed, a large source motivation of this work lies in its future applications.  We hope to further extend Okounkov and Pandharipande's methods and use degeneration techniques to address more complicated orbifold curves.  The attraction of orbifold curves is that while their geometry is simple enough that the Gromov-Witten theory is readily approachable, there is enough structure that we can hope to get interesting answers.

A particular goal of this plan is a proof of the Virasoro conjecture for all effective orbifold curves.  Although integrable hierarchies known for a space are rather rare, the Virasoro conjecture gives a set of differential operators that should annihilate the generating function of any Kahler target space.  These operators do not commute, but rather form (half of) the Virasoro algebra.  The Virasoro conjecture has been proven for toric varieties by Givental, and flag and grassmannian varieties.  In these cases, there is some semisimplicity that produces the Virasoro operators.  In the third paper of Okounkov and Pandharipande's trilogy, the Virasoro conjecture is proven for all curves - this remains the only verification of the Virasoro conjecture in the non-semisimple setting.   Jiang and Tseng \cite{JT} have stated a version of the Virasoro conjecture for orbifolds; our current work should allow us to extend Okounkov and Pandharipande's proof to orbifold curves, which would be the first verification of the Virasoro conjecture for orbifolds, and would be particularly valuable  as a nonsemisimple example.

We note that some work toward the Gromov-Witten theory of more general orbifold curves has been done by Paolo Rossi \cite{Rossi1}, \cite{Rossi2} using techniques from symplectic field theory; one important observation is that a genus 0 curves with 3 orbifold points produce several more semisimple examples, where integrable hierarchies are expected and relatively understood, and his techniques seem quite powerful here.  In the nonsemisimple case his approach is quite concrete, but as of yet does not appear to give an approach to the Virasoro constraints.

\section{Background on Okounkov and Pandharipande's work}
 As we will largely be building on the work of Okounkov and Pandharipande, we now provide a broad sketch of the key points of their method, beginning with their motivation.  The starting point for their work was the Toda conjecture, first put forth by physicists \cite{EHY}, \cite{EY}, which suggests that the (non-equivariant) Gromow-Witten theory of $\proj^1$ is governed by the 2-Toda hierarchy.  Proving the Toda conjecture is one of the main achievements of the trilogy.  In earlier work \cite{PToda}, Pandharipande had shown that the Toda conjecture implies a certain Toda equation for Hurwitz numbers, which count covers of the sphere with prescribed ramification.  In particular, the double Hurwitz number $H_g(\mu,\nu)$ is the number of covers of $\proj^1$ by genus $g$ curves with arbitrary profiles $\mu, \nu$ over $0,\infty$, and simple ramifications at the appropriate number of other points.

With Pandharipande's work as motivation, Okounkov \cite{OHur} showed that in fact double Hurwitz numbers satisfy the entire 2-Toda hierarchy.  Combining the classical expression of Hurwitz numbers in terms of the representation theory of the symmetric group with the more recent description of the representation theory of the symmetric group in terms of the infinite wedge, Okounkov encoded Hurwitz numbers as operator expectations on the infinite wedge.   The work of the Kyoto school connects these operator expectations with integrable hierarchies, and so it quickly follows that double Hurwitz numbers satisfy the whole 2-Toda hierarchy.  Starting from a conjectural 2-Toda hierarchy for Gromov-Witten theory, an operator description of Hurwitz numbers was obtained.

The starting point of the trilogy is to work backwards from this development.  From the operator description of Hurwitz numbers an operator description of equivariant Gromov-Witten theory is derived, which again leads to a 2-Toda structure.  This is done in two steps.  First, virtual localization \cite{GP} allows integrals over the moduli space of stable maps to $\proj^1$ to be computed in terms of related integrals over the locus of $\C^*$ fixed maps to $\proj^1$.  This locus has components indexed by labeled graphs, and each component is essentially just a product of copies of $\overline{\mathcal{M}}_{g,n}$, and in particular is smooth.  The integrals over $\Mbar_{g,n}$ that result from the localization procedure involve terms coming from the normal bundle of the fixed locus, and are known Hodge integrals.

The second step is to evaluate these Hodge integrals.  This is done via the ELSV formula, which expresses Hodge integrals in terms of single Hurwitz numbers - a specialization of double Hurwitz numbers to the case where the cover is unramified over $\infty$.  An operator description for the equivariant Gromov-Witten invariants is obtained from Okounkov's operator description of Hurwitz numbers together with combinatorial factors coming from localization and the ELSV formula.  There are technical issues that must be resolved, stemming largely from the fact that Hurwitz numbers only make sense for integer values (the order of ramification), and so the operators must be interpolated to complex values.  This results in more complicated operators than were present in the Hurwitz case, but it can be shown that they are conjugate to the standard operators used in the 2-Toda hierarchy, and so the equivariant Gromov-Witten theory of $\proj^1$ also satisfies the 2-Toda hierarchy.

 \section{Overview of the present work}

We will use the general method of Okounkov and Pandharipande outlined above to construct our operator expression, although several new ideas are needed, particularly to deal with ineffective orbifolds.  First, the orbifold structure introduces some new features to the localization process.  Again, the fixed point loci are smooth orbifolds indexed by labeled graphs.  Each component is essentially a product of $\overline{\mathcal{M}}_{g,\rrr}(\mathcal{B}R)$ and $\overline{\mathcal{M}}_{g,\sss}(\mathcal{B}S)$, where $R$ and $S$ are the isotropy groups of $\mathcal{X}$ over $0$ and $\infty$, respectively.  Here $\rrr$ (respectively $\sss$) is an $n$-tuple of elements of $R$ (or $S$) that record the orbifold structure at the marked points, and are an essential new feature.  In the effective case, $\rrr$ is determined locally by the degree of the map, and so localization is not much more complicated.  In the case of a gerbe there is a global relation between $\rrr$, $\sss$, and the degree of the map which must be understood to carry out localization.   Our first important result is Lemma \ref{gerbelemma}, which explicitly describes this relationship.  This is perhaps best understood in analogy with the monodromy of a principal bundle.  A principal bundle is determined by its monodromy (or, for a lie group with connection, its holonomy) around closed loops; for gerbes this monodromy is around closed surfaces.   Lemma \ref{gerbelemma} computes the monodromy of our gerbes in a way that can be easily applied to localization.

The normal bundle over the fixed locus produces integrals known as Hurwitz-Hodge integrals, and to complete the second step, an analog of the ELSV formula for them is needed.  Such a formula was recently obtained in joint work with Pandharipande and Tseng \cite{JPT}.  For the effective orbifolds $\mathcal{C}_{r,s}$, the resulting Hurwitz-Hodge integrals are expressed in terms of certain double Hurwitz numbers.  These are still covered by Okounkov's operator expression, and so in this case the same procedure leads to an expression of the equivariant Gromov-Witten invariants in terms of operator expectations on the infinite wedge.  The operators have the same general form, and so although the technical arguments must be adapted,  Okounkov and Pandharipande's work can be followed quite closely.

In the presence of a $K$ gerbe the relevant Hurwitz-Hodge integrals are encoded not in double Hurwitz numbers but in $K$ wreath double Hurwitz numbers, which count ramified covers of $\proj^1$ with monodromy lying in a wreath product of $K$ and the symmetric group.  Constructions parallel to those used for usual Hurwitz numbers exist for these $K$ wreath double Hurwitz numbers.  In particular, they can be expressed in terms of the representation theory of the wreath product, and there is a Fock space approach to this representation theory - that is, an analog of the infinite wedge.  This Fock space, developed and applied by Wang and collaborators (e.g. \cite{FW}), is a tensor product of $|K|$ copies of the infinite wedge.  As a result, for operators of the correct form, vacuum expectations decompose into a product of operator expectations on the infinite wedge.  This is a reflection of the decomposition of the representation theory of $K$ wreath products into the representation theory of $K$ and multiple copies of the representation theory of the symmetric group.

This decomposition has been utilized in Qin and Wang's work on the equivariant cohomology of the Hilbert scheme of points on the $A_n$ resolution \cite{QW}.  For the case of points in the plane, this is governed by a 2-Toda hierarchy; Qin and Wang's use the wreath product Fock space to show that in the case of points on the $A_n$ resolution, there are $n+1$ commuting copies of the 2-Toda hierarchy.  In light of the decomposition conjecture, this result suggests applying their machinery to the Gromov-Witten theory of gerbes, and was motivation for the formulation of the orbifold ELSV formula in terms of wreath Hurwitz numbers.

As a warmup to our proof of the decomposition conjecture we give another application of this wreath Fock space, by extending Okounkov's result on double Hurwitz to wreath Hurwitz numbers.
\begin{blanktheorem}
 Let $\tau$ be the generating function of double Hurwitz numbers, and $\tau_K$ be the generating function of $K$-wreath Hurwitz numbers.  Then, after an explicit linear change of variables, we have
$$\tau_K=\prod_{\gamma\in K^*}\tau_\gamma$$
where the product is over $K^*$, the set of irreducible representations of $K$, and $\tau_\gamma$ is a rescaling of $\tau$.
\end{blanktheorem}

\section{Detailed Outline}

Our first task, in chapter 2 is to recall the basics of orbifolds and orbifold Gromov-Witten theory.  We put an emphasis on gerbes, which are treated rather abstractly in the literature, but are relatively concrete in our situation.  In particular, we classify and give explicit geometric constructions of all banded $K$ gerbes over $\mathcal{C}_{r,s}$ for $K$ finite abelian.  The key new result is lemma \ref{gerbelemma}, which describes how the gerbe effects the Gromov-Witten theory.  We also briefly describe the twistings of Gromov-Witten theory that appear in the decomposition conjecture.

Chapter 3 carries out the localization part of the argument.  Using disconnected generating functions with unstable contributions greatly simplifies the final form by allowing for a more uniform treatment of localization,  and by allowing us to deal with a sum over partitions rather than a sum over trees.  The unwanted unstable terms can be removed easily later.  In the noneffective case, localization produces a sum over partitions with parts labeled by elements of $K$, which correspond to conjugacy classes in the wreath product.  The end result is an expression for the equivariant Gromow-Witten generating function of $\mathcal{X}$ in terms of the Hurwitz-Hodge generating functions and combinatorial factors.

Chapter 4 provides the background needed on the wreath product and the Fock space formalism for its representation theory.  We explain how this Fock space is essentially the tensor product of multiple copies of the infinite wedge.  Wreath double-Hurwitz numbers are introduced, reduced to the representation theory of wreath products, and expressed as vacuum expectations of operators on this Fock space.   After a quick review of the 2-Toda hierarchy, we show that the generating function for wreath Hurwtiz numbers satisfies multiple commuting copies of it.

Chapter 5 combines the previous sections with the orbifold ELSV formula.   We begin by recalling the orbifold ELSV formula of \cite{JPT}, which evaluates the generating function of Hurwitz-Hodge integrals at appropriate integer values in terms of the wreath Hurwitz numbers.  Using the machinery of section \ref{Wreath Products and Fock Spaces}, this provides operator expectation formulas for the Hurwitz-Hodge integrals at these values.  Interpolating from these integer values to an open subset of $\C^n$ is the technical heart of the paper.  The resulting operators are rather complicated, and some technical proofs are postponed until the last chapter.  Having produced an operator expression for Hurwitz-Hodge integrals, we combine it with the localization analysis of section \ref{Localization} to obtain our main result: an operator expression for equivariant Gromov-Witten theory.

Chapter 6 applies the operator formula to derive the two main theorems mentioned above.   First we prove the decomposition conjecture, exhibiting a change of variables that produces the desired decomposition at the level of operators. We careful interpret the rescaling of variables as twisting by flat $\C^*$ gerbes.  Then we address the 2-Toda hierarchy, first deriving an explicit form of the lowest 2-Toda equation by hand, and then presenting a change of variables between the Gromov-Witten variables and the standard 2-Toda times to establish the entire hierarchy.

Finally, chapter 7 contains the proofs of the two main technical lemmas.  First, we show that the operators we define actually converge on a certain region.  Then, we determine their commutators. 
\chapter{Orbifolds and their Gromov-Witten theory}
\label{orbifoldbackground}
\section{Orbifold Background} \label{orbifoldbackground} \label{Orbifold Background}
This section is an idiosyncratic introduction to orbifolds and their Gromov-Witten theory.  The literature on stacks is notoriously abstract and general.   To counteract this, we strive to be intuitive, concrete and specific, and focus on the simplest examples we will need.  We place particular emphasis on gerbes, as this material is less standard and none of the existing presentations are particularly concrete.  Our treatment is by no means complete - for instance, there is no real discussion of stacks as groupoids.  For a more thorough introduction to orbifolds, we suggest \cite{ALR} as a good starting place.

Section \ref{orbandgerbe} introduces orbifolds, focusing on the notion of ineffective orbifold or gerbe.  Section \ref{orbcohomology} reviews orbifold cohomology and Chen-Ruan cohomology in preparation for the cohomological classification of gerbes in \ref{classification}.  Then \ref{linebundles} reviews orbifold line bundles and their Chern classes in preparation for the explicit construction of our class of gerbes in terms of the $r$th root construction in \ref{root}.

\subsection{Orbifolds and gerbes}
\label{orbandgerbe}

An orbifold $\mathcal{X}$ consists first of an underlying topological space, denoted $|\mathcal{X}|$ and called the coarse moduli space.  Each point $x\in|\mathcal{X}|$ has a neighborhood $U$ that is isomorphic to a quotient $\C^n/G_x$, for some finite group $G_x$, called the isotropy group of $x$.

\begin{example}
If $M$ is a manifold, and $G$ is a finite group acting on $M$, then the quotient $M/G$ admits the structure of an orbifold, which we denote $[M/G]$.  For instance, $\Z_n$ acts on $\proj^1$ by multiplying by roots of unity.  The action is free away from 0 to $\infty$, and the quotient $\proj^1/\Z_n$ is again topologically $\proj^1$, but the orbifold $[\proj^1/\Z_n]$ remembers that the action was not free at $0$ and $\infty$ - these points each have isotropy group $\Z_n$.  Orbifolds of the type $[M/G]$ are called \emph{global quotients}
\end{example}

\begin{example}
Not all orbifolds are global quotients.  A pertinent example is weighted projective space.  As in the construction of normal projective space, let $\C^*$ act on $\C^n-\{0\}$ .  However, rather than the standard action, take the action given by $t\cdot (z_1,\dots, z_n)=(t^{a_i}z_1,\dots,t^{a_n}z_n)$, for some $a_i\in\Z_{\geq 1}$.  We denote the resulting quotient $\C^n-\{0\}/\C^*$ as $\proj(a_1,\dots,a_n)$.  One can check that each point has only a finite stabilizer group, so that each point in $\proj(a_1,\dots, a_n)$ is locally a quotient by a finite group.  For instance $\proj(2,3)$ will topologically be a two sphere, but the point corresponding to the orbit of $(1,0)$ will have an isotropy group of $\Z_2$, while the point corresponding to the orbit of $(0,1)$ will have isotropy group $\Z_3$.  All other points will have trivial isotropy group, as all other $\C^*$ orbits are free.  On the other hand, we can see that $\proj(2,2)$ will topologically just be a standard $\proj^1$, but every point will have a $\Z_2$ as their isotropy group, as $-1\in\C^*$ will act trivially.
\end{example}

The last example, $\proj(2,2)$, is one of the simplest examples of an {\em ineffective orbifold}.   Frequently when studying orbifolds, the action of $G_x$ was required to be effective - that is, the only element that fixed everything was the identity.  Then we can describe the orbifold in terms of local charts - each point $x\in|\mathcal{X}|$ has a neighborhood $U_x$ and a chosen isomorphism between $U_x$ and $\widetilde{U}_x/G_x$, with $\widetilde{U}_x$ isomorphic to $\R^n$, and we can describe how to glue these charts on the overlaps.  In the case of ineffective orbifolds, this gluing picture is inadequate, and these ineffective orbifolds are now most conveniently described in terms of a morita equivalence classes of proper lie groupoids.

Briefly, this means we can represent an orbifold $\mathcal{X}$ as a small (objects form set) category $\mathfrak{g}$.  We use $\mathfrak{g}_0$ and $\mathfrak{g}_1$ to denote the sets of objects and morphisms, respectively.  Being a groupoid means that every morphism in this category is an isomorphism; being a lie groupoid means that $\mathfrak{g}_0$ and $\mathfrak{g}_1$ both smooth manifolds and all structure morphisms are smooth.  The coarse moduli space $|\mathcal{X}|$ is obtained as the quotient of the $\mathfrak{g}_0$ by the equivalence relation that identifies isomorphic objects - that is, $x\sim y$ if there is $g\in\mathfrak{g}_1$ with $s(g)=x, t(g)=y$.   The representation of an orbifold as a groupoid is not unique.  This should not be surprising, as is general in category theory, the notion of isomorphic categories is too strong, and normally the right notion is an  equivalence categories.   This is basically the case here, except some care is needed because of the topologies on our categories; the notion of Morita equivalence captures this.  We refer the to \cite{ALR} for precise details and a more leisurely exposition of this material.

A caveat: some algebraic geometers reserve the word ``orbifold'' for the effective case, while the general case is merely a Deligne-Mumford stack; for us an orbifold may be ineffective.

\begin{example} \label{bg}
We have already seen one example of an ineffective orbifold: $\proj(2,2)$.  A more fundamental example are the spaces $\mathcal{B}G=[\mathrm{pt}/G]$, the quotient of a point by a finite group $G$ acting trivially.  These are called classifying stacks, as an orbifold map from $\mathcal{X}$ to $\mathcal{B}G$ is equivalent to a principal $G$-bundle on $\mathcal{X}$.  It is instructive to compare this to the topological classifying spaces, $BG$.  These gain something in that they are honest spaces rather than orbifolds, but lose something in that they only classify homotopy classes of maps, rather than just maps.
\end{example}

For each point $x$ the elements of $G_x$ that act ineffectively will form a normal subgroup $K_x\triangleleft G_x$.  The isomorphism type of the group $K_x$ will be locally constant, and we will denote this common group by $K$.  We can {\em rigidify} our orbifold by taking the quotient of each isotropy group $G_x$ by $K$ to form the effective orbifold $\mathcal{X}^{\mathrm{rig}}$.  Then our original orbifold $\mathcal{X}$ will be a bundle over $\mathcal{X}^{\mathrm{rig}}$ with fiber $\mathcal{B}K$:
\begin{equation*}
\begin{CD}
\mathcal{B}K@>>>\mathcal{X} \\
@. @VVV \\
@. \mathcal{X}^{\mathrm{rig}}
\end{CD}
 \end{equation*}
In this case, we say that $\mathcal{X}$ is a $K$-gerbe over $\mathcal{X}^{\mathrm{rig}}$. Note that if $\mathcal{Y}$ is a $K$-gerbe over $\mathcal{X}$, then $|\mathcal{Y}|$ is naturally isomorphic to $|\mathcal{X}|$, and so allowing ineffective group actions does not change the singularities that can be achieved in the underlying topological space.  Instead, gerbes should be thought of as an extra structure included.

\begin{example}
If $\mathcal{X}$ is an orbifold, the space $\mathcal{B}K\times\mathcal{X}$ is a $K$-gerbe over $\mathcal{X}$, called the trivial gerbe.
\end{example}

\begin{example} In the discussion above, we have seen how gerbes are like fiber bundles.  Another important analogy to have in mind is group extensions.  If
\begin{equation*}
0\to K\to H\to G\to 0
\end{equation*}
is a short exact sequence of groups, then $\mathcal{B}H$ will be a $K$-gerbe over $\mathcal{B}G$.  In general, if $\mathcal{Y}$ is a $G$-gerbe over $\mathcal{X}$, and $x\in\mathcal{X}$, we can also think of $x$ as a point in $\mathcal{Y}$.  If $G_x$ is the isotropy group of $x\in\mathcal{X}$, and $H_x$ is the isotropy group of $x\in\mathcal{Y}$, then $H_x$ will be a $K$ extension of $G_x$.
\end{example}

\subsection{Orbifold Cohomology} \label{orbcohomology}

We will make use of two types of orbifold cohomology.  When we write $H^*(\mathcal{X})$ we will mean the ``classical'' orbifold cohomology, i.e. $H^*(\mathcal{X})=H^*(B\mathcal{X})$  the topological cohomology of the classifying space of a groupoid representing $\mathcal{X}$.  This is a natural extension of the classifying space of a group: indeed, $B\mathcal{B}G=BG$, so that $H^*(\mathcal{B}G)$ is just the usual group cohomology of $G$.  We will be able to calculate the orbifold cohomologies we are interested in without using a detailed understanding of the groupoid presentation by using the fact that $B\mathcal{X}$ has a map to $|\mathcal{X}|$, with fibers $BG_x$ over $x$.  In any case, the homotopy type of the classifying space turns out not depend on the choice of groupoid used to represent $\mathcal{X}$.  Furthermore, although $H^*(\mathcal{X},\Z)$ can contain interesting torsion, it is a fact that $H^*(\mathcal{X},\R)\cong H^*(|\mathcal{X}|,\R)$.

In adition to $H^*(\mathcal{X})$, we will also make use of the Chen-Ruan cohomology, $H^*_{CR}(\mathcal{X})$.  As a vector space, Chen-Ruan cohomology is just the usual cohomology, with coefficients in $\C$, of the inertia stack of $\mathcal{X}$, denoted $\mathcal{IX}$.  An important and relatively concrete way of viewing $\mathcal{IX}$  is as the space of constant maps from $S^1$ to $\mathcal{X}$.  There is a map $\mathcal{IX}\to\mathcal{X}$ given by evaluating the map, and the fiber over $x\in |\mathcal{X}|$ is the space of maps from $S^1$ to $\mathcal{B}G_x$, which, as discussed in Example \ref{bg}, is the space of principal $G$ bundles on $S^1$.  Considering the monodromy of this bundle around a generator of $\pi_1(S^1)$, we see that these are classified up to automorphism by conjugacy classes of elements in $G_x$.  Hence, elements of the inertia stack are often denoted $(x,g)$, with $x\in\mathcal{X}$ and $g\in G_x$.  In general, $\mathcal{IX}$ will have many components.  The elements $(x,e)$ with $e$ the identity of $G_x$ together form a component isomorphic to $\mathcal{X}$.  The other components are often called the twisted sectors, using language originating from string theory.

Although we have $H^*_{CR}(\mathcal{X})=H^*(\mathcal{IX},\C)$ as vector spaces, the cup product and grading are different.  We will not need the cup product, but it is obtained by a push-pull construction and the Euler class of a natural obstruction bundle living on the higher twisted sectors of $\mathcal{X}$.  We will, however, make use of the shift in grading.   Pulling back the tangent bundle of $T\mathcal{X}$ to $\mathcal{IX}$, we get for each point $(x,g)\in\mathcal{IX}$ a representation of $\langle g \rangle$, the group generated by $g$.  As this group is cyclic, the representation diagonalizes, so that $g$ acts by $\textrm{diag}(e^{2\pi i r_1},\dots,e^{2\pi i r_n}), n=\dim(\mathcal{X}), 0\leq r_i<1 \in\Q$.  The degree shifting number $\iota(g)$, or age, of $(x,g)$ is defined to be the sum of the $r_i$.   It is easy to see that the degree shifting number is constant on each component of the inertia stack, and so for $T$ a twisted sector we will often write $\iota(T)$.  For $\alpha\in H^k(T)$, with $T$ a twisted sector, the degree $\alpha$ as an element of  $H^*_{CR}(\mathcal{X})$ is defined to be $k+2\iota(T)$,   In general, this is only a rational number.

We will be interested in one dimensional orbifolds $\mathcal{X}$ that admit a $\C^*$ action.  This immediately forces $|\mathcal{X}|=\proj^1$, and furthermore as any orbifolds points of $\mathcal{X}^{\mathrm{rig}}$ must be fixed by the $\C^*$ action, there can be at most two of them which we will take to be $0$ and $\infty$.  We will denote by $\mathcal{C}_{r,s}$ the effective orbifold with $|\mathcal{C}_{r,s}|=\proj^1$ and isotropy groups $\Z_r$ at $0$ and $\Z_s$ at $\infty$.  If $r=s$, then $\mathcal{C}_{r,s}=[\proj^1/\Z_r]$, while if $r$ and $s$ are relatively prime, then $\mathcal{C}_{r,s}=\proj(r,s)$, the weighted projective space.

\begin{example} We calculate $H^*_{CR}(\mathcal{C}_{r,s})$.  First, we need to understand the inertia stack $\mathcal{IC}_{r,s}$.  The nontwisted sector will just be $\mathcal{C}_{r,s}$, and since this is topologically just $\proj^1$, its cohomology $H^*(\mathcal{C}_{r,s},\C)$ will be two dimensional, with a generator $1\in H^0$ and a generator $\omega\in H^2$.  There is no normal bundle, and so there is no degree shifting.

There are $r-1$ twisted sectors lying over $0$, corresponding to the $r-1$ nonidentity elements of the isotropy group at $0$.  We will denote these twisted sectors as $\mathcal{IC}_{r,s}(x), x=1/r,\dots, (r-1)/r$, each of which is just a copy of $\BZ_r$, and hence when we take cohomology with $\C$ coefficients, we get a one dimensional space in degree 0.  We denote the multiplicative identity of $H^*(\mathcal{IC}_{r,s}(k/r),\C)$ by $\zero_k$.  Since the element $k/r$ acts on the tangent bundle of $\mathcal{C}_{r,s}$ by $e^{2\pi k/r}$, the degree shifting number is $k/r$, and so $\zero_k\in H^{2k/r}_{CR}(\mathcal{C}_{r,s})$.  Similarly, we have elements $\finity_k\in H^{2k/s}_{CR}(\mathcal{C}_{r,s}),1\leq k\leq s-1$ coming from the twisted sectors over infinity.
\end{example}

\subsection{Cohomological classification of Gerbes} \label{classification}

We will work with a particular class of gerbes over the effective orbifolds $\mathcal{C}_{r,s}$.  To understand which ones it will be useful to have a general discussion of gerbes.   Gerbes were introduced and given a cohomological classification by Giraud in \cite{Giraud}.  This classification is conveniently available in \cite{Breen1} and \cite{Breen2}.  The general answer uses nonabelian cohomology, which appears intimidating, but is actually exactly what arises when we try to glue together a gerbe on $X$ from locally trivial ones and use \v Cech theory.

We will only consider the case when $K$ is abelian, in which case the relevant nonabelian cohomology groups have a simplified description.  The first part is a principal $\aut(K)$ bundle, which are classified by $\alpha\in H^1(\mathcal{X},\aut(K))$, even if $\aut(K)$ is not abelian.  From $\alpha$, one can construct the associated $K$ bundle $\mathcal{K}$.  Note that $\mathcal{K}$ is not a principal $K$ bundle - the fibers are not $K$-torsors, but copies of $K$ as a group.  Thus, $\mathcal{K}$ can be viewed as a system of local coefficients, or as a locally constant sheaf.  Either way, we can consider elements $\beta\in H^2(\mathcal{X},\mathcal{K})$.  Gerbes are essentially classified up to isomorphism by pairs $(\alpha, \beta)$.  More precisely, such pairs classifies $\mathcal{K}$-gerbes, which is a gerbe with slightly more structure.  In the case that $\mathcal{K}$ is the trivial $K$ sheaf, then we say the gerbe is banded; banded gerbes on $\mathcal{X}$ are classified by the usual $H^2(\mathcal{X}, K)$.

\begin{example}
We have already noted that an extension of groups $$0\to K\to H \to G\to 0$$ can be viewed as a $K$-gerbe over $\mathcal{B}G$.  In the case that $K$ is abelian, the cohomological description of the gerbe coincides with the well known cohomological description of group extensions as $H^2(G, K)$.  The 1-cocycle $\alpha \in H^1(G, \aut(K))$ describes how pullbacks of elements of $G$ to $H$ act on $K$ by conjugation.  If the gerbe is banded, then the element $H^2(G, K)$ is just the usual description of a central extension in terms of 2-cocycles.
\end{example}

\begin{example}
Consider $\Z_3$ gerbes over $\proj^1$.  Since $H^1(\proj^1, \aut(\Z_3))=0$, we see that $\Z_3$ gerbes on $\proj^1$ are classified by $H^2(\proj^1,\Z_3)=\Z_3$.  The element $0$ corresponds to the trivial gerbe $\proj^1\times \BZ_3$, and the element $1\in H^2(\proj^1,\Z_3$ corresponds to the weighted projective space $\proj(3,3)$.  The third element is essentially $\proj(3,3)$ again, but relabeled by the nontrivial automorphism of $\Z_3$.  Topologically, the two nontrivial gerbes are isomorphic; but as banded gerbes they have a labeling of the elements of the isotropy groups that distinguish them.
\end{example}

\begin{example}
Consider the global quotient $[\proj^1/S_3]$, where $\sigma\in S_3$ acts on $\proj^1$ by $\sigma\cdot z\mapsto \mathrm{sgn}(\sigma)z$, where $\mathrm{sgn}(\sigma)\in\pm 1$ is the sign of the permutation.  Then $\Z_3\cong A_3\subset S_3$ acts trivially, and so we see that $[\proj^1/S_3]$ is a $\Z_3$ gerbe over $[\proj^1/\Z_2]$.  Although topologically $[\proj^1/\Z_2]$ is simply connected, as an orbifold it has fundamental group $\Z_2$, with $\proj^1$ is its universal cover, and so $H^1([\proj^1/\Z_2] ,\aut(\Z_3))=\Z_2$.  The gerbe $[\proj^1/S_3]$ cannot be banded by the trivial $\Z_2$ gerbe.  One way to see this is that the isotropy group of $0$ is $S_3$, a nonabelian group.  If $[\proj^1/S_3]$ were banded, then restricting to $0$ would give a class in $H^2(\Z_2,\Z_3)$, which classifies abelian extensions of $\Z_2$ by $\Z_3$.
\end{example}

For the remainder of this paper,we will focus our attention on $K$-gerbes banded by the trivial bundle, with $K$ an abelian group. Another way of understanding what being banded means is to examine the map from $|\mathcal{IX}|\to|\mathcal{X}|$.  If we restrict this to points of the form $(x,k)$ with $g\in K$ then, since $K$ is abelian and conjugacy classes are just elements, we expect the map to be a $|K|$-fold cover.  Asking $\mathcal{X}$ to be banded is equivalent to asking this to be the trivial cover.

In our simple case of gerbes over $\mathcal{C}_{r,s}$, requiring our gerbes to banded is equivalent to requiring that all of our isotropy groups are abelian, essentially because the central extensions of cyclic groups are precisely the abelian extensions.

Since we want to understand banded abelian gerbes over $\mathcal{C}_{r,s}$, we will need to understand the second cohomology group $H^2(\mathcal{C}_{r,s}, \Z_n)$.  We have:

\begin{lemma} \label{cohomologylemma} Let $q=\gcd(r,s,n)$.  Then
$$
H^2(\mathcal{C}_{r,s},\Z_n)=H^2(\mathcal{C}_{r,s},\Z)/nH^2(\mathcal{C}_{r,s},\Z)=\Z_n\oplus \Z_q.
$$
\end{lemma}

\begin{proof}
 This is a simple Mayer-Vietoris calculation.  Consider the map from $f:B\mathcal{C}_{r,s}\to|\mathcal{C}_{r,s}|=\proj^1$. Define $U=f^{-1}(\proj^1\setminus\{\infty\})$ and $V=f^{-1}(\proj^1\setminus\{0\})$.  $U$ and $V$ have deformation retracts onto $f^{-1}(0), f^{-1}(\infty)$, respectively, and so their cohomology will be isomorphic to that of $B\Z_r$ and $B\Z_s$.   The intersection $U\cap V$ will be homotopy equivalent to $S^1$, and so we Mayer-Vietoris gives
$$
\cdots \to H^i(\mathcal{C}_{r,s},\Z)\to H^i(B\Z_r,\Z)\oplus H^i(B\Z_s,\Z)\to H^i(S^1, \Z)\to \cdots
$$

Since $H^i(B\Z_r,\Z)$ is 0 for $i$ even and $\Z_r$ for $i$ odd, we see that $H^3(\mathcal{C}_{r,s},\Z)=0$, while $H^2(\mathcal{C}_{r,s},\Z)$ is isomorphic to $\Z\oplus\Z_p$, with $p=\gcd(r,s)$.  Together with the long exact sequence (\ref{CohLES}), this gives the desired result.
\end{proof}

The isotropy group of a generic point of a banded $K$-gerbe over $\C_{r,s}$ will come with an isomorphism to $K$, but the isotropy groups over $0$ and $\infty$, which we will denote $R$ and $S$, will be potentially nontrivial extensions of $\Z_r$ and $\Z_s$ by $K$.  In this case we will have a short exact sequence $0\to K\to R\to \Z_r\to 0$, and we can determine which extension we have by pulling back the cohomology class of the gerbe in $H^2(\mathcal{C}_{r,s},K)$ to $H^2(\BZ_r,K)$ via the inclusion $\mathcal{B}\Z_r\to\C_{r,s}$ and noticing that $H^2(\BZ_r,K)$ classifies central extensions of $\Z_r$ by $K$.  In the next section, following a concrete construction of $K$-gerbes, we will give an alternate description of these isotropy groups.

\subsection{Orbifold Line Bundles and Chern Classes}
\label{linebundles}

Orbifold vector bundles are given locally by a line bundle on the orbifold chart, together with a lifting of the group action to this bundle.  In particular, the fiber of an orbifold vector bundle over a point $x$ is a representation of $G_x$.  A vector bundles on $\mathcal{B}G$ is precisely a representation of $G$.

As with topological spaces, orbifold line bundles $L$ on $\mathcal{X}$ are classified by the first Chern class $c_1(L)\in H^2(\mathcal{X},\Z)$.  For example, in the case of $\mathcal{B}G$, a line bundle is a one dimensional representation, or equivalently a homomorphism from $G$ to $\C^*$.  Indeed, these homomorphisms are classified by $H^1(G, \C^*)\cong H^2(G, \Z)$.  In general, from an orbifold line bundle one can construct an honest line bundle on the classifying space, and we can take the chern class of that bundle.

In the case that $\mathcal{X}$ is effective, we can also construct, via Chern-Weil theory, a class $cw_1(L)\in H^2(\mathcal{X},\Q)$ \cite{CR}.  In general, $cw_1(L)$ will be be a rational class, not integral, and will not be enough to determine $L$.  However, if we also remember the representations of the isotropy groups that $L$ provides, it is enough in the case when $\mathcal{X}$ is an effective orbifold curve.

Let $\mathcal{X}$ be a smooth, effective, orbifold curve with orbifold points $x_1,\dots, x_m$ of order $r_1,\dots, r_m$, with $L$ an orbifold line bundle on $\mathcal{X}$.  Each isotropy group has a distinguished generator $1_j\in G_{x_j}=\Z_{r_j}$, namely the element that acts on the tangent space by $e^{2\pi i/r_j}$.  This element must act on $L_{x_j}$ by some $e^{2\pi i k_j/r_j}, 0\leq k_k<r_j$.  Then the Chern-Weil class $cw_1(L)$ and the numbers $k_j/r_j$ determine $L$.  Furthermore, $|\mathcal{X}|$ is a smooth curve, and from $L$ one can construct a line bundle $|L|$ on $|\mathcal{X}|$, called the \emph{desingularization} of $L$, on the underlying smooth curve.  Since we will use the desginularization to calculate various cohomology groups, we review it now.

Around a point $x\in\mathcal{X}$ with orbifold structure $\Z_r$, an orbifold chart of the total space of $L$ will have coordinates $(z,w)$, with $z$ the curve direction and $w$ the $L$ direction. Then the preferred generator $1\in \Z_r$ will act by $1\cdot (z,w)=(e^{2\pi i/r}z, e^{a 2\pi i /r} w)$, with $a\in \{0,\dots, r-1\}$ determined by the representation of $\Z_r$ and $L_x$.  Then the map $d:(\C^2\to\C^2)$ given by:
$$
d:(z,w)\mapsto (z^r, z^{-a}w)
$$
from $\C^2\to \C^2$ will be $\Z_r$ equivariant if we give the target $\Z_r$ the trivial $\Z_r$ action.  These later coordinates will form a local coordinate chart for the total space of the desingularized line bundle $|L|$.  The main point of the desingularization is that there's an isomorphism $H^i(\mathcal{X}, L)=H^i(|\mathcal{X}|,|L|)$, and $c_1(|L|)+\sum a_i/r_i=cw_1(L)$.  Since $c_1(|L|)$ is integral, the fractional part of $cw_1(L)$ is determined by the representations of the isotropy groups.  In addition, together with the data for the local chart, this allows us to calculate the $\C^*$ weights on cohomology by working with the desingularization instead.

\subsection{Gerbes from the root construction}
\label{root}

We now show how the root construction can be used to construct any banded $K$ gerbe on $\mathcal{C}_{r,s}$ from line bundles.  Our gerbes are precisely the 1-dimensional toric stacks.  Toric stacks were introduce in \cite{BCS} and investigated further in \cite{FMN}, where it is shown that all toric stacks can be constructed using the root construction.  The root construction of these gerbes will provide us with an explicit description of isotropy groups $R$ (and $S$) as extensions of $\Z_r$ (or $\Z_s$) by $K$, as well as help us understand the effect the gerbe structure has on maps into $\mathcal{X}$.  As $K$ is finite and abelian, we fix a (not unique) isomorphism $K=\oplus \Z_{n_i}$.  We will construct any $K$ gerbe as a fibered product of $\Z_{n_i}$ gerbes.

We begin with a cohomological description of the root construction.  Consider the short exact sequence:
\begin{equation*}
0\to \Z\stackrel{\cdot n}{\to}\Z \to \Z_n\to 0
\end{equation*}
and part of the corresponding long exact sequence in cohomology:
\begin{equation} \label{CohLES}
H^2(\mathcal{C}_{r,s},\Z)\stackrel{\cdot n}{\to} H^2(\mathcal{C}_{r,s},\Z)\stackrel{g}{\to} H^2(\mathcal{C}_{r,s},\Z_n)\to H^3(\mathcal{C}_{r,s},\Z).
\end{equation}
Since $H^2(\mathcal{C}_{r,s},\Z)$ classifies line bundles over $\mathcal{C}_{r,s}$ and $H^2(\mathcal{C}_{r,s},\Z_n)$ classifies $\Z_n$ gerbes, the map $g$ should give a way to construct a banded $\Z_n$ gerbe out of a line bundle - this is the root construction.  Furthermore, we saw in the proof of Lemma \ref{cohomologylemma} that $H^3(\mathcal{C}_{r,s},\Z)=0$, and so every banded $\Z_n$ gerbe can be constructed this way.  We denote the gerbe constructed in this manner from a line bundle $L$ by $\mathcal{C}_{r,s}^{(L,n)}$.  Note that this map is not injective: each gerbe can be constructed from infinitely many different line bundles.   We will use $f:\mathcal{C}_{r,s}^{(L,n)}\to\mathcal{C}_{r,s}$ to denote the natural rigidification map that forgets the gerbe.

The above construction for cyclic groups generalizes to any abelian group by a fibered product construction.  Explicitly, with $K=\oplus_{j=1}^m \Z_{n_j}$, let $L_1,\dots, L_m$ be line bundles on $\mathcal{C}_{r,s}$.   Then we form a $K$-gerbe by taking the fibered product of the root constructions for each of these line bundles.  We will use the following notation for this:
$$
\mathcal{C}_{r,s}^{(L_j,n_j)}=\mathcal{C}_{r,s}^{(L_1,n_1)}\times_{\mathcal{C}_{r,s}}\cdots\times_{\mathcal{C}_{r,s}}\mathcal{C}_{r,s}^{(L_m,n_m)}
$$

Geometrically, the construction of a $\Z_n$ gerbe on $\mathcal{C}_{r,s}$ from a line bundle $L$ is easily described: the total space of the line bundle has a natural $C^*$ action, which is just the usual $\C^*$ action on each fiber.   If we remove the zero section, and then quotient by $C^*$ acting by the $n$th power of the standard action, we get $\mathcal{C}_{r,s}^{(L,n)}$.

From the cohomological and geometrical descriptions it is difficult to understand the more subtle properties of the gerbe, such as the space of maps into the gerbe, and we will find it convenient to understand the gerbe in terms of the categorical property that provides its name: on $\mathcal{C}_{r,s}^{(L,n)}, f^*(L)$ has a natural $n$th root.  More carefully, on $\mathcal{C}_{r,s}^{(L,n)}$ there is an orbifold line bundle $M$, and an isomorphism $\phi: M^{\otimes n}\to f^*L$.  We now illustrate the use of this property by using it give an explicit description of the isotropy groups $R$ and $S$ of $\mathcal{C}_{r,s}^{(L_j,n_j)}$, and later we will use it understand the effect of the gerbe on maps into the target space.

First, note that $1_r\in\Z_r$ acts on $L_j|_0$ by some $e^{2\pi i k_j/r}$, and similarly $1_s\in\Z_s$ acts on $L_j|_\infty$ by some $e^{2\pi i l_j/s}$. Let $M_j$ be the $n_j$th root of $L_j$.  The isotropy group of each point of $\mathcal{X}$ contains a natural copy of $\Z_{n_j}$ for $1\leq j\leq m$.  From the fiber product construction, it is clear that each $\Z_{n_j}$ acts trivially on $M_k$ for $k\neq j$, but nontrivially on $M_j$.  In particular, each isotropy group has an element $g_j$ that generates the given $\Z_{n_j}$ and acts on $M_j$ by multiplication by $e^{2\pi i/n_j}$.  Additionally, $T_0\mathcal{X}$ is a representation of $R$, on which $K$ acts trivially, but any element mapping to $1\in\Z_r$ under the rigidification map $\mathcal{X}\to\mathcal{C}_{r,s}$ acts as $e^{2\pi i/r}$.  Hence, we see that the vector space
$$
V=\bigoplus_{j=1}^m M_j|_0\oplus T_0\mathcal{X}
$$
is a faithful representation of $R$, and we have identified $R$ as a subgroup of $(S^1)^{m+1}$.  We will identify $S^1$ with $\R/\Z$.  Then the elements $g_j, 1\leq j \leq m$ are identified with $(0,\dots, 0,\frac{1}{n_j},0,\dots, 0)$.

To get a cocycle description of $G_0$, we will need to choose a lifting of $1_r\in \Z_r$ to $R$, and the representation gives us a good one to chose: we will pick the element $g_0$ that acts on $M_j$ by $e^{2\pi i k_j/(rn_j)}$, so that $g_0=(\frac{k_j}{n_jr},\dots,\frac{k_m}{n_mr},\frac{1}{r})$ in $(S^1)^{m+1}$.   Writing down the multiplication in terms of the $g_i, 0\leq i \leq m$ then gives us a $K$ 2-cocycle $\beta$ on $\Z_r$ that describes $R$ as a possibly nontrivial extension of $\Z_r$ by $K$:
\begin{equation} \label{cocycle}
\beta(a,b)=\left\{\begin{array}{ll}
(\frac{k_1}{n_1},\dots,\frac{k_m}{n_m}) & a+b\geq 1\\
0 & a+b <1
\end{array}\right.
\end{equation}

The data of $\beta$ is contained in the element $(\frac{k_1}{n_1},\dots,\frac{k_m}{n_m})\in K$, which we will denote $\kk_0$ for the element describing $R$ and $\kk_\infty$ as the element describing $S$.  Frequently, we will work just with $R$, as the argument over $S$ is analagous, and we'll frequently drop the subscript and just refer to the element as $\kk$.  When we wish to use the cocycle description of elements of $R$, we will write them as $r_i=(a_i, k_i)\in \Z_r\times_\beta K=R$.

Some care is required with this notation, in that typically $-r_i\neq (-a_i, -k_i)$.  In fact, introduce
$$
\delta_r(x)=\left\{
\begin{array}{ll}
0 & x\neq 0 \mod r \\
1 & x=0 \mod r
\end{array}
\right.
$$
and $\delta_r^\vee(x)$ by
$$
\delta_r(x)+\delta_r^\vee(x)=1.
$$
Then we have
\begin{equation} \label{inverse}
-(a,k)=(-a,-k-\delta^\vee_r(a)).
\end{equation}

Recall that different line bundles $L$ can produce the same gerbe.  In this case, they may give us different, but cohomologous, cocyle descriptions of the isotropy groups.

\section{Orbifold Gromov-Witten Theory} \label{GWreview}

The main object of study is $\overline{\mathcal{M}}_{g,n}(\mathcal{X},\beta)$, the moduli space of stable maps $f:\Sigma\to\mathcal{X}$, where $\Sigma$ is an $n$-pointed, genus $g$ nodal orbifold curve, and $f_*[\Sigma]=\beta$.   We will assume familiarity with the basics of Gromov-Witten theory in the smooth case, and provide a sketch of how it differs in the orbifold case.  Orbifold Gromov-Witten theory was first introduced working in the sypmlectic category in \cite{CRGW}, and worked out in the algebraic category in \cite{AGV}.  The basic adjustment is that we wa t to allow our orbifold curves to have some orbifold structure in order to probe the orbifold structure of $\mathcal{X}$, but we keep this to a minimum - the curve $\Sigma$ might be forced to have certain orbifold structure at the marked points or the nodes, and these are the only structures we consider.

As in the manifold case, this moduli space will not in general be smooth, but nevertheless we may construct a virtual fundamental class on it.   Following Gromov-Witten theory in the smooth case, we might then expect that $\mbar_{g,n}(\mathcal{X},\beta)$ would have a virtual fundamental class of dimension
$$\dim \left[\mbar_{g,n}(\mathcal{X},\beta)\right]^{\vir}=(1-g)(\dim\mathcal{X}-3)+n+\leftover{cw_1(T\mathcal{X}),\beta}.$$
An obvious concern is that this is in general only a rational number, as $cw_1(T\mathcal{X})$ will only be a rational class.  We will return to this later, but we mention now that in fact $\mbar_{g,n}(\mathcal{X},\beta)$ will be disconnected, and different components will have virtual classes of different (integral) dimensions.   To understand these components, we need first to establish the second main difference of Orbifold Gromov-Witten theory: the target of the evaluation maps is not $\mathcal{X}$, as one might expect, but the inertia stack $\mathcal{IX}$.  This is best explained by the following important example.

\begin{example}
Consider the moduli space $\mbar_{g,n}(\mathcal{B} G)$, for $G$ a finite group.  The first task is to understand how the evaluation map at the $i$th marked point takes values in the inertia stack $\mathcal{IB}G$.  The inertia stack $\mathcal{IB}G$ has components $\mathcal{IB}G(c)$ indexed by $c\in G_*$, where $G_*$ is the set of conjugacy classes of $G$.  It is useful to follow the conventions of string theory and consider the marked points as being punctures - points removed from the curve.  Consider a map
$$[f:\Sigma\to\mathcal{B}G, p_1,\dots, p_n]$$
in the smooth locus of $\mbar_{g,n}(\mathcal{B}G)$.
This is equivalent to giving a principal $G$-bundle
$$\pi:\widetilde{\Sigma}\to\Sigma\setminus\{p_1,\dots,p_n\}.$$
Transporting the fiber of $\pi$ along a small loop $\Gamma_i$ around the $i$th puncture $p_i$ gives a well defined conjugacy class $c_i\in G_*$.  The evaluation map $ev_i$ sends $[f]$ to $\mathcal{I}(\mathcal{B}G)(c_i)$.

Now, try to extend $f$ over the $i$th puncture.  This is equivalent to extending the principal bundle $\pi$ over the puncture.  If $c_i$ is not the identity, there is nontrivial monodromy around $p_i$, and so the bundle will not extend smoothly.  We will now show, however, that if we give $p_i$ the appropriate orbifold structure, the principal bundle will extend as an orbifold principal bundle.

Working in a neighborhood isomorphic to $\C^*$ around $p_i$, if $c_i$ has order $n$, then we see that if $\phi_n:\C^*\to \C^*$ is the map $\phi_n(z)=z^n$, then $\phi_n^*(\pi)$ has monodromy $c_i^n=1$.  Thus, while $\pi$ does not extend across $0, \phi_n^*(\pi)$ does.
Moreover, choosing an isomorphism of one of the fibers of $\pi$ with $G$ gives both a particular element $\gamma_i$ of the conjugacy class $c_i$, and an isomorphism $\phi_n^*(\pi)\stackrel{\sim}{=}\C^*\times G$.  If we put a $\Z_n$ action on $\C\times G$ by
$$l\cdot (z, g)\mapsto (e^{2\pi i l/n}z, \gamma_i^\ell g)$$
then this extends the $\Z_n$ action on $\phi_n^*(\pi)$ and gives an orbifold chart for an extension of $\pi$ to $p_i$ as an orbifold principal bundle.
A similar phenomenon governs the orbifold structures at the nodes.  Intuitively, we think of nodes arising when some loop $S^1\subset\Sigma$ shrinks to a point.  Restricting $f$ to this $S^1$, we get a map $f:S^1\to \B G$.  If this map corresponds to the trivial principal bundle, then we can add in a non-orbifold node; if it is nontrivial, then to extend the principal bundle to the node we must allow the node to develop an orbifold structure as in the preceding paragraph. Switching which branch of the node we are focusing on reverses the orientation of the $S^1$, and replaces a principal bundle having monodromy $c$ with one having monodromy $c^{-1}$, hence opposite branches of a node must map under the evaluation map to inverse twisted sectors.  This is known as a {\em balanced node}; we only consider maps where all nodes are balanced.

For $\cc=(c_1,\dots, c_n)$ a tuple of $n$ elements of $G_*$, we define the space
$$\mbar_{g,\cc}(\mathcal{B}G)=ev_1^{-1}(\mathcal{IB}G(c_1))\cap\dots\cap ev_n^{-1}(\mathcal{IB}G(c_n)).$$
Since the $\mathcal{IB}G(c_i)$ are open and closed, $\mbar_{g,\cc}$ will be as well, but it need not be a component - in general, it will still be disconnected.

\end{example}

The picture for general $\mathcal{X}$ is similar: if the $i$th marked point maps to a point $x$, then the evaluation map should map to $\mathcal{IB}G_x$, and this picture fits together in a way that globally the evaluation maps land in $\mathcal{IX}$.

Knowing this, we can start to make sense of the apparent fractional dimension of our moduli spaces: the natural cohomology classes we will want to integrate on $\overline{\mathcal{M}}_{g,n}(\mathcal{X},\beta)$ will be pull backs via the evaluation maps of Chen-Ruan cohomology classes, which can have non-integral degree.  The fractional dimension simply means that to get a nonzero number, we should have to pull back classes whose total degree is the dimension.

To understand this even further, we consider the analogs of the subspaces $\mbar_{g,\rrr}(\mathcal{B}G)$ for more general stacks $\mathcal{X}$.  But while all components of $\mbar_{g})(\mathcal{B}G)$ had the same dimension, in general $\overline{\mathcal{M}}_{g,n}(\mathcal{X},\beta)$ will break into open and closed subsets of different (virtual) dimension depending on how the marked points interact with the isotropy of $\mathcal{X}$.   Each marked point $p_i$ will map under the evaluation map to some twisted sector $T_i$.  Since the twisted sectors are open and closed, their inverse images under the evaluation maps will be open and closed as well.  Letting $T=(T_1,\dots,T_n)$ be an $n$-tuple of twisted sectors, we can consider the subspace
\begin{equation*}
\overline{\mathcal{M}}_{g,T}(\mathcal{X},\beta)=\ev_1^{-1}(T_1)\cap\dots\cap\ev_n^{-1}(T_n)\subseteq\overline{\mathcal{M}}_{g,n}(\mathcal{X},\beta)
\end{equation*}
Recalling that the degree shifting number is constant on components, we define
\begin{equation}
\iota(T)=\sum_{k=1}^n \iota(T_k)
\end{equation}
Then $2\iota(T)$ will be the total contribution of the degree shifting numbers from all cohomology classes pulled back via evaluation maps.  Apart from this contribution, the degree of the pulled back cohomology classes will be integral, and so we have that $\mbar_{g,T}(\mathcal{X},\beta)$ has a virtual fundamental class of complex dimension
 $$\dim\left[ \mbar_{g,T}(\mathcal{X},\beta)\right]^{\vir}=(1-g)(\dim\mathcal{X}-3)+n+\leftover{cw_1(T\mathcal{X}),\beta}-\iota(T).$$

We note that there is some subtlety in defining the line bundle $L_i$ corresponding to the cotangent space of the $i$th marked point.  This should in general be an orbifold line bundle, with the standard representation of the group action there.  The usual $\psi$ classes,  $\psi_i=cw_i(L_i)$ take this orbifold structure into account.  However, we could also consider a $\psi$ class that corresponded to the cotangent bundle over the coarse underlying curve, forgetting this orbifold structure.  We denote these classes by $\psibar_i$; on a component where the $i$th marked point has an orbifold structure of order $n$, we have $\psibar_i=n\psi_i$.

\subsection{Equivariant Theory} \label{equivariant}
A $G$ action on $\mathcal{X}$ naturally induces a $G$ action on $\mathcal{IX}$ in its guise as constant maps from $S^1$ to $\mathcal{X}$.  Thus, we define the $G$-equivariant Chen-Ruan cohomology of $\mathcal{X}, H_{CR,G}^*(\mathcal{X})$ to be isomorphic to $H_G^*(\mathcal{IX})$ as a vector space, with the grading shifted by the age, and the cup product deformed by the \emph{equivariant} Euler class of the obstruction bundle, though we will not require the use of the cup product.

We give $\mathcal{C}_{r,s}$ a $\C^*$ action as follows.  Removing the two orbifold points gives a copy of $\C^*$ which we give the standard $\C^*$ action.  This then extends naturally to an action on all of $\mathcal{C}_{r,s}$.  Since $0$ and $\infty$ are fixed points, their tangent spaces, which are already $\Z_r$ and $\Z_s$ representations, are also $\C^*$ representations.  To calculate these representations, consider an orbifold chart around $0$, i.e. a map $f:\widetilde{U}\to U$, with $U$ a neighborhood of $0$ in $\mathcal{C}_{r,s,}$, and $f$ invariant under the $\Z_r$ action on $\widetilde{U}: z\mapsto e^{2\pi i/r}z$.  Then $T_0\mathcal{C}_{r,s}$ is identified with $T_0\widetilde{U}$.  To give $\widetilde{U}$ a $\C^*$ action so that $f$ is equivariant, we see that this tangent space must have weight $1/r$, since $f$ is the map $z\mapsto z^r$. Similarly, $T_\infty\mathcal{C}_{r,s}$ is a $\C^*$ representation with weight $-1/s$.  This remains true when we consider a banded gerbe over $\mathcal{C}_{r,s}$.

The $\C^*$ equivariant cohomology ring of a point is a polynomial ring in one variable $\C[t]$, with $t\in H_{\C^*}^2(\{pt\},\C)$ being the first Chern class of the standard representation of $\C^*$ on $\C$.  By the map $f:X\to pt$, we get a map $f^*:\C[t]\to H^*_{\C^*}$, turning the equivariant cohomology of any space into a $\C[t]$-module.

Let $i:F\to X$ be the inclusion of the fixed point set $F$ of a $\C^*$ action on $X$.  Then Atiyah-Bott localization says (see \cite{AB} for an introduction) that, after localizing the appropriate element of $\C[t]$, the equivariant cohomology rings of $\mathcal{X}$ and $F$ are isomorphic, with explicit isomorphisms given by $i^*$ and $\frac{i_*}{e_{\C^*}(N_F)}$.  As a result, we will only need to understand the fixed point sets of our spaces and their equivariant normal bundles (or virtual normal bundles) in the larger spaces.

The inertial orbifold $\mathcal{IX}$ will have will have $|K|$ components isomorphic to $\mathcal{X}$, and $(r-1)|K|$ or $(s-1)|K|$ components isomorphic to $\mathcal{B}R$ or $\mathcal{B}S$, respectively.  Thus, the fixed point set of the action on $\mathcal{IX}$ will consist of a disjoint union of points: there will be $|R|=r|K|$ components over 0, and $|S|$ components over $\infty$.  The cocycle $\beta$ we constructed in \ref{cocycle} gives us a bijection between $R$ and $K\times \Z_r$, and we will use this to identity the fixed point components of the $\C^*$ action on $\mathcal{IX}$, and hence to identify a basis of the localized cohomology.  We will denote by $\zero(a, k), \finity(b,k), a\in\Z_r, b\in\Z_s, k\in K$ the generators of the cohomology of the component fixed point set of $\mathcal{IX}$ corresponding to the element $(a,k)\in R, (b,k)\in S$, respectively.  By the degree shifting, we see that $\zero(a,k)\in H^{2a}_{CR}(\mathcal{X})$, and $\finity(b,k)\in H^{2b}_{CR}(\mathcal{X})$, where we recall our abusive identification of $\Z_n$ with the corresponding subgroup of $\Q/\Z$.

 The main result of this paper is the calculation of equivariant Gromov-Witten invariants.  Our target space $\mathcal{X}$ will have a $\C^*$ action, which naturally induces a $\C^*$ action on $\overline{\mathcal{M}}_{g,n}(\mathcal{X},\beta)$.  Then there will be a virtual fundamental class of the expected dimension in equivariant homology, the $\psi$ classes will be equivariant, and we will be able to pull back and integrate equivariant classes from the target.

We will use the following notation for the equivariant Gromov-Witten invariants:
\begin{equation*}
\left\langle\prod_{i=1}^n\tau_{k_i}(\gamma_i)\right\rangle_{\mathcal{X}, g,\beta}^\circ=\int_{[\overline{\mathcal{M}}_{g,n}(\mathcal{X},\beta)]_{\C^*}}\prod_{i=1}^n\psibar_i^{k_i}\ev_i^*(\gamma_i),
\end{equation*}
where $\gamma_i\in H_{orb,\C^*}^*(\mathcal{X},\C)$, and $\beta\in H_2(\mathcal{X})$.  The superscript $\circ$ denotes the connected theory, while $\langle\quad\rangle^\bullet$ represents the theory where possibly disconnected domain curves are allowed.

The main object of study will be the equivariant Gromov-Witten potential function $F$ of $\mathcal{X}$.  We introduce variables $\{x_i(a,k)\}, \{x^*_i(b,h)\}$, corresponding to insertions of $\tau_i(\zero(a,k))$ and  $\tau_i(\finity(b,k))$, respectively.  Then we define
\begin{multline}
\tau=\sum_{g=0}^\infty\sum_{d=0}^\infty u^{2g-2}q^d \notag \\
\left\langle\exp\left(\sum_{i,a,k}x_i(a,k)\tau_i(\zero(a,k))+\sum_{j,b,h} x^*_j(b,h) \tau_j(\finity(b,h))\right)\right\rangle^\bullet_{\mathcal{X},g,d} \notag.
\end{multline}

\subsection{Orbifold Structure and Maps from Curves}
\label{maps}

Our goal in this section is to understand how orbifold structures on $\mathcal{X}$ affect maps from curves into $\mathcal{X}$.  Due to localization, we only need to under $\C^*$ invariant maps, which have a particular simple form.  Contracted components must be mapped to a fixed point - namely zero or infinity.  Furthermore, all ramification must happen over zero or infinity, and so the only noncontracted maps possible are topologically the standard $z\mapsto z^d$ maps from $\proj^1\to\proj^1$.  In this section we work describe the behavior of the orbifold structure of such maps.  The behavior of the effective isotropy is determined locally by the degree, while determining the interaction with the gerby isotropy is a global phenomenon: the behavior at $0$ effects the behavior at $\infty$.

The interaction of the degree of the map and the image in the effective quotient of the isotropy group is contained in the following:

\begin{lemma}  \label{effectivelemma} Let $\Z_r$ and $\Z_n$ act on $\C$ as their standard embeddings in $\C^*$, and let $f:\C/\Z_r\to\C/\Z_n$ be a representable map of orbifolds which on coarse moduli spaces gives the map $z\mapsto z^d$.    Then $r=n/\gcd(d,n)$ and the standard generator of $\Z_r$ maps to $d\in\Z_n$.
\end{lemma}
\begin{proof}
The map $f:\C/\Z_r\to \C/\Z_d$ must lift to an equivariant map $g$ from $\C\to\C$ which covers $f$:
\begin{equation*}
\begin{CD}
\C @ >g>> \C\\
@Vz^r VV @V z^n VV \\
\mathcal{X}@> f >> \C/\Z_n
\end{CD}
\end{equation*}
Since $f$ is of the form $z\mapsto z^d$, then we must have $g(z)=z^a$, and commutativity gives $an=rd$.  Then there is some $k$ with $a=kd/\gcd(d,n)$ and $r=kn/\gcd(d,n)$, we need to show that $k=1$.

Suppose the generator $1\in\Z_r$ maps to $l\in\Z_n$, then since $g$ is equivariant, we have
\begin{equation*}
e^{2\pi i a/r}z^a=g(e^{2\pi i/r}z)=e^{2\pi i l/r} g(z)=e^{2\pi i l/r}z^a
\end{equation*}

Since $f$ is representable, the map on isotropy groups must be injective, and so we must have that $e^{2\pi i a /r}$ has order $r$, that is, $a$ and $r$ are relatively prime - which forces $k=1$.

Finally, we see that the generator of $\Z_r$ maps to $a/r=d/n\in\Z_d$.
\end{proof}

In contrast to the effective part, the image of a degree $d$ map in the ineffective part of the isotropy is completely unconstrained locally.  There are, however, global monodromy constraints.  We will only need to consider maps from $\proj^1$ with two orbifold points, mapping to zero and infinity. The key point is that if the degree of the map and the orbifold behavior of one of the points is fixed, the orbifold behavior at the other marked point uniquely determined by the gerbe structure.  We prove this in the case of the $\Z_n$ gerbe coming from a line bundle $L$; the general case follows via the fibered product construction.

\begin{lemma} \label{gerbelemma}
Suppose that $\mathcal{C}$ is an orbifold is that is topologically a $\proj^1$ with orbifold structure only over $0$ and $\infty, \mathcal{X}=\mathcal{C}^{(L,n)}_{r,s}$, and $f:\mathcal{C}\to\mathcal{X}$ is a representable, $\C^*$ fixed map of degree $d$.  Suppose $1_r\in\Z_r$ acts on $L_0$ by $a/r, 1_s\in\Z_s$ acts on $L_\infty$ by $b/s$.  Then $cw_1(L)=\ell+a/r+b/s$ for some $\ell\in\Z$.

Then if the generator of the isotropy group of $0$ in $\mathcal{C}$ maps to $(d,u)\in R$, and the generator of the isotropy over $\infty$ maps to $(d, v)\in S$, we have that
$$d\ell+\floor{\frac{d}{r}}a+\floor{\frac{d}{s}}b=u+v\mod n.$$
\end{lemma}

Since $u$ and $v$ are in $\Z_n$, this determines one from the other.

\begin{proof}
By lemma \ref{effectivelemma}, the image of the isotropy group in the effective parts of the isotropy groups are indeed as given, and so we must show that the ineffective parts of the isotropy satisfy the above relation.

 By construction, over $\mathcal{X}$ $L$ has an $n$th root $M$.  We have
 \begin{equation} \label{cwgerbe}
 cw_1(f^*(M))=\frac{d}{n}cw_1(L)=\frac{d\ell}{n}+\frac{d}{r}\frac{a}{n}+\frac{d}{s}\frac{b}{n}
 \end{equation}

 On the other hand, we know that the fractional part of $cw_1(f^*(M))$ is determined by the behavior of the isotropy groups on $M$, which are known:  $1\in\Z_{n}$ acts as $1/n$, and $1/r\in\Z_r, 1/s\in\Z_s$ act by $a/(nr), b/(ns)$, respectively.

 We see then that the generator of the isotropy group at $0$ on $\mathcal{C}$ acts on $f^*(M)$ by $\leftover{\frac{d}{r}}\frac{a}{n}+\frac{u}{n}$, while the generator of the isotropy group at $\infty$ acts on $f^*(M)$ by $\leftover{\frac{d}{s}}\frac{b}{n}+\frac{v}{n}$.

Subtracting these contributions from the total Chern-Weil class of $f^*(M)$ in (\ref{cwgerbe}), we see that the contribution from $0$ can be viewed as
\begin{equation*}
\frac{d}{r}\frac{a}{n}-\leftover{\frac{d}{r}}\frac{a}{n}-\frac{u}{n}=\floor{\frac{d}{r}}\frac{a}{n}-\frac{u}{n}
\end{equation*}
and a similar equation holds for the contribution from zero.  Thus we see that
\begin{equation*}
\frac{d\ell}{n}+\frac{a}{n}\floor{\frac{d}{r}}+\frac{b}{n}\floor{\frac{d}{s}}-\frac{u}{n}-\frac{v}{n}
\end{equation*}
must be an integer, which is the desired result.

\end{proof}

For the fibered product case, with $\mathcal{X}$ a $K$ gerbe over $\mathcal{C}_{r,s}$, $u$ and $v$ will be elements of $K$.  The result of our lemma will be an equation that holds in each $\Z_{n_i}$, with $a, b, \ell$ replaced by $a_i, b_i, \ell_i$.  The $a_i$ and $b_i$ package together to $\kk_0$ and $\kk_\infty$, respectively, and we will package the $\ell_i$ as $\mathbb{L}$, so that we have:
\begin{equation} \label{gerbelemmafinal}
u+v=d\mathbb{L}+\floor{\frac{d}{r}}\kk_0+\floor{\frac{d}{s}}\kk_\infty
\end{equation}
as an equation in $K$.

Note that this monodromy condition seemingly depends upon which line bundle we pick, and not just the gerbe.  This is because different line bundles produce different cocycles for the group extension - the changes in monodromy a different line bundle gives are exactly what are needed to account for the different cocycle.

\subsection{Results on $\Mbar(\mathcal{B}R)$} \label{mbar}
In the last section we derived what we would need to know about the positive degree maps; in this section we examine the contracted maps.
We will be interested in the moduli spaces $\Mbar_{g,n}(\mathcal{B}R)$, with $R$ a finite abelian group.  The evaluation maps $\ev_i$ take values in $\mathcal{IB}R$, which since $R$ is abelian is the disjoint union of $|R|$ components, $\mathcal{IB}G(y), y\in R$.  We will use
$$\rrr=(r_1,\dots, r_n)$$ to denote an $n$-tuple of elements of $R$, and so work with $\mbar_{g,\rrr}(\mathcal{B}R)$.

Note that the $\mbar_{g,\rrr}(\mathcal{B}R)$ may themselves be composed of multiple components; we have fixed the monodromy of the $R$ cover around the marked points, but not the monodromy around the $2g$ noncontractable curves.  In particular, in the case the $r_i=0$ are all trivial, we have the {\em trivial monodromy component}, consisting of trivial covers.  Similarly, in case the $r_i$ are all contained in some subgroup $H<R$, the subset of covers where all the monodromy is contained in $H$ will be a union of components.

Over $\Mbar_{g,\rrr}(\mathcal{B}R)$ the orbifold curves $\mathcal{C}$ and their principal $R$-bundles $\widetilde{\mathcal{C}}$ fit together into universal curves $\mathcal{U}=[\widetilde{\mathcal{U}}/R]$.  There is a bundle $\E$, the Hodge bundle, over $\Mbar_{g,\rrr}(\mathcal{B}R)$, whose fiber over a point is $H^0(\widetilde{\mathcal{C}},\omega_{\widetilde{\mathcal{C}}})$, i.e. sections of the dualizing sheaf.  The $R$ action on $\widetilde{\mathcal{U}}$ induces an $R$ action on $\E$, and thus $\E$ will be split into sub-bundles on which $R$ acts by its irreducible representations.  We will label these subbundles either by the representation or the irreducible character $\rho\in R^*$ it affords:
\begin{equation}
\E=\sum_{\rho\in R^*} \E_\rho
\end{equation}

The bundles $\E_\rho$ are called {\em Hurwitz-Hodge bundles}.  We will denote their chern classes by
$$\lambda_i^\rho=c_i(\E_\rho).$$
Integrals on $\mbar_{g,\rrr}(\mathcal{B}R)$ of $\lambda_i^\rho$ and $\psibar_i$ are called {\em Hurwitz-Hodge integrals}.

In chapter \ref{localization}, localization will reduce the calculation of equivariant Gromov-Witten invariants of $\mathcal{X}$ to certain Hurwitz-Hodge integrals.  The integrals appearing will be those corresponding to the $R$ representation $T_0\mathcal{X}$.  Since $T_0\mathcal{X}$ is one dimensional, this arises from a multiplicative character $\phi_0:R\to\C^*$.  If $\mathcal{X}$ is ineffective, $\phi_0$ will have kernel $K$; in any case, its image will exactly be $\Z_r\subset\C^*$.  Let $U$ be one dimensional representation of $\Z_r$ induced by the standard inclusion $\Z_r\subset \C^*$, in other words, on $U$, $1$ acts by multiplication by $e^{2\pi i/r}$.  Then
$$T_0\stackrel{\sim}{=}\phi_0^{-1}(U).$$

In fact, the map $\phi_0$ induces a morphism
$$\mbar_{g,\rrr}(\mathcal{B}R)\stackrel{\widetilde{\phi}_0}{\to}\mbar_{g,\phi_0(\rrr})(\mathcal{B}\Z_r)$$
by taking the quotient of the $K$ action on each cover (see \cite{JPT}),
and $\E_{T_0}=\widetilde{\phi}_0^*\E_{U}$.

It will be convenient to know the dimension of $\E_{T_0}$ over $\mbar_{g,\rrr}(\mathcal{B}R)$.   From the above, we see that this is the dimension of $\E_{U}$ over $\mbar_{g,\phi_0(\rrr)}(\mathcal{B}\Z_r)$, and so in particular the dimension depends only on the image of $\rrr$ in $\Z_r$.

The orbifold Riemann-Roch formula computes this dimension:
\begin{equation} \label{orbifoldriemannroch}
\dim \E_{T_0}=g-1+\iota(\rrr)+\delta_K
\end{equation}
here $\iota(\rrr)$ is the degree shifting number of the total space $T_0$, i.e., the rational number obtained by taking the image of $R$ in $\Z_r$, identifying elements of $\Z_r$ with the rational numbers $a/r, 0\leq a<r$, and then adding them in $\Q$, and
\begin{equation} \label{deltaK}
\delta_K=
\left \{
\begin{array}{ll}
1 & \text{on those components where the monodromy generates a subgroup of $K$} \\
0 & \text{on all other components}
\end{array}
\right. .
\end{equation}

\subsection{Decomposition and Discrete Torsion}

 This section examines the decomposition conjecture of \cite{HHPSA} and how it pertains to our situation.  The precise general statement of the decomposition conjecture is somewhat involved, and involves twisting Gromov-Witten theory by a flat $\C^*$ gerbe.  Although the effects of this twisting in general are highly nontrivial, in our case the twisting is extremely simple, and amounts to simply rescaling some of the variables of the generating function.  Thus, this section can be skipped without much loss.  We begin with a brief and incomplete discussion of twisting by flat gerbes before explaining the general form of the decomposition conjecture and sketching that the twistings relevant in our case can be entirely captured by rescaling variables.

Twisted Gromov-Witten theory was introduced mathematically in \cite{RuanDiscreteTorsion} and \cite{PRY}, although it had existed in some form in the physics literature under the name discrete torsion since \cite{V}.  There, given a global quotient orbifold $\mathcal{X}=Y/G$ for some finite group $G$, Vafa shows how to twist by an element $\alpha\in H^2(G, S^1)$.  This has since been generalized to twisting by a flat $S^1$ gerbe with connection, which are classified by $H^2(\mathcal{X}, S^1)$.  In the case of a global quotient $\mathcal{X}= Y/G$, there is an induced map $\mathcal{X}\to\mathcal{B}G$, and so pulling back cohomology classes we see that twisting by a flat gerbe indeed extends Vafa
s discrete torsion.  We note that while the mathematical literature cited above reserves the term ``discrete torsion'' for twisting by an element in $H^2(G, S^1)$,  the physics literature appears to use it to reference any such twisting: see e.g. \cite{S1}.

These twistings should be understood as an extended and orbifold version of the physical notion of ``B-fields''.  We briefly recall this story in the case of a smooth manifold.  Mirror symmetry predicts that Gromov-Witten theory should have $H^2(X,\C)$ as a parameter space.  The real part $H^2(X,\R)$ corresponds to the choice of symplectic form, the imaginary part $H^2(X, i\R)$ corresponds to the B-fields.  Physically, the twisting winds up appearing in an exponent, and so the only dependence is on the class up to the image of $H^2(X, 2\pi i\Z)$.  By the long exact sequence induced from
$$0\to\Z\to\R\to S^1\to 0$$
we see that the space of $B$-fields includes into the group $H^2(X, S^1)$, with cokernel the torsion part of $H^3(X, \Z)$.

These cohomology groups have geometric significance.  As mentioned above $H^2(X, S^1)$ classifies isomorphism classes of flat $S^1$ gerbes with connection, and $H^3(X, \Z)$ classifies isomorphism types of $S^1$ gerbes, with the torsion part being flat gerbes.  The map between the two cohomology groups corresponds to forgetting the connection.  Thus, it appears that classically a $B$-field corresponds to a choice of flat connection on the trivial $S^1$ gerbe, which provides some motivation for the idea of trying to twist Gromov-Witten theory by nontrivial $S^1$ gerbes with flat connection.

For smooth $X$, the twisting procedure traces through some complicated geometry only to result in simple algebra.  The state space of twisted Gromov-Witten invariants are the same, namely $H^*(X,\C)$.   Twisting by a class $\Phi\in H^2(X, S^1)$ simply multiplies the Gromov-Witten invariants with curve class $\beta\in H_2(X, \Z)$ by $\Phi(\beta)$ - which is easily capture by rescaling the degree variable $q$ by the appropriate root of unity.  To give the briefest sketch of the story, the flat gerbe with connection gives rise to the holonomy line bundle, a line bundle with connection on the loop space $LX$.  The state space should really be the cohomology of $X$ with coefficients in the holonomy line bundle restricted to $X\subset LX$ as the space of constant loops; it turns out that this line bundle on $X$ is canonically trivial, and thus our state space is the usual cohomology of $X$.

For a flat gerbe over an orbifold $\mathcal{X}$, we play the same game, but things are more complicated.  Again, from the flat gerbe a line bundle is constructed on $L\mathcal{X}$.  The space of constant loops gives a containment $\mathcal{IX}\subset L\mathcal{X}$, and the state space of the twisted theory is the cohomology of the cohomology of the holonomy line bundle restricted to $\mathcal{IX}$.  The holonomy line bundle restricted to $\mathcal{IX}$, with some related structure, is known as an inner local system, which are used to twist Chen-Ruan cohomology.  Furthermore, while topologically trivial gerbes produce topological trivial inner local systems, there is no longer a canonical trivialization.  As a result, in this case, the twisting is only slightly more complicated than the twisting in the smooth case: in addition to rescaling the degree variable $q$, we must also rescale the cohomology variables $x_i$.

We now give an explanation of the decomposition conjecture, restricting ourselves to the case of abelian groups for simplicity.

Recall that part of the cohomological classification of $K$ gerbes was a principal $\Aut(K)$ bundle over $\mathcal{X}$.  Using the obvious action of $\Aut(K)$ on $K^*$ - the set of irreducible representations of $K$ - we construct the associated principal $K^*$ bundle over $\mathcal{X}$, which we denote $\mathcal{Y}$.  The decomposition conjecture asserts that up to a physically meaningless rescaling of the genus variable $u$, the Gromov-Witten invariants of $\mathcal{K}$ are equal to particular twisted Gromov-Witten invariants of $\mathcal{Y}$.  It is known as the decomposition conjecture because $\mathcal{Y}$ will in general be disconnected, and so the Gromov-Witten theory of $\mathcal{X}$ will decompose as a product of the twisted Gromov-Witten invariants of the components of $\mathcal{Y}$.

In the case we will be interested in, the space $\mathcal{Y}$ and the flat gerbes we twist by take a particularly simple form.  For a trivially banded abelian gerbe, which we are primarily interested in, the $\Aut(K)$ bundle is trivial, and so $\mathcal{Y}$ will consist of one copy of $\mathcal{X}$ for each element of $K^*$.  Since $K$ is abelian, each such representation will be one dimensional, and hence be equivalent to a homomorphism $\varphi:K\to\C^*$.  Since trivially banded abelian gerbes are classified by $\alpha\in H^2(\mathcal{X},K)$, we see an easy way to construct the cohomology class of a $C^*$ gerbe with connection:  on the component of $\mathcal{Y}$ labeled by $\varphi$, we take the image $\varphi_*(\alpha)$ of $\alpha$ under the map
$$H^2(\mathcal{X}, K)\stackrel{\varphi_*}{\to} H^2(\mathcal{X},\C^*)$$
induced by $\varphi$.  The decomposition conjecture states that the Gromov-Witten invariants of $\mathcal{X}$ are those of $\mathcal{Y}$, twisted by $\varphi_*(\alpha)$ on the component of $\alpha$ labeled by $\varphi$.

In our case, this twisting simply results in rescaling of variables.  Since $H^3(\mathcal{C}_{r,s},\Z)=0$, we see that this class must correspond to a trivial $\C^*$ gerbe, but with a potentially nontrivially connection;  and so the resulting holonomy line bundle and inner local system must also be trivial.  However, since the inner local system is not canonically trivial, some rescaling of the cohomology variables also appears.

\chapter{Localization}
 \label{localization}

In this section we carry out Atiyah-Bott localization with respect to the induced $\C^*$ action on the moduli space of maps.  As is typical, this allows us to express any Gromov-Witten invariant as a sum over certain labeled trees, with terms weighted by linear Hurwitz-Hodge integrals.  While sums of trees are complicated to deal with, by working with the disconnected generating function we find that we can instead write this as a sum over partitions in the effective case, or sums of $K$-labeled partitions in the case of a $K$-gerbe.

\section{Generating Functions}

Localization will express the Gromov-Witten invariants of $\mathcal{X}$ in terms of integrals of tautological classes over $\overline{M}_{g,\rrr}(\mathcal{B}R)$ and $\overline{M}_{g,\sss}(\mathcal{B}S)$.  These integrals are conveniently encoded in the generating function:
$$
H^{0,\circ}_{g,\rrr}(z_1,\dots, z_n)
=
\int_{\overline{\mathcal{M}}_{g,\rrr}(\mathcal{B}R)}
\prod_{i=1}^{\ell(\rrr)} \frac{z_i}{1-z_i\psibar_i}
\sum_{i=0}^\infty (-r)^i\lambda_i^{T_0}.
$$

Similarly, $H^{\infty,\circ}_{g,\sss}$ encodes the Hodge integrals that occur at $\infty$; we will sometimes omit the superscript.  As usual, $\bullet$ will denote the disconnected theory, and if neither symbol is present the connected theory is used.

In the cases where $\overline{\mathcal{M}}_{g,\rrr}(\mathcal{B}R)$ is unstable, we will find it convenient to set the value of $H^\circ_g$ as follows.  First, if $\sum_{i=1}^n r_i\neq 0\in R$, the monodromy condition is not met, and we set $H^\circ_{g,\rrr}=0$.  The remaining balanced contributions are:

\begin{equation} \label{unstableintegrals}
H^0_{0, id}(z)=\frac{1}{|R|z},\quad H^0_{0,(r_1, -r_1)}(z_1, z_2)=\frac{z_1z_2}{|R|(z_1+z_2)}
\end{equation}

Including the unstable terms will allow for a uniform treatment of localization.  Note that in the stable cases $H^\circ_{g,\rrr}$ is a polynomial, while in the unstable cases it is only a rational function.  This fact will allow us to easily remove the unwanted unstable cases later.

We assemble the $H_g$ into an all genus generating function:
$$
 H^\circ_{\rrr}(z_1,\dots, z_n,u)=\sum_{g\geq 0} u^{2g-2}H^\circ_{g,\rrr}(z_1,\dots, z_n).
$$
We use $H_{\rrr}^\bullet(z_1,\dots,z_n,u)$ to denote the disconnected function, where our source curve is potentially disconnected:
\begin{eqnarray*}
H^\bullet_{\rrr}(z_1,\dots, z_n,u)&=&\sum_{P\in\mathrm{Part}(\rrr)}\prod_{i=1}^{\ell(P)} H^\circ_{P_i}(z_{P_i},u) \\
&=&
\int_{\overline{\mathcal{M}}^\bullet_{g,\rrr}(\mathcal{B}R)}
\prod_{i=1}^n \frac{z_i}{1-z_i\psibar_i}
\sum_{i=0}^m (-r)^i\lambda_i^{T_0}.
\end{eqnarray*}

Here $\mathrm{Part}(\rrr)$ denotes the set of partitions of the set $\{1,\dots,n\}$, with possibly empty parts, and $\ell(P)$ denotes the number of parts of the partition $P$.  Note that we can break $H^\bullet$ into a genus expansion, but must allow curves of negative genus:
$$
 H^\bullet_{\rrr}(z_{\rrr}, u)=\sum_{g\in\Z} u^{2g-2}H^\bullet_{g,\rrr}(z_{\rrr}).
$$
Because of the unstable contributions, $H^\bullet_{\rrr}(z_{\rrr})$ will be a rational function, with simple poles occurring at $z_i=0$ for those $i$ with $r_i=0$ and at $z_i+z_j=0$ when $r_i=-r_j$.

We denote by $G^\circ_{g,d,\rrr,\sss}(z_{\rrr},w_{\sss})$ the $\ell(\rrr)+\ell(\sss)$-point function of genus $g$, degree $d$ equivariant Gromow-Witten invariants of $\mathcal{X}$:
$$
G^\circ_{g,d,\rrr,\sss}(z_{\rrr}, w_{\sss})=
\int_{[\overline{\mathcal{M}}_{g, \rrr +\sss}(\mathcal{X},d)]^{\vir}_{\C^*}}
\prod_{i=1}^{\ell(\rrr)}\frac{z_i\ev_i^*(\zero_{r_i})}{1-z_i\psibar_i}
 \prod_{j=1}^{\ell(\sss)}\frac{w_j\ev_j^*(\finity_{s_j})}{1-w_j\psibar_j} .
 $$
When $d=0$ and the moduli space would be unstable, we make the following conventions, which are compatible with the localization procedure and the unstable contributions we defined earlier.
All unstable 0-point functions are set to 0:
$$
G^\circ_{0,0}()=G^\circ_{1,0}()=0
$$
Any unstable 1 or 2 point functions that would be empty by monodromy considerations are set to zero, i.e. if $\sum r_i$ or $\sum s_j$ are nonzero.  The remaining 1 and 2-point functions are defined as follows:
\begin{gather} \label{unstableG}
G^\circ_{0,\text{id}_R}(z_1)=\frac{1}{|R|z_1},\quad G^\circ_{0,\text{id}_S}(w_1)=\frac{1}{|S|w_1} \\
G^\circ_{0,0,r_1,-r_1}(z_1,z_2)=\frac{tz_1z_2}{|R|(z_1+z_2)},\quad G^\circ_{0,0,s_1,-s_1}(w_1,w_2)=\frac{tw_1w_2}{|S|(w_1+w_2)}\notag \\
\quad G^\circ_{0,0,\{0\},\{0\}}(z_1,w_1)=0. \notag
\end{gather}
We define $G^\circ_{d,\rrr,\sss}(z_{\rrr},w_{\sss},u)$ to take into account all genus invariants:
\begin{equation*}
G_{d,\rrr,\sss}^\circ(z_{\rrr},w_{\sss},u)=\sum_{g\geq 0}u^{2g-2}G^\circ_{g,d,\rrr,\sss}(z_{\rrr}, w_{\sss}).
\end{equation*}

Similarly, we denote the disconnected functions by $G^\bullet_{d,\rrr,\sss}(z_{\rrr},w_{\sss}, u)$.
\begin{equation*}
G^\bullet_{d,\rrr,\sss}(z_{\rrr},w_{\sss},u)=
\sum_{P\in\mathrm{Part}_d(\rrr,\sss)}\frac{1}{\Aut(P)}\prod_{i=1}^{\ell(P)}G^\circ_{d_i}(z_{P^{\phantom{\prime}}_i},w_{P_i^\prime},u).
\end{equation*}
Here an element $P\in\mathrm{Part}_d(\rrr, \sss)$ is a set of triples $(d_i, P^{\phantom{\prime}}_i, P^\prime_i)$ such that the $d_i$ form a partition of $d$, where some parts could be 0, and the $P^{\phantom{\prime}}_i, P^\prime_i$ form partitions of $\rrr, \sss$, respectively, where some parts are allowed to be empty.  Since the unstable zero point, zero degree functions are defined to vanish only a finite number of partitions have nonzero contribution to any given term.

\section{Localization graphs and weighted partitions}

In this section we describe the fixed point loci of the $\C^*$ action on $\overline{\mathcal{M}}_{g,\rrr,\sss}(\mathcal{X},d)$.  For connected curves, the loci will be indexed by certain labeled graphs $\Gamma$, and the corresponding fixed point locus will be denoted $\mbar_\Gamma$.  However, summing over all graphs is a complicated procedure.  Considering disconnected curves simplifies matters: in this case, we will sum over all $K$-labeled partitions of $d$.

The sets of edges $e$ and vertices $v$ of $\Gamma$ will be denoted $E(\Gamma)$ and $V(\Gamma)$, respectively.  The vertices of the graph will represent contracted components, while the edges will represent components mapping to $\mathcal{X}$ with positive degree.  An incident edge-vertex pair will be called a {\em flag} of $\Gamma$, and will be denoted $F$, with $F(\Gamma)$ being the set of all flags; a flag then represents a node between a contracted and noncontracted component.

Consider a stable map $f:C\to\mathcal{X}$ fixed under the induced $\C^*$ action on $\overline{\mathcal{M}}_{g,\rrr,\sss}(\mathcal{X},d)$. Any marked point, node, or contracted component must map to a fixed point of $\mathcal{X}$, namely $0$ or $\infty$.  Furthermore, any ramification points of a noncontracted component must lie over $0$ or $\infty$ as well.  Thus, any noncontracted component can be ramified over at most two points, and so, on the level of coarse curves the only possible noncontracted component allowed is the standard degree $d$ map $\proj^1\to\proj^1, z\mapsto z^d$, which we call \emph{edge} maps, because they will be represented by edges of $\Gamma$.

If we have two contracted components connected directly by a node, we may smooth that node and remain a fixed map.  We call a maximal set of contracted components connected by nodes a \emph{vertex} map, represented by vertices of $\Gamma$.

In contrast to nodes between contracted components, a node between an edge and a vertex map cannot be smoothed without leaving the fixed point locus.  So the fixed point loci correspond to bipartite graphs $\Gamma$ with the following labels.

Each edge $e\in E(\Gamma)$ carries a labeling of the degree $d(e)\in\Z_{\geq 1}$ of the map from $\proj^1\to\proj^1$ that the edge represents.  Additionally, we must specify the behavior of each edge with regard to the gerbe structure.  By Lemma \ref{gerbelemma}, to do so it is enough to specify a single element $k(e)\in K$.  We choose this element so that $(d(e), -k(e)+ \floor{\frac{d(e)}{s}}\kk_\infty)\in S$ is the monodromy on the {\em edge side} of the node.  Let $\sedge(d)$ denote the monodromy on the
{\em vertex side} of the node over $\infty$.  Recalling Equation (\ref{inverse})
$$
-(a,k)=(-a,-k-\delta^\vee_r(a))
$$
and using
\begin{equation} \label{floortrick}
\floor{\frac{a}{r}}+\floor{\frac{-a}{r}}=-\delta^\vee_r(a)
\end{equation}
 we have that:
\begin{equation} \label{sedge}
\sedge(e)=\left(-d(e),k(e)+\floor{\frac{-d(e)}{s}}\kk_\infty\right).
\end{equation}
The factor of $\floor{\frac{\pm d(e)}{s}}\kk_\infty$ appears awkward, but it is a convenient, symmetric way to account for the Equation (\ref{inverse}) for $-(a,k)$.   Additionally, this choice of $k(e)$ will be convenient later in our operator description of Gromov-Witten theory.

 Given the monodromy at $\infty$ and the degree of the map, Lemma \ref{gerbelemma} determines the monodromy at $0$.   In particular, for a degree $d$ map with monodromy $(d(e), -k(e)+ \floor{\frac{d(e)}{s}\kk_\infty})$ at infinity, the monodromy on the edge side at 0 must be $(d(e), k(e)+\floor{\frac{d(e)}{r}}\kk_0+d(e)\mathbb{L})$ from \ref{gerbelemmafinal}.  Then if $\redge(e)$ is the monodromy on the {\em vertex side} over $\infty$, we have, similar to Equation (\ref{sedge}), that
\begin{equation*}
\redge(e)=\left(-d(e), -k(e)-d(e)\mathbb{L}+\floor{\frac{-d(e)}{r}}\kk_0\right).
\end{equation*}
Had we not included the extra term in the definition of $\sedge(e)$, it would have appeared in the formula for $\redge(e)$.  If $F$ is a flag on $e$, we will sometimes write $\redge(F)$ or $\sedge(F)$ to denote $\redge(e)$ or $\sedge(e)$.

Each vertex $v\in V(\Gamma)$ carries the labeling of which fixed point it mapped to - we will write $v_0$ or $v_\infty$ when we want to indicate that a vertex is mapped to zero or infinity.   The genus of the contracted curve will be denoted $g(v)$, and the marked points contained on the contracted curve, with orbifold data, will be a subset $\rrr(v_0)\subseteq \{1,\dots,\ell(\rrr)\}$  (or $\sss(v_\infty)\subseteq \{1,\dots,\ell(\sss)\}$).  We will find it convenient to write $e(v)$ for the number of edges incident to a vertex $e$, and choose a labeling for them: $e_1,\dots, e_{e(v)}$.  We will use $d_i$ to denote the degree of $e_i$.

Each vertex curve will have a marked point where it is glued to each adjacent edge.  The image of this marked point in the isotropy group is determined by the behavior of the edge as described above.  We will write $\redge(v_0)$ or $\sedge(v_\infty)$ to denote the tuples $\redge(e_i)$ or $\sedge(e_i)$, where $i$ ranges over all adjacent edges.

We will want some additional notation when working with disconnected curves.  For any map, the set of $d(e)$ and $k(e)$ together form a $K$-weighted partition of $d$, which we will denote $\overline{\mu}=\{(\mu_i, k_i)\}$.  We will write $\redge(\overline{\mu})$ and $\sedge(\overline{\mu})$ to denote the set of all $\redge(e_i)$ and $\sedge(e_i)$, and $\redge(\overline{\mu}_i)$ to denote $\redge(e_i)$.  We will write $g_0$ and $g_\infty$ for the genus of the (disconnected) curves over $0$ and $\infty$.

To determine the topology of the the fixed point locus $\mbar_\Gamma$, we note that the only deformations allowed while staying within the fixed point locus are deforming the vertex curves.  Thus, each vertex $v$ will contribute a moduli space of stable maps to $R$ or $S$, which we denote $\mbar_v$.  From the above, we see that
\begin{equation*}
\mbar_{v_0}=\mbar_{g(v_0),\rrr(v_0)+\redge(v_0)}(\B R)
\end{equation*}
and similarly for a vertex over infinity.

We must also keep track of the isometries of the curve.  Automorphisms of the graph preserving all labelings will give automorphisms of curves.  Automorphisms of the vertex curves are included in the fine moduli spaces $\mbar_v$.

 For an edge curve, we have the usual automorphisms of the degree $d$ map obtained by rotating by a $d$th root of unity.  Additionally, in the presence of a gerbe, each edge has an additional $|K|$ worth of automorphisms, see, e.g., \cite{CC}.

Finally, there are subtle factors coming from gluing nodes together.  As presented in \cite{AGV}, maps from a nodal curve $C$ with components $C_1$ and $C_2$ glue along the rigidified inertia stack
\begin{equation}
\hom(C,\mathcal{X})=\hom(C_1,\mathcal{X})\amalg_{\overline{\mathcal{I}}(\mathcal{X})}\hom(C_2,\mathcal{X})
\end{equation}
The important point is that the gluing really happens not over the inertia stack $\mathcal{I}$, but over the rigidified inertia stack $\overline{\mathcal{I}}$.  Recall that point $(x,g)$ of the inertia stack has isotropy group $G_{(x,g)}$ isomorphic to $C(g)$, the centralizer of $g$ in $G_x$.  In the rigidified inertia stack, these isotropy groups are replaced by $C(g)/\langle g\rangle$.

In our case, this will be the group $R/\langle \redge(e) \rangle$.   Thus, for each node over $0$, gluing over the rigidified inertia stack means that we must multiply the virtual fundamental class by a factor of $|R|/|\redge(e)|$.

A more geometric explanation of this factor is as follows.  Let $\widetilde{C_i}$ to be an orbifold chart of $C_i$ in a neighborhood of the orbifold node in question.  Then the fibers of $\widetilde{C_1}, \widetilde{C_2}$ over the node are each isomorphic as $R$-sets to $R/\redge(e)$.  Gluing the map into a map of nodal curves is equivalent to giving a $R$-equivariant isomorphism of these two fibers, and there are clearly $|R|/|\redge(e)|$ distinct such isomorphisms.

Taking all of these factors into account, on the level of virtual fundamental classes, we have:

\begin{equation} \label{autandgluing}
\left[\mbar_\Gamma\right]=
\frac{1}{\Aut(\Gamma)}
\prod_{e\in E(\Gamma)}\frac{1}{|K|d(e)}\frac{|R|}{|\redge(e)|}\frac{|S|}{|\sedge(e)|}
\prod_{v\in v(\Gamma)}\left[\overline{\mathcal{M}}_v\right]
\end{equation}
Some caution is in order when dealing with these fixed point loci - some of the vertices of the graph might not actually correspond to a collapsed component.  We will conventionally act as if this does not happen, and ``destabilize'' our curves by requiring that every edge actually be adjacent to two vertices, and that every vertex represents a contracted curve.  This will take us out of the context of stable curves, but our conventions for dealing with these unstable contributions will give the same answer as if we had dealt with the stable curve, as verified in Section \ref{unstable}.  Furthermore, by considering the destabilized curves our formulas will become much more uniform.

An example illustrating essentially all the possibilities is a genus 0 degree 3 stable map to $\mathcal{C}_{2,3}$, consisting of two $\proj^1$ components, one mapping with degree 2 to the target, joined by a node mapping to $\infty$ to the other component, which has degree 1.  For this to happen, we see that our two components must be joined with a node with $\Z_3$ isotropy, and the point mapping to $0$ with degree 1 must have $\Z_2$ isotropy.  The destabilization will consist of a chain of 5 $\proj^1$ components, joined with nodes.  The first and last components will be contracted to 0, and the middle component will be contracted to $\infty$.  Thus, we would consider the fixed point set of this graph to be $\overline{\mathcal{M}}_{0,2}(\BZ_2,1/2,1/2)\times\overline{\mathcal{M}}_{0,1}(\BZ_2,0)\times\overline{\mathcal{M}}_{0,2}(\BZ_3,1/3,2/3)$.  We will evaluate tautological classes on these unstable moduli spaces using equation (\ref{unstableintegrals}), and in Section \ref{unstable} check that these conventions give the correct contributions.

\section{The virtual normal bundle}
To use Atiyah-Bott localization, we need to compute the equivariant Euler class of the virtual normal bundle of each fixed point component.  If we have a point $f\in \mbar_\Gamma\subset \overline{\mathcal{M}}$ in some fixed point locus $\mbar_\Gamma$, then the splitting of $T_f\overline{\mathcal{M}}$ into $T_f\mbar_\Gamma$, the tangent space of the fixed locus, and $N_f\mbar_\Gamma$, the normal bundle of the fixed locus, can be accomplished by looking at the $\C^*$ action:  $T_f\mbar_\Gamma$ will be the 0-eigenspaces of the action, and the nonzero, or \emph{moving}, eigenspaces will make up the normal directions of $\mbar_\Gamma$ at $f$.

Intuitively, if we think of $f$ as a map $C\to\mathcal{X}$, then $f^*(T\mathcal{X})$ should describe the ways of deforming $f$.  But since we identify isomorphic maps, we should quotient out by $TC$, which corresponds to reparametrizations of the source curve, and the result should in some sense be the tangent space of stable maps.  Using obstruction theory and derived categories, this can all be made precise, and we get an exact sequence:
\begin{align}
0\to & H^0(C, TC)\to H^0(C,f^*(T\mathcal{X}))\to \mathcal{T}^1\to \notag \\
\to & H^1(C, TC)\to H^1(C,f^*(T\mathcal{X}))\to \mathcal{T}^2\to 0\notag
\end{align}
where $\mathcal{T}^1-\mathcal{T}^2$ is the tangent space to $\mbar$ in $K$-theory.  So, using superscript $m$ to denote the moving part, we find that the reciprocal of the Euler class of the normal bundle should be:
\begin{equation} \label{eulernormal}
\frac{1}{e(N)}=\frac{e(H^0(C, TC)^m)}{e(H^1(C, TC)^m)}\frac{e(H^1(C,f^*(T\mathcal{X}))^m)}{e(H^0(C,f^*(T\mathcal{X}))^m)}
\end{equation}
In the remainder of this section we compute each equivariant Euler class in turn, writing them in terms of the contribution by each vertex.  We write only the contributions of vertices over $0$;  the contributions coming of vertices over $\infty$ have the same form with $t$ replaced by $-t$, $r$ by $s$, and $R$ by $S$.

The term $H^0(C, TC)$ parameterizes infinitesimal automorphisms of the source curve.  Conventionally, all vertex components are stable, and hence have no infinitesimal automorphisms.  Each edge curve is a topological $\proj^1$ with two nodes that need to be fixed by the automorphisms, and so we see that each $H^0(C_e,TC_e)$ should be one dimensional.  It is spanned by any section $s$ of $TC_e$ vanishing at both $0$ and $\infty$.  By Chern-Weil theory, we see that in this case $s$ must vanish simply at each of $0$ and $\infty$.  Choose a connection $\nabla$ on $T\proj^1$, then locally the derivative $\nabla s$ would be a section of $TC_e\otimes T^*C_e$.  Furthermore, since $s$ vanished simply, $\nabla s$ would not vanish on the fiber over $0$, and so we may identify $H^0(C_e,TC_e)$ with $TC_e\otimes T^*C_e$.   Since the $\C^*$ actions on $T\proj^1$ and $T^*\proj^1$ have opposite weights and group actions, $ T_0C_e\otimes T_0^*C_e$ would have a weight 0 $\C^*$ action and a trivial action of the isotropy group. As this vector space can be identified with $H^0(C_e,TC_e)$,  we see that this contributes only to the tangent bundle, and not the normal bundle, and so $e(H^0(C_e,TC_e)^m)=1$.

The term $H^1(C,TC)$ parameterizes infinitesimal smoothings of the nodes in the source curve.  By our graph conventions there is a node for every flag, and these are the only nodes that contribute to the normal bundle.  The node $n$ between $C_e$ and $C_v$  contributes $T_nC_e\otimes  T_nC_v$.    Note that even if we have a twisted node, this space has trivial group action, since all nodes are balanced.

As $T_nC_v$ is on the contracted component, it will have trivial $\C^*$ action.  This is dual to the cotangent space of the contracted curve at that special point, and the underlying curve may have an orbifold point there.  So the Euler class of this line bundle would be the $\psi$-class, or using our $\psibar$ classes pulled back from $\mbar_{g,n}, e(T_nC_e)=-\frac{1}{|\redge(e)|}\psibar$.

On the other hand, $T_nC_e$ will be topologically trivial, but will have a nontrivial $\C^*$ action.  The weight of this $\C^*$ action picks up a factor of $1/d$ to ensure the map is equivariant, and a factor of $1/|\redge(e)|$ from the orbifold structure, and so the total $\C^*$ weight is $1/(d|\redge(e)|)$.  So in total, a node $n$ attached to an edge of degree $d$, with isotropy mapping to $\redge(e)$ at $0$ will contribute to $e(H^1(C_e,TC_e))$ by:
\begin{equation*}
\frac{t}{|\redge(e)|d}-\frac{1}{|\redge(e)|}\psibar=\frac{t}{|\redge(e)|d}(1-d\psibar/t)
\end{equation*}
Thus, the total contribution of all node smothing terms at a vetex $v_0$ to $\frac{1}{e(N_\Gamma)}$ is:
\begin{equation}\label{nodevertex}
 t^{-e(v)}\prod_{i=1}^{e(v_0)}\left(|\redge(e_i)|\frac{d_i}{1-d_i\psibar_i/t}\right)
 \end{equation}

\section{Normalization exact sequence}
We calculate $H^0(C,f^*(T\mathcal{X}))$ and $H^1(C,f^*(T\mathcal{X}))$ together using the normalization long exact sequence, coming from resolving the nodes forced by the graph.  For convenience, we will introduce the notation $\xi=f^*(T\mathcal{X})$.  Tensoring the short exact sequence
$$
0\to\mathcal{O}_C\to\bigoplus_{e\in\mathrm{E}(\Gamma)}\mathcal{O}_{C_e}
\bigoplus_{v\in \mathrm{V}(\Gamma)}\mathcal{O}_{C_v}\to
\bigoplus_{F\in\mathrm{F}(\Gamma)}\mathcal{O}_{F}\to 0
$$
by $\xi$ and taking the long exact sequence in cohomology, we have:
\begin{eqnarray*}
0 & \to & H^0(C, \xi)\to\bigoplus_{e\in\mathrm{E}(\Gamma)}H^0(C_e,\xi)\bigoplus_{v\in \mathrm{V}(\Gamma)}H^0(C_v,\xi)\to\bigoplus_{F\in\mathrm{F}(\Gamma)}H^0(C_f,\xi)\to \\
 & \to &  H^1(C, \xi)\to\bigoplus_{e\in\mathrm{E}(\Gamma)}H^1(C_e,\xi)\bigoplus_{v\in \mathrm{V}(\Gamma)}H^1(C_v,\xi)\to\bigoplus_{F\in\mathrm{F}(\Gamma)}H^1(C_f,\xi)\to 0.\notag
\end{eqnarray*}
We now exam the terms of this sequence in detail.

\subsubsection{Flags}
As the $C_F$ are not curves but nodes, they are zero dimensional and so $H^1(C_F,\xi)=0$.  To calculate $H^0(C_F,\xi)$, we need to understand the action of the isotropy group on $\xi_0$.   The group action on this vector space is pulled back from the standard representation of $\Z_r$, so if the image of $\redge(F)$ in $\Z_r$ is nonzero, this is a nontrivial representation, and so contributes 0.  However, if $\redge(F)\in K$, the representation will be trivial.  Since $T_0$ has $\C^*$ weight $1/r$, it will contribute $t/r$ to the Euler class.

The image of $\redge(F)$ in $\Z_r$ is completely determined by the edge degree modulo $r$.  Thus, the flag contribution from each vertex $v_0$ will be
\begin{equation}\label{flagvertex}
\left(\frac{t}{r}\right)^{(\# d_i=0\mod r)}.
\end{equation}

 \subsubsection{Edges}
 We compute the contribution of $H^i(C_e, \xi)$ by using the isomorphism with the cohomology of the desingularization $H^i(|C_e|, |\xi|)$.  If $C_e$ is an edge of degree $d$, then since $cw_1(\mathcal{X})$ has degree $1/r+1/s$, $\xi=f^*T\mathcal{X}$ will have degree $d(1/r+1/s)$.  The curve $C_e$ has isotropy $\Z_{|\redge(e)|}$ at $0$, and the generator acts on  $f^*(T_0\mathcal{X})$ by its image in $\Z_r$, which is $d\mod r$.  So the generator acts by $1/|\redge(e)|$ on the tangent bundle and $\leftover{\frac{d}{r}}$ on $f^*T\mathcal{X}$.  Recalling the discussion of the desingularization in (\ref{linebundles}), we see that the $\C^*$ weight of $|f^*T\mathcal{X}|$ will be that of $f^*T\mathcal{X}\otimes (T^*C_e)^a$, where $a=|\redge(e)|\leftover{\frac{d}{r}}$.  Thus, the $\C^*$ weight of $|f^*T\mathcal{X}|$ at $0$ is:
\begin{equation*}
 \frac{1}{r}-|\redge(e)|\leftover{\frac{d}{r}}\frac{1}{|\redge(e)|d}=\frac{d}{dr}-\frac{d\mod{r}}{dr}=\frac{1}{d}\floor{\frac{d}{r}}
 \end{equation*}
Similarly, we see that the degree of $|f^*T\mathcal{X}|$ will be
\begin{equation*}
\frac{d}{r}+\frac{d}{s}-\frac{d\mod r}{r}-\frac{d\mod s}{s}=\floor{\frac{d}{r}}+\floor{\frac{d}{s}}
\end{equation*}

 As this is nonnegative, $H^1(C_e, f^*T\mathcal{X})$ will be zero, while $H^0(C_e, f^*T\mathcal{X})$ will be $\floor{\frac{d}{r}}+\floor{\frac{d}{s}}+1$ dimensional.   Any eigensection of the desingularization  $|f^*T\mathcal{X}|$ can vanish only at $0$ and $\infty$, and so our eigensections are given by sections that vanish to order $k$ at $0$ and order $\floor{\frac{d}{r}}+\floor{\frac{d}{s}}-k$ at $\infty$, for $0\leq k \leq \floor{\frac{d}{r}}+\floor{\frac{d}{s}}$.

 To determine the weight of a section vanishing $k$ times at $0$, note that the $k$th derivative would locally be a section of $|f^*T\mathcal{X}|\otimes \omega_{|C_e|}^k$ that is nonzero at $0$.   Now, the weight of $\omega_{|C_e|}$ at 0 is $-1/d$, and so we see that the eigensection vanishing to order $k$ at 0 has $\C^*$ weight $\leftover{\frac{d}{r}}\frac{1}{d}-\frac{k}{d}$. Hence, as $k$ varies there are sections of every weight that's a multiple of $1/d$ from $-\floor{\frac{d}{s}}/d$ to $\floor{\frac{d}{r}}/d$.  One section has zero weight, and so contributes to the virtual tangent bundle rather than the virtual normal bundle.  We split this edge contribution between the zero and infinity by associating the positive weighted sections with $0$ and the negatively weighted sections with $\infty$.  With this convention, the contribution coming from a degree $d$ edge attached to $0$ is:
 \begin{equation}\label{edgevertex}
\frac{d^\floor{\frac{d}{r}}t^{-\floor{\frac{d}{r}}}}{\floor{\frac{d}{r}}!}.
\end{equation}

\subsubsection{Vertices}

We now consider the terms $H^i(C_v,f^*T_0\mathcal{X})$.  For a vertex over $0$, we have $f:C_v\to\mathcal{B}R$, and so $f$ is equivalent to a principal $R$ bundle $\widetilde{C}_v$ over $C_v$.  The bundle $f^*T_0\mathcal{X}$ on $C_v$ corresponds to a topological trivial bundle on $\widetilde{C}_v$, but with a potentially nontrivial lift of the $R$ action.  The group $H^i(C_v,f^*T_0\mathcal{X})$ is isomorphic to the $R$ invariant part of $H^i(\widetilde{C}_v, \mathcal{O})\otimes T_0\mathcal{X}$.

The dimension of $H^0(\widetilde{C}_v,\mathcal{O})$ will be the number of components of the $R$-cover $\widetilde{C}$.  If the collection of all monodromy around loops generates some subgroup $H\subset R$, then $\widetilde{C}$ will have $|R|/|H|$ components, and as an $R$ representation $H^0(\widetilde{C}_v,\mathcal{O})$ will be the regular representation of $R/H$.  Thus, $H^0(C_v,f^*T_0\mathcal{X})$ will be one dimensional if the $R$ action on $T_0\mathcal{X}$ factors through $R/H$, and zero otherwise, or equivalently, it will be one dimensional if $H\subset K$, and zero otherwise.
  Since $T_0\mathcal{X}$ has $\C^*$ weight $1/r$, we can use the notation from \ref{deltaK} and write
\begin{equation}  \label{H0v}
e(H^0(C_v,f^*T_0\mathcal{X})=\left(\frac{t}{r}\right)^{\delta_K}
\end{equation}

To calculate $H^1(C_v,f^*T_0\mathcal{X})$, we apply Serre duality to see that:
\begin{align}
(H^1(\widetilde{C}_v,\mathcal{O})\otimes T_0\mathcal{X})^R&=(H^0(\widetilde{C}_v,\omega)^\vee\otimes T_0)^R \notag \\
&=(\E^\vee)_{T^*_0} \notag\\
&=\E_{T_0}^\vee. \notag
\end{align}

In addition to the topological structure of the bundle, $T_0\mathcal{X}$  has a $\C^*$ action with weight $1/r$.  So the equivariant Euler class of this bundle is
\begin{equation} \label{vertexvertex}
\left(\frac{t}{r}\right)^m-\left(\frac{t}{r}\right)^{m-1}\lambda^{T_0}_1+\dots\pm\lambda^{T_0}_m
=\left(\frac{t}{r}\right)^m\sum_{i=0}^m\left(-\frac{r}{t}\right)^i\lambda^{T_0}_i
\end{equation}

Here we are using $m=\dim(\E_{T_0})=g-1+\iota(\rrr(v)+\redge(v))+\delta_K$.

The $\delta_K$ term here exactly cancels the contribution of $H^0(C_v,\xi)$, and so in future appearances we will cancel it.

\subsubsection{Total Contribution}

We combine (\ref{nodevertex}),(\ref{flagvertex}),(\ref{edgevertex}) and (\ref{vertexvertex}) to find the total contribution of a vertex lying over 0 to $\frac{1}{e(N_\Gamma)}$.  Additionally, we draw the $1/|\redge(e)|$ and $1/|\sedge(e)|$ factors from the gluing
factors appearing in (\ref{autandgluing}) to cancel the similar term appearing in (\ref{nodevertex}).
Combining those terms and simplifying using
 $$\redge(v)+\# d_i=0\mod r=e(v)-\sum_{i=1}^{e(v)}\leftover{\frac{d_i}{r}}$$
and
$$\leftover{\frac{d_i}{r}}+\floor{\frac{d_i}{r}}=\frac{d_i}{r}$$
we obtain
\begin{equation} \label{normalvertex}
\frac{t^{g-1+\iota(\rrr(v))-\sum d_i/r}}{r^{g-1+\iota(\rrr(v))-\sum\leftover{\frac{d_i}{r}}+e(v)}}
\prod_{i=1}^{e(v)}\left( \frac{d_i^\floor{\frac{d_i}{r}}}{\floor{\frac{d_i}{r}}!}\frac{d_i}{1-d_i\psibar_i/t}\right)
\sum (-r/t)^i\lambda^{T_0}_i
\end{equation}
as the total contribution of a vertex $v_0$.

The contribution from a vertex $v_\infty$ is completely analogous,with $t$ replaced by $-t$, and $r$ and $R$ replaced by $s$ and $S$.

\section{Global localization contributions}

We now apply the localization calculations to express the Gromov-Witten generating function $G^\bullet$ in terms of the Hurwitz-Hodge generating functions $H^\bullet$.   We've established the contribution to the virtually normal bundle from each vertex appearing in a localization graph.  We now investigate the effect of localization on the integrands appearing in $G^\bullet_{\rrr,\sss}$.

The integrand over the point $0$ is exactly
$$
\prod_{i=1}^{\ell(\rrr)}            \frac{z_i\ev_i^*(\zero_{r_i})}{1-z_i\psibar_i}.
$$

Let $\varphi:\mbar_\Gamma\to \mbar_{g,\rrr,\sss}(\mathcal{X}, d)$ be the inclusion.  We will have $\varphi^*(\psibar)=\psibar$; consider $\varphi^*(\zero_{r_i})$.  For $r_i$ not belonging to $K$, the corresponding component of the $\mathcal{IX}$ is zero dimensional, and we have $\varphi^*(r_i)=r_i$.  However, for $r_i\in K$, the component of the twisted sector will be one dimensional, and we will have $\varphi^*(r_i)=t r_i$.  Thus, localizing gives us a factor of
$$t^{\# (r_i\in K)},$$
and otherwise, considering both the integrand and the virtual normal bundle, the integral appearing will be
$$\int_{\mbar^\bullet_{g_0,\rrr+\redge(\overline{\mu})}(\B R)}
\prod_{i=1}^{\ell(\rrr)}            \frac{z_i}{1-z_i\psibar_i}
\prod_{j=1}^{\ell(\overline{\mu})}  \frac{\mu_i}{1-\mu_i\psibar/t}
\sum_{\ell=0}^\infty (-r/t)^\ell\lambda^{T_0}_\ell.
$$

After some rescalings, we can express this in terms of $H^{0,\bullet}_{g,0}$.  Namely,  replacing $\psibar$ with $t\psibar$ and $\lambda_i$ with $t^i\lambda_i$ multiplies the integral by $t$ to the dimension of $\mbar_{g_0,\rrr+\redge(\overline{\mu})}(\B R)$, which is $3g_0-3+\ell(\rrr)+\ell(\overline{\mu})$.  After this, the $z$ terms will appear as
$$\frac{z_i}{1-tz_i\psibar_i},$$
and so we must multiply the integrand by $t^{\ell(\rrr)}$.  Canceling part of this term with the factor of $t^{\# (r_i\in K)}$ appearing from localization, we see that the contribution can be written as
$$
t^{-(3g_0-3+\ell(\rrr)+\ell(\overline{\mu})+\#(r_i\notin K))}H^{0,\bullet}_{g_0, \rrr + \redge(\overline{\mu})}(tz_{\rrr}, \mu),
$$
with the analogous statement for the integrals appearing over $\infty$.

Combining this with the other factors appearing in (\ref{normalvertex}), as well as the factors of $|R|$ and $|S|$ appearing from the node gluing and automorphism in equation (\ref{autandgluing}), we can write the total vertex contribution over $0$ as
\begin{equation} \label{vertexfinal}
|K|^{\ell(\mu)}\frac{t^{2-2g_0+\iota(\rrr(v))-|\mu|/r-\#(r_i\notin K)-\ell(\overline{\mu})-\ell(\rrr)}}
{r^{g-1+\iota(\rrr(v))-\sum\leftover{\frac{\mu_i}{r}}}}
\left(\prod_{i=1}^{\ell(\overline{\mu})}\frac{\mu_i^{\floor{\frac{\mu_i}{r}}}}{\floor{\frac{\mu_i}{r}}!}\right)
H^{0,\bullet}_{g_0, \rrr + \redge(\overline{\mu})}(tz_{\rrr}, \mu)
\end{equation}
with similar contribution over $\infty$.  To obtain the global contribution, we must combine these with the remaining global gluing and automorphism factor of
$$\frac{1}{\Aut(\Gamma)}\prod_{e\in E(\Gamma)}\frac{1}{|K|d(e)}.$$
This differs from equation (\ref{autandgluing}) because we have canceled the factor of $|\sedge(e)||\redge(e)|$ in the previous section, as well as the contributions of $|R|$ and $|S|$ appearing just previously.  Additionally, working with automorphisms of the weighted partition $\overline{\mu}$ correctly accounts of the weight shift:
$$\frac{1}{\mathfrak{z}(\overline{\mu})},$$

where
$$\mathfrak{z}(\overline{\mu})=\Aut(\overline{\mu})\prod_{i=1}^{\ell(\mu)}|K|\mu_i$$
will appear again in section \ref{wreathfock}.

 Some global constraints will be useful.

 For each component of a graph, the corresponding genus is the sum of the genera of all the vertices, plus the number of loops in the graph, which can be calculated by $e-v+1$.  So the total genus of a connected graph is:
$$
g(\Gamma)=|E(\Gamma)|-|V(\Gamma)|+1+\sum_{v\in V(\Gamma)} g(v)=|E(\Gamma)|+1+\sum_{v\in V(\Gamma)}(g(v)-1).
$$
Working with our disconnected curves and partitions, it will be more convenient to use the euler characteristic, which is additive under disjoin union:
$$2g(\Gamma)-2=(2g_0-2)+(2g_\infty-2)+2\ell(\overline{\mu}).$$

We can now express the disconnected $n+m$ point function $G^\bullet_{d,\rrr,\sss}$ in terms of the functions $H$

Define
\begin{equation} \label{jaydef}
\jay_\rrr(z_\rrr,\overline{\mu},u,t)=
\frac{r^{\sum\leftover{\frac{\mu_i}{r}}-\iota(\rrr)}(|K|u/t)^{\ell(\mu)}}{t^{|\mu|/r+\#(r_i\notin K)+\ell(\rrr)-\iota(\rrr)}}
\left(\prod_{i=1}^{\ell(\overline{\mu})}\frac{\mu_i^\floor{\frac{\mu_i}{r}}}{\floor{\frac{\mu_i}{r}}!}\right) H^{0,\bullet}_{\rrr+\redge(\overline{\mu})}\left(\mu, tz_{\rrr}, \frac{u}{tr^{1/2}}\right)
\end{equation}

And for $\jay_\sss$ we replace $r$ with $s$ and $\redge$ with $\sedge$, but otherwise keep things the same.  Then we have:
\begin{equation} \label{globalformula}
G^\bullet_{d,\rrr,\sss}(z_{\rrr},w_{\sss},u)=
\sum_{|\overline{\mu}|=d}
\frac{1}{\mathfrak{z}(\overline{\mu})}\jay_\rrr(z_\rrr,\overline{\mu},u,t)
\jay_\sss(z_\sss,\overline{\mu},u,-t).
\end{equation}

\section{Unstable Contributions} \label{unstable}
We check here that the unstable localization contributions we have defined agree with the localization procedure.   There are two unstable moduli spaces to consider, $\mbar_{0,1}$ and $\mbar_{0,2}$.  The first arises from vertices with one edge and no marked points; the second from vertices with either one marked point and one edge, or no marked points and two edges.  Each case is checked, first presenting the result of our localization scheme, followed by the actual contribution.  We work with a vertex over $0$; the usual adaptations cover vertices over $\infty$.

\subsubsection{One edge, no marked points}

By monodromy considerations the edge must have degree divisible by $r$ and the node must have trivial monodromy.  Using this, and the fact that $\rrr(v)$ is empty, we see that the contribution of this vertex to (\ref{vertexfinal}) is
$$
|K|\frac{t^{1-d/r}}{r^{-1}} \frac{d^{d/r}}{(d/r)!} H^{0,\circ}_{0,0}(d).
$$

Our convention (\ref{unstableintegrals}) for unstable hodge integrals evaluations $H$ to be $\frac{1}{|R|d}$.  Thus, factoring out the
 $t^{-d/r}\frac{d^{d/r}}{(d/r)!}$ as edge contribution, we see that the total vertex contribution by our localization scheme should be simply $\frac{t}{d}$.

On the actual curve, we have an edge $C_e$ of degree $d$ with no marked points over $0$.  This has the usual edge contribution, but there is an additional factor, as this map has infinitesimal automorphisms contributing to the $H^0(C, TC)$ term of the equation (\ref{eulernormal}) for the inverse euler class of the virtual normal bundle.  These automorphisms exactly correspond to the $T_0C_e$, which has $\C^*$ weight precisely $\frac{t}{d}$.

\subsubsection{One edge, one marked point}

Denoting by $r_1\in R$ the monodromy of the marked point, and $d$ the edge degree.  Then the contribution to equation (\ref{vertexfinal}) is
$$
|K|\frac{t^{1+\iota(r_1)-\delta_{r_1\notin K}-d/r}}{r^{-1+\iota{r}-\leftover{\frac{d}{r}}}}
\frac{d^{d/r}}{(d/r)!}
H^{0,\circ}_{0,\{r_1, -r_1\}}(tz_1, d).
$$
Factor out the edge contribution of $(d/t)^{\floor{\frac{d}{r}}}/(\floor{\frac{d}{r}})!$ leaves a factor of $t^{-\leftover{\frac{d}{r}}}$.  Simplifying using $\iota(r_1)=\leftover{\frac{d}{r}}$ and evaluating the Hurwitz Hodge term according to (\ref{unstableintegrals}), we have that the total contribution here is:
$$\frac{t^{\delta_{r_1\in K}}z_1d}{tz_1+d}.$$
For the actual curve, we have an edge of degree $d$, with a marked point with monodromy $r_1$ over $0$.  There are no node gluing or automorphism terms here, simply the contribution from the integrand in the definition of $G$, namely
$$\frac{z_1\ev_1^*{\zero_{r_1}}}{1-z_1\psibar}.$$
We saw in the previous section that $\ev_1^*(\zero_{r_1})$ localizes to $t^{\delta_{r_1\in K}}$.  In this case, since the marked point constrained to map to $0$, we see that $\psibar_1$ localizes to $-t/d$.  Substituting these in and simplifying, again the actual localization contribution agrees with the contribution of our scheme.

\subsubsection{Two edges, no marked points}

In our scheme, letting $d_1$ and $d_2$ be the two sides of the node, with monodromies $\rho_1$ and $\rho_2$ we have $\rho_1=-\rho_2$ by monodromy considerations.  The contribution to equation (\ref{vertexfinal}) of this vertex is is:
$$|K|^2\frac{t^{-d_1/r-d_2/r}}{r^{-1-\leftover{\frac{d_1}{r}}-\leftover{\frac{d_2}{r}}}}
\frac{d_1^{\floor{\frac{d_1}{r}}}}{\floor{\frac{d_1}{r}}!}
\frac{d_2^{\floor{\frac{d_2}{r}}}}{\floor{\frac{d_2}{r}}!}
H^{0,\circ}_{0,\{\rho_1,\rho_2\}}(d_1, d_2).
$$
Factoring out the usual edge terms and evaluating the unstable Hodge integral according to convention, we see that this simplifies to $$|K|(t/r)^{-\leftover{\frac{d_1}{r}}-\leftover{\frac{d_2}{r}}}\frac{d_1d_2}{d_1+d_2}.$$

On the actual curve, we have edges of degrees $d_1$ and $d_2$ meeting directly in a node.  Although in the destabilized curve there would be two nodes, in fact there is only one, and so the correct contribution from the automoprhism and gluing term from (\ref{autandgluing}) is $\frac{|R|}{|\rho_1|}$; note that we have $|\rho_1|=|\rho_2|$.

The node smoothing term contributing to $1/e(H^1(C, TC)^m)$ in (\ref{eulernormal}) contributes the Euler class of the tensor product of the tangent spaces at either side of the node, which is
$$\frac{1}{t/(d_1|\rho_1|)+t/(d_2|\rho_2|)}=\frac{t^{-1}|\rho_1|d_1d_2}{d_1+d_2}.$$

Finally, although there are no contracted vertices, there is a flag term appearing in the normalization long exact sequence.  The isotropy group will act nontrivially on the flag term if $\leftover{\frac{d_1}{r}}\neq 0$, in which case the contribution is zero; otherwise, the contribution is $t/r$.  Using the fact that the node is balanced, we can write this as $(t/r)^{1-\leftover{\frac{d_1}{r}}-\leftover{\frac{d_2}{r}}}$.

Combining these three contributions gives
$$
\frac{|R|}{|\rho_1|}\frac{t^{-1}|\rho_1|d_1d_2}{d_1+d_2}(t/r)^{1-\leftover{\frac{d_1}{r}}-\leftover{\frac{d_2}{r}}}
=|K|(t/r)^{-\leftover{\frac{d_1}{r}}-\leftover{\frac{d_2}{r}}}\frac{d_1d_2}{d_1+d_2},
$$
which agrees with the contribution given by our localization scheme.

\chapter{Wreath Products and Fock Spaces}
\label{wreathfock}

In the previous chapter, virtual Atiyah-Bott localization reduced the Gromov-Witten invariants we are interested in to the calculation of Hurwitz-Hodge intergrals.  The orbifold $ELSV$ formula, which we will introduce at the beginning of the next chapter, expresses these integrals in terms of wreath Hurwitz numbers.  In this chapter, we pause to introduce Hurwitz numbers and wreath Hurwitz numbers.  These have expressions in terms of the representation theory of the symmetric group and of wreath products.  These representation theories, in turn, are conveniently expressed have Fock space formalisms: in the case of the symmetric group, this is the infinite wedge; in the case of wreath products, it is essential a tensor product of copies of the infinite wedge.  Finally, the Kyoto school of integrable hierarchies relates operators on these Fock spaces to integrable hierarchies.  This chapter reviews these elements and their connections, first deriving an expression for wreath Hurwitz numbers in terms of operators on these Fock spaces, and then using this connection to show that in fact wreath Hurwitz numbers satisfy multiple commuting copies of the 2-Toda hierarchy.
We begin in section \ref{symmetric group} with a review of double Hurwitz numbers and their connection with the center of the group algebra of the symmetric group.  Section \ref{wreath product} extends this familiar story to wreath Hurwitz numbers.   In the following section, we review the representation theory of the wreath product, and use it obtain an expression for the wreath Hurwitz numbers.  The content of these three sections is classical.  More modern material begins in section \ref{fock space}, which introduces the Fock space formalism.  It begins with the infinite wedge, and then presents the corresponding Fock spaces for wreath products, a formalism that has been developed and applied and Wang and collaborators.  This section concludes with an operator expression for wreath Hurwitz numbers, which will be applied in the next chapter.  The last section \ref{okounkov detour}, is a short detour illustrating the connection to integrable systems: it shows that wreath Hurwitz numbers satisfy multiple commuting copies of the 2-Toda hierarchy.  This result is an easy generalization of \cite{OHur}, and should be viewed as a gentle inroduction to the 2-Toda hierarchy.

\section{Hurwitz Numbers and the Symmetric Group} \label{symmetric group}

The double Hurwitz number $\Hur^\circ_{g,d}(\mu,\nu)$ counts the number of maps
$$f:\Sigma \to \proj^1$$
from smooth connected curves $\Sigma$, where $f$ has ramification profile $\mu$ over $0$, $\nu$ over $\infty$, and simple ramification over $$b=2g-2+\ell(\mu)+\ell(\nu)$$
fixed other points.  The number $b$ is determined by the Riemann-Hurwitz formula to ensure that $\Sigma$ will have genus $g$.  The number of such covers does not depend on the location of the $b$ points in the base; for convenience, we will fix them to occur at the points of $U_b$, the set of $b$th roots of unity.  In the case that $\nu=(1)^d$ corresponds to no ramification, we call the result a single Hurwitz number, and denote it $\Hur^\circ_g(\mu)$.

In addition to this geometric description, Hurwitz numbers have a simple expression in terms of multiplication in $\mathcal{Z}S_d$, the center of the group algebra of the symmetric group.  Let $f:\Sigma\to\proj^1$ be a cover counted by $\Hur^\circ_{g,d}(\mu,\nu)$.  Away from the preimages of $0,\infty$ and $U_b$, the map $f$ is a topological covering space.  Choose a basepoint $b_0\in\proj^1$ and loops $\Gamma_0,\Gamma_\infty, \Gamma_1,\dots,\Gamma_b$, based at $b_0$, around $0, \infty$ and the $b$ roots of unity respectively, so that  $$\Gamma_0\Gamma_1\cdots\Gamma_b\Gamma_\infty=1\in\pi_1(\proj^1\setminus\{0,\infty,U_r\},b_0).$$
Labeling the $d$ preimages of $b_0$ in $\Sigma$, we see that parallel transport of the preimages of $f$ around $\Gamma_i$ gives an element $\sigma_i\in S_d$.  The ramification conditions imposed on $f$ imply that $\sigma_1,\dots, \sigma_b$ are all transpositions, while $\sigma_0$ has cycle type $\mu$ and $\sigma_\infty$ has cycle type $\nu$.  Since $\Sigma$ is connected, the $\sigma_i$ must act transitively on $1,\dots, d$.  Finally, since the loop $\Gamma_0\Gamma_1\cdots\Gamma_b\Gamma_\infty$ is contractible, we must have $$\sigma_0\sigma_1\cdots\sigma_b\sigma_\infty=1.$$  So, from the map $f$ and our choice of labeling, we constructed elements $\sigma_i\in S_d$ satisfying:

\begin{enumerate}
\item[(i)] $\sigma_0$ and $\sigma_\infty$ have cycle types $\mu$ and $\nu$, respectively.
\item[(ii)] The elements $\sigma_1,\dots, \sigma_b$ are transpositions.
\item[(iii)] $\sigma_0\sigma_1\cdots\sigma_b\sigma_\infty=1$
\item[(iv)] The group generated by all the $\sigma_i$ acts transitively on $\{1,\dots, d\}$
\end{enumerate}

This process is reversible: given elements $\sigma_i$ satisfying properties (i)-(iv), by the Riemann existence theorem we may construct a Hurwitz cover $\Sigma$, together with a labeling of the sheets.  As property (iv) is what guarantees the cover is connected, elements satisfying properties (i)-(iii) correspond to Hurwitz covers where $\Sigma$ is possibly disconnected.

Recall for any group $G$, the group algebra $\C[G]$ is an algebra with basis $[g]\in G$ and multiplication $[g]*[g^\prime]=[gg^\prime]$; it can also be viewed as $\C$ valued functions on $G$ with product given by convolution. The center of the group algebra $\mathcal{Z}\C[G]$ is sometimes called the class algebra, because it consists of those functions that take on the same value for elements in the same conjugacy class.  Thus, for $c\in G_*$ a conjugacy class, the elements $C_c$ defined by
$$C_c=\sum_{g\in c} [g]\in \mathcal{Z}\C[G]$$
form a basis of $\mathcal{Z}\C[G]$.  For $z\in\mathcal{Z}\C[G]$, we will use the expression $[1]z$ to denote the coefficient of the identity in $z$, or in the function point of view, the value of $z$ on $1$.

For $S_d$, conjugacy classes correspond, via the cycle type, to partitions $\mu$; we denote the resulting element of $\mathcal{Z}S_d$ by $C_\mu$, and denote by $T$ the sum of all transpositions (corresponding to the partition $(21^{d-2})$).  From the above discussion it is immediate that:

\begin{equation*} \Hur^\bullet_{g,d}(\mu, \nu)=\frac{1}{d!}[1] C_\mu \cdot C_\nu\cdot T^b.
\end{equation*}

\section{Wreath Products} \label{wreath product}
The wreath product $G_d=G\wr S_d$ is defined by
$$
G_d=\{(g,\sigma)|g=(g_1,\dots, g_d)\in G^d, \sigma\in S_d\},
$$
$$
(g,\sigma)(g^\prime, \sigma^\prime)=(g\sigma(g^\prime),\sigma \sigma^\prime).
$$

Conjugacy classes of $G_d$ are determined by their cycle type \cite{M}: for each $m$-cycle $(i_1i_2\cdots i_m)$ of $\sigma$, the element $\prod_{j=1}^m g_{i_j}$ is well defined up to conjugacy in $G$.  We will denote the set of conjugacy classes of $G$ by $G_*$, and use $c$ to denote a conjugacy class.  The {\em cycle type} of an element $(G,\sigma)$ is the $G_*$-labeled partition $\overline{\mu}$ where the underlying partition $\mu$ is the usual cycle type of the permutation $\sigma$, and the part $\mu_i$ corresponding to the cycle $(i_1i_2\cdots i_m)$ is labeled with the conjugacy class $c^\mu_i=(\prod_{j=1}^m g_{i_j})$. Two elements of $G_d$ are conjugate exactly when they have the same cycle types, and so cycle types label the conjugacy classes of $G_d$.

Given a $G_*$-labeled partition $\overline{\mu}$, we can form $|G_*|$ separate partitions $\overline{\mu}^c$, for $c\in G_*$, by taking only those parts of $\overline{\mu}$ labeled by $c$.  We denote
$$\ell(\overline{\mu})=\ell(\mu)=\sum_{c\in K_*} \ell(\overline{\mu}^c).$$

Let $\zeta_c$ denote the size of the centralizer of an element in the conjugacy class $c$, and $\mathfrak{z}({\overline{\mu}})$ denote the size of the ecntralizer of an element in the conjugacy class $\overline{\mu}$. Then we have
$$\mathfrak{z}(\overline{\mu})=|\aut(\overline{\mu})|
\prod_{i=1}^{\ell(\overline{\mu})} \zeta_c\overline{\mu}_i. $$

For a cycle type $\overline{\mu}$ we denote the corresponding element in $\mathcal{Z} G_d$, the center of the group algebra of $G_d$, as $C_{\overline{\mu}}$.  For $c\in G_*$, we will denote by $T_c\in\mathcal{Z} G_d$ the element corresponding to the conjugacy class $(2_c)=\{(2,c), (1,\text{Id}),\dots (1,\text{Id})\}$.  Of particular interest will be the class $T_0$, corresponding to the case where $c=\text{Id}$.

There are several plausible ways to define wreath Hurwitz numbers; we give here the one naturally occurring in relation to abelian Hurwitz-Hodge integrals.  We define the $G_d$ Hurwitz numbers $\Hur^\bullet_{g,G}(\overline{\mu},\overline{\nu})$ to be the count of degree $d|G|$ covers $f:\Sigma\to\proj^1$, with monodromy in the group $G_d$, with prescribed monodromy:  the monodromy over $0$ and $\infty$ must be $\overline{\mu}$ and $\overline{\nu}$, respectively, the monodromy over each of the $b=2g-2+\ell(\overline{\mu})+\ell(\overline{\nu})$ roots of unity must be $\tau_0$, and the map must be unramified elsewhere.

In the cases we will consider, $G$ will be abelian, and so the diagonal copy of $G\in G_d$ will commute with the natural copy of the symmetric group $S_d\subset G_d$.  Thus, each cover counted by $\Hur^\bullet_{g,G}(\overline{\mu},\overline{\nu})$ will have a $G$ action.  The quotient space $\Sigma/G$ will be a usual Hurwitz cover counted by $H^\bullet_g(\mu,\eta)$, and away from $f^{-1}(0), f^{-1}(\infty)$, $\Sigma\to\Sigma/G$ will be a principal $G$ bundle.  From the definition of the cycle type, if $p_i\in f^{-1}(0)$ corresponds to part $\overline{\mu}_i$, then the monodromy of the principal bundle around $p_i$ will be $c_i^\mu$.  This process is reversible: given a degree $d$ cover counted by $H^\bullet_g(\mu, \eta)$, and a principal $G$ bundle as above, we can construct a $G_d$ Hurwitz cover.

The connectivity requirement we put on $\Hur^\circ_{g,G}(\overline{\mu},\overline{\nu})$ is not that the total cover is connected, but only that the quotient of this cover by $G$ is connected.  Thus, $\Hur^\circ_{g,G}(\overline{\mu},\overline{\nu})$ may be seen as counting the covers counted in the usual double Hurwitz problem, but each cover is weighted by the number of principal $G$ bundles over it with monodromies $c^\mu$ over $f^{-1}(0)$ and monodromies $c^\nu$ over $f^{-1}(\infty)$.

By the same logic as the previous section, we have that:

\begin{equation} \label{Kdhurwreath}
\Hur^\bullet_{g,K}(\overline{\mu}, \overline{\nu})=\frac{1}{|K_d|}[1] C_{\overline{\mu}} \cdot C_{\overline{\nu}}\cdot T_0^b.
\end{equation}

\section{Representation Theory} \label{representation theory}

We will use $G^*$ to denote the set of irreducible characters of $G$, and use $\gamma$ to denote an element of $G^*$.  Just as the conjugacy classes of $G_d$ are indexed by $G_*$-labeled partitions of $d$, irreducible characters of $G_d$ are indexed by $G^*$-labeled partitions of $d$.  We will use $\overline{\lambda}$ to denote such a labeled partition, where the part $\lambda_i$ is labeled by the representation $\gamma^\lambda_i$.  The character indexed by $\overline{\lambda}$ may be described as follows.  Given an irreducible character $\gamma$ of $G$ induced by the representation $V_\gamma$, the wreath product $G_d$ acts naturally on $V_\gamma^{\otimes d}$, with $S_d$ permuting the factors and $G^d$ acting factor by factor.  If $U_\lambda$ is the irreducible representation of $S_d$ indexed by $\lambda$, then, $G_d$ acts on $U_\lambda$ via the forgetful map $G_d\to S_d$.  It turns out that the action of $G_d$ on $U_\lambda\otimes V_\gamma^{\otimes d}$ is irreducible.

More generally, given a $G^*$ labeled partition $\overline{\lambda}$, we can form the $|G^*|$ partitions $\overline{\lambda}^\gamma$, where $\gamma$ is an irreducible character and $\overline{\lambda}^\gamma$ consists of those parts of $\overline{\lambda}$ labeled by $\gamma$.  Then
$$\bigotimes_{\gamma\in G^*} U_{\overline{\lambda}^\gamma}\otimes V_\gamma^{\otimes |\overline{\lambda}^\gamma|}$$
is an irreducible representation of the subgroup
$$\prod_{\gamma\in G^*} G_{|\overline{\lambda}^\gamma|}$$
of $G_d$, and it induces up to an irreducible representation $U_{\overline{\lambda}}$ of $K_d$, which yields the irreducible character indeed by $\overline{\lambda}$.

As with any finite group, the center of a group algebra $\mathcal{Z}G_d$ has two natural bases: the conjugacy class basis $C_{\overline{\mu}}$, which we have used above, and the character basis, $R_{\overline{\lambda}}$, indexed by $\lambda$ the irreducible characters, which on an element $g$ takes the value $\chi_{\overline{\lambda}}(g)$.
$$ R_{\overline{\lambda}}=\sum_{\overline{\eta}\in \text{Conj}(K_d)}\chi^{\overline{\lambda}}_{\overline{\eta}} C_{\overline{\eta}}
$$

The basis $R_{\overline{\lambda}}$ has two nice properties.

First, $\mathcal{Z}G_d$ is a Frobenius algebra, and so has an invariant hermitian inner product given by the linear form
$$\left\langle C_{\overline{\mu}}\right\rangle=\frac{1}{|G_d|}\delta_{\overline{\mu}, \text{id}},$$
i.e., on basis elements we have
$$
\left\langle C_{\overline{\mu}},C_{\overline{\eta}}\right\rangle=
\left\langle C_{\overline{\mu}}\cdot C_{\overline{\eta}}\right\rangle
$$
and the form is extended to all of $\mathcal{Z}G_d$.  The basis $R_{\overline{\lambda}}$ is orthonormal with respect to this inner product.

Secondly, multiplication in the $R_\lambda$ basis is semisimple; we have
$$
R_{\overline{\lambda}}\cdot R_{\overline{\mu}}=\delta_{\overline{\lambda}, \overline{\mu}} \left(\frac{|G_d|}{\dim \overline{\lambda}}\right) R_{\overline{\lambda}}.
$$ 

Expressed in terms of the representation basis, we have
$$
C_{\overline{\mu}}=\sum_{\overline{\lambda}} \frac{1}{\mathfrak{z}(\overline{\mu})}\overline{\chi}^{\overline{\lambda}}_{\overline{\mu}}R_{\overline{\lambda}}.
$$
Note the $\overline{\chi}$ - we are taking the complex conjugate.  In the case when all characters are real - for example, with $S_d$ - this is not necessary.

Since formula (\ref{Kdhurwreath}) for the $G_d$ Hurwitz numbers involves calculating powers of the element $T_0$, it will be convenient to work with the representation basis, where multiplication is diagonal.   We choose to write the change of basis as follows: define the central character by
\begin{equation*}
 f_{\overline{\mu}}(\overline{\lambda})=\frac{|C_{\overline{\mu}}|\chi^{\overline{\lambda}}_{\overline{\mu}}}{\dim \overline{\lambda}}
 \end{equation*}
and $\overline{f}_{\overline{\mu}}(\overline{\lambda})$ is its complex conjugate.

Then 
\begin{eqnarray*}
C_{\overline{\mu}}&=&\sum_{\overline{\lambda}\vdash d} \frac{\overline{\chi}^{\overline{\lambda}}_{\overline{\mu}}}{\mathfrak{z}(\overline{\mu})}
R_{\overline{\lambda}} \\
&=&\sum_{\overline{\lambda}\vdash d}
\frac{\dim \overline{\lambda}}{|G_d|}
\left[\frac{|C_{\overline{\mu}}|}{\dim \overline{\lambda}}
\overline{\chi}^{\overline{\lambda}}_{\overline{\mu}}
\right]
R_{\overline{\lambda}} \\
&=&\sum_{\overline{\lambda}\vdash d}
\frac{\dim \overline{\lambda}}{|G_d|}
\overline{f}_{\overline{\mu}}(\overline{\lambda})
R_{\overline{\lambda}}.
\end{eqnarray*}
 
With this notation, we see that the number of covers of $\proj^1$ with monodromy $\overline{\mu}^i$ around point $p_i$ can be expressed as:
\begin{eqnarray*}
\frac{1}{|G_d|}[\text{id}]\prod_{i=1}^n C_{\overline{\mu}^i}
&=&\frac{1}{|G_d|}[\text{id}]\prod_{i=1}^n
\left(\sum_{\overline{\lambda}\vdash d} \frac{\dim \overline{\lambda}}{|G_d|}\overline{f}_{\overline{\mu}^i}(\overline{\lambda}) R_{\overline{\lambda}}\right) \\ \notag
&=&\frac{1}{|G_d|}[\text{id}]\sum_{\overline{\lambda}\vdash d} \frac{\dim \overline{\lambda}}{|G_d|}\prod_{i=1}^n \overline{f}_{\overline{\mu}^i}(\overline{\lambda})R_{\overline{\lambda}} \\ \notag
&=&\sum_{\overline{\lambda}\vdash d} \left(\frac{\dim \overline{\lambda}}{|G_d|}\right)^2\prod_{i=1}^n \overline{f}_{\overline{\mu}^i}(\overline{\lambda})
 \end{eqnarray*}
Since the left hand side is real, we may replace the right hand side with its complex conjugate, which just replaces every occurence of $\overline{f}$ with $f$.

 So in particular, since
 \begin{equation*}
 \frac{\dim\overline{\lambda}}{|G_d|}f_{\overline{\mu}}(\overline{\lambda})
 =\frac{|C_{\overline{\mu}}|}{|G_d|}\chi^{\overline{\lambda}}_{\overline{\mu}}
 =\frac{1}{\mathfrak{z}(\overline{\mu})}\chi^{\overline{\lambda}}_{\overline{\mu}}
 \end{equation*}
we have
 \begin{equation} \label{hurwitzrep}
 \Hur^\bullet_{g,G}(\overline{\mu}, \overline{\nu})
 =\frac{1}{\mathfrak{z}(\overline{\mu})}\frac{1}{\mathfrak{z}(\overline{\nu})}
 \sum_{\overline{\lambda}\vdash d}
 \chi^{\overline{\lambda}}_{\overline{\mu}}
 \chi^{\overline{\lambda}}_{\overline{\nu}}
 f_T(\overline{\lambda})^b.
 \end{equation}

 For this formula to be of much use, we must haveconvenient ways to calculate $\chi^{\overline{\lambda}}_{\overline{\nu}}$ and $f_T(\overline{\lambda})$.  This will be provided by the Fock space formalism, an extension of the infinite wedge used to study the symmetric group.

 \section{Fock Space} \label{fock space}
 It is natural to study the representation theory of $G_d$ for all $d$ at once.  In this section we will construct a graded vector space with inner product, $\mathcal{Z}_G$, whose degree $d$ piece will be isomorphic to $\mathcal{Z}\C[G_d]$ as a normed vector space.  The Fock space formalism identifies $\mathcal{Z}_G$ as an irreducible heighest weight representation of a certain Heisenberg algebra, which can also be viewed as the tensor product of copies of the charge zero part of the infinite wedge, $\infwedge_0V$.  The Heisenberg algebra structure provides a convenient description of two bases for $\mathcal{Z}_G$, $v_{\overline{\lambda}}$ and $P_{\overline{\mu}}$, so that the change of basis between them is the character table of $G_d$.

We begin with a brief introduction to the infinite wedge, mmostly following Okounkov.

Let $V$ is the vector space with basis labeled by the half-integers.  We use the underscore to represent the corresponding basis vector
\begin{equation*}
V=\bigoplus_{i\in\Z} \underline{i+\frac{1}{2}}.
\end{equation*}

The infinite wedge $\infwedge V$ is the span of vectors of the form $\underline{i_1}\wedge \underline{i_2}\wedge\dots$ with $i_k\in\Z+\frac{1}{2}$ a decreasing series of half integers so that $i_k+k+1/2$ is constant for $k>>0$.  Physicists call the infinite wedge space fermionic Fock space, and it can be thought of as a model for Dirac's sea of electrons, where all but finitely many negative energy states must be filled.

The fermionic creation and annihilation operators $\psi_j$ and $\psi^*_j, j\in\Z+\frac{1}{2}$ act on $\infwedge V$ as follows:
$$\psi_j(\underline{i_1}\wedge\underline{i_2}\wedge\dots)
=\underline{j}\wedge\underline{i_1}\wedge\underline{i_2}\wedge\dots$$
and
$$\psi^*_j(\underline{i_1}\wedge\underline{i_2}\wedge\dots)=
\delta_{j,i_1}\underline{i_2}\wedge\underline{i_3}\dots - \delta_{j,i_2}\underline{i_1}\wedge\underline{i_3}\wedge\dots
+\delta_{j,i_3}\underline{i_1}\wedge\underline{i_2}\wedge\dots -\cdots.$$
In words, the operator $\psi_j$ adds a $v_j$ to the wedge.  The operator $\psi^*_j$ removes a $v_j$, with the appropriate sign convention, if $v_j$ is present, and annihilates vectors with no $v_j$ present.  They satisfy the following anticommutation relations, where we use the notation $[x,y]_+$ for the anticommutator $xy+yx$:
\begin{equation*}
[ \psi_i,\psi_j^*]{}_+=\delta_{ij}
\end{equation*}
\begin{equation*}
[\psi_i,\psi_j]{}_+=[\psi_i^*,\psi_j^*]{}_+=0.
\end{equation*}

Observe that the operator $\psi_i\psi^*_j$ acts as the operator $E_{i,j}\in\gl(V)$ would on the infinite wedge.  However, a matrix $M$ with an infinite number of nonzero entries may not have a well defined action on the infinite wedge, because it would involve an infinite sum.  In particular, under the naive representation of $\gl(V)$, the identity matrix would have an inifnite sum for every element in the infinite wedge.  To remedy this situation, we normalize the representation by introducing the normal ordering convention:
\begin{equation*}
:\psi_i\psi_j^*:=\left\{\begin{array}{ll} \psi_i\psi_j^*, & j>0 \\
-\psi_j^*\psi_i,  & j<0\end{array}\right. .
\end{equation*}

We extend this to quadratic expressions in the $\psi$ and $\psi^*$ linearly.  Following this convention, we see that the the normal ordering of the operator that would correspond to the identity matrix, which we call the charge operator $C$, has a well defined action on the infinite wedge:
$$C=\sum_{i\in\Z+1/2} E_{i,i}= \sum_{i\in\Z+1/2}:\psi_i\psi_i^*:.$$

Indeed, this process extends: letting $E_{ij}$ for $i,j\in\Z+\frac{1}{2}$ be the standard basis of $\gl(\infty)$.  Then $$E_{ij}\mapsto:\psi_i\psi_j^*:$$ gives a projective representation of the lie algebra $\mathfrak{gl}(\infty)$ on $\infwedge V$.

For $0\neq k\in \Z$, we define the operators
$$\alpha_k=\sum_{i\in\Z+\frac{1}{2}} E_{i-k, i}.$$

The operators $\alpha_k$ form a Heisenberg algebra:
$$[\alpha_n,\alpha_m]=n\delta_{n,-m}.$$

Vectors of $\infwedge V$ that are eigenvectors of $C$ with eigenvalue $x$ are said to have charge $x$.  Similarly, if $L$ is an operator on $\infwedge V$ with $[C, L]=x$ then $L$ is said to have charge $x$.

The energy operator is defined by
$$H=\sum_{k\in\Z+\frac{1}{2}}ke_{kk}.$$

Similarly, vectors of $\infwedge V$ that are eigenvectors of $H$ with eigenvalue $h$ have energy $h$, and operators with $[H, L]=h$ are also said to have energy $h$.

The kernel of $C$ consists of those vectors of $\infwedge V$ with charge 0, and will be denoted $\infwedge_0V$; we will mostly be working in this subspace.  Note that charge $0$ operators preserve $\infwedge_0V$, and since the operators in $\gl_\infty$ have charge 0, $\infwedge_0V$ will be a representation of $\gl_\infty$.

The subspace $\infwedge_0V$ has a natural basis $v_\lambda$ labeled by partitions $\lambda$:
$$v_\lambda=\underline{\lambda_1-\frac{1}{2}}\wedge\underline{\lambda_2-\frac{3}{2}}\wedge\underline{\lambda_3-\frac{5}{2}}\wedge\dots$$

We give $\infwedge_0V$ an inner product by making the basis $\{v_\lambda\}$ orthonormal.

It is easily seen that the $v_\lambda$ form an eigenbasis of $\bigwedge_0V$ for $H$, with $Hv_\lambda=|\lambda|v_\lambda.$

The vacuum vector $v_\emptyset$ corresponds to the zero partition.  The vacuum expectation $\leftover{A}$ of an operator $A$ on $\infwedge V$ is defined by the inner product
$$\leftover{A}=(Av_\emptyset, v_\emptyset)$$

We define the vector space $\mathcal{Z}_G$ to be the tensor product of $|G^*|$ copies of $\infwedge_0 V$.  We give it the inner product coming from the tensor product; in other words, we introduce the basis $$v_{\overline{\lambda}}\in\mathcal{Z}_G=\bigotimes_{\gamma\in G^*}v_{\overline{\lambda}^\gamma}$$
and declare it to be orthonormal.

As $\mathcal{Z}_G\subset \bigotimes_{\gamma\in G^*} \infwedge V$, and we've defined a lot of interesting operators acting on the infinite wedge, we get a lot of interesting operators acting on $\mathcal{Z}_G$.  To describe them, we use:

\begin{convention} \label{tensorconvention}
For $M$ an operator on $\infwedge V$, and $\gamma\in G^*$,  we define $$M^{\gamma}=\text{Id}\otimes\cdots \text{Id}\otimes M \otimes \text{Id}\cdots \text{Id}$$
where the $M$ occurs on the component labeled $\gamma$.
\end{convention}

We will most often use convention \ref{tensorconvention} on operators $M$ that have charge $0$, in which the resulting operator $M^{\gamma}$ will act on $\mathcal{Z}_G$.

The operators $\alpha^\gamma_n$, then, satisfy the commutation relations

\begin{equation} \label{alphagammacommutator}
[\alpha^\gamma_n, \alpha^{\gamma^\prime}_m]=n\delta_{\gamma,\gamma^\prime}\delta_{n,-m}
\end{equation}

We have a linear map
$$
\text{ch}: \bigoplus_{d=0}^\infty \mathcal{Z}_{G_d}\to \mathcal{Z}_G
$$
defined by
$$\text{ch}(R_{\overline{\lambda}})=v_{\overline{\lambda}}.$$
We immediately see that $\text{ch}$ preserves the inner product.

We now give a useful description of the vectors $\text{ch}(C_{\overline{\mu}})$.

For $c\in G_*$, we define
 \begin{equation} \label{cintermsofgamma}
 \alpha^c_n=\alpha_n(c)=\sum_{\gamma\in G^*} \gamma(c^{-1})\alpha_n^\gamma
 \end{equation}

The second notation will be useful to avoid nested subscripts.  A warning is in order: we have expressed $\alpha^c_n$ as a linear combination of the $\alpha^\gamma_n$.  We also have, in $\mathcal{Z}\C G$, the expression for $C_c$ in terms of $R_\gamma$, i.e., the inverse character table of $G$.  Though simlar, these expressions do not agree; they differ by a factor of $\zeta_c$.

We see from the above that the $\alpha^c_n$ span the same space of operators as the $\alpha^\gamma_n$, and the relationship can be inverted:
\begin{equation} \label{gammaintermsofc}
\alpha^\gamma_n=\sum_{c\in G_*} \frac{1}{\zeta_c}\gamma(c)\alpha^c_n
\end{equation}

This observation and basic character theory tell us that
$$[\alpha^c_n, \alpha^{c^\prime}_m]=n\zeta_c\delta_{c,c^\prime}\zeta_{c^\prime}\delta_{n,-m}$$

Define the vector $P_{\overline{\mu}}\in\mathcal{Z}_G$ by
$$P_{\overline{\mu}}=\prod_{i=1}^{\ell(\mu)}\alpha^{c_i}_{-\mu_i}\vac$$
Then $P_{\overline{\mu}}$ form a basis of $\mathcal{Z}_G$, and in fact we have
$$\text{ch}(C_{\overline{\mu}})=\frac{1}{\mathfrak{z}(\overline{\mu})} P_{\overline{\mu}}$$

Since $\text{ch}$ is an isomorphism, we have that the change of basis between the $P_{\overline{\mu}}$ and the $v_{\overline{\lambda}}$ are given by a multiple of the character table of $G_n$; explicitly, we have
$$\left\langle P_{\overline{\mu}}, S_{\overline{\lambda}}\right\rangle=$$

In addition to the bases $C_{\overline{\mu}}$ and $P_{\overline{\lambda}}$, which come simply from the fact that the graded pieces of $\mathcal{Z}_G$ are the centers of group algebras, the heisenberg algebra structure provides us with another basis for $\mathcal{Z}_G$, namely
$$P_{\overline{\lambda}}=\prod \alpha_{-\lambda_i}(R^i)\vac.$$

The basis $P_{\overline{\lambda}}$ nicely factors the representation theory of $\mathcal{Z}_G$, decoupling it into the representation theory of $G$ and the representation theory of $S_d$.  The change of basis between the $v_{\overline{\lambda}}$ and the $P_{\overline{\lambda}}$ is determined, via (\ref{gammaintermsofc}, \ref{cintermsofgamma}), by the character table of $G$.  On the other hand, the change of basis between the $P_{\overline{\lambda}}$ and the $P_{\overline{\mu}}$ is clearly just $|G_*|$ copies of the character table of $S_d$.  This is one of the key observations of Qin and Wang \cite{QW}.

In their work on the Gromov-Witten theory of curves, Okounkov and Pandharipande make extensive use of operators $\mathcal{E}_r$ for $r\in \Z$.  We follow our above convention in extending their definition to
\begin{equation*}
\mathcal{E}^\gamma_r(z)=\sum_{k\in\Z+\frac{1}{2}}e^{z(k-\frac{r}{2})}E^\gamma_{k-r,k}+\frac{\delta_{r,0}}{\varsigma(z)}.
\end{equation*}
where $$\varsigma(z)=e^\frac{z}{2}-e^\frac{-z}{2}.$$

We warn the reader that this definition conflicts with the definition of $\mathcal{E}^{(i)}_r(z)$ given in section 4.2 of \cite{QW}.

The operator $\mathcal{E}^\gamma_r(z)$ has energy $-r$, and specialize to the standard bosonic operators $\alpha^\gamma_r$ on $\mathcal{Z}_G$:
$$\mathcal{E}^\gamma_r(0)=\sum_{k\in\Z+\frac{1}{2}}E^{\gamma}_{k-r,k}=\alpha^\gamma_r, r\neq 0.$$

The operators $\mathcal{E}^\gamma_r(z)$ satisfy $\mathcal{E}^\gamma_r(z)^*=\mathcal{E}^\gamma_{-r}(z)$, and satisfy the following commutation relation:
\begin{equation} \label{comE}
[\mathcal{E}^\gamma_r(z),\mathcal{E}^\gamma_s(w)]=\varsigma\left(\det\left[\begin{array}{cc} a & z \\ b & w \end{array}\right]\right)\mathcal{E}^\gamma_{r+s}(z+w).
\end{equation}

For $\gamma\neq \gamma^\prime$, we of course have
$$[\mathcal{E}^\gamma_r(z),\mathcal{E}^{\gamma^\prime}_s(w)]=0.$$

Further following Okounkov and Pandharipande, we define operators $\mathcal{P}^\gamma_k, k>0$:
\begin{equation*}
\frac{\mathcal{P}^\gamma_k}{k!}=[z^k]\mathcal{E}^\gamma_0(z).
\end{equation*}

The operators $\mathcal{P}_k$ are intimately related to the character theory of the symmetric group, and their simplification via the use of completed cycles \cite{OP1}. In particular, the operator
\begin{equation*}
\mathcal{F}_2=\frac{\mathcal{P}_2}{2}=\sum_{k\in\Z+\frac{1}{2}}\frac{k^2}{2}E_{k,k}
\end{equation*}
acts diagonally on the basis $v_{\lambda}$, and multiplies it by $\chi_{\lambda}(T)$.

In \cite{FW}, this result is extended to the wedge product.  For $c\in G_*$, they define
\begin{equation} \label{FWdef}
\mathcal{F}_2^c=\sum_{\gamma\in G^*} \frac{|G|^2\gamma(c)}{(\dim \gamma)^2\zeta_c}\mathcal{F}^\gamma_2
\end{equation}
and in Theorem (3), show that:

\begin{equation} \label{FWtheorem}
\mathcal{F}_2^c v_{\overline{\lambda}}=\chi_{\overline{\lambda}}(T_c)v_{\overline{\lambda}}
\end{equation}

When we return to Gromov-Witten theory, we will restrict our attention to $G=K$ is abelian, and focus on the case $c$ is the class of the identity, which we will denote by $0$.  In this case formula (\ref{FWdef}) simplifies considerably:
\begin{equation} \label{FWabelian}
\mathcal{F}_2^{0}=|K|\sum_{\gamma}\mathcal{F}^{\gamma}_2 .
\end{equation}

We note that we can now express wreath Hurwitz numbers as an expectation on $\mathcal{Z}_G$.  Starting with equation (\ref{hurwitzrep}), using the transition functions between $P_{\overline{\mu}}$ and $v_{\overline{\lambda}}$, and the quoted result (\ref{FWtheorem}), we see that

\begin{lemma} \label{hurwitzwedge}
\begin{equation*} 
\Hur^\bullet_{g,d}(\overline{\mu}, \overline{\nu})=\frac{1}{\mathfrak{z}(\overline{\mu})}\frac{1}{\mathfrak{z}(\overline{\nu})}
\left\langle
\prod_{i=1}^{\ell(\mu)} \alpha_{\mu_i}(c^\mu_i)
\left(\mathcal{F}^0_2\right)^b
\prod_{j=1}^{\ell(\nu)} \alpha_{-\nu_j}(c^\nu_j) \right\rangle .
\end{equation*}
\end{lemma}

\section{Commuting 2-Toda Hierarchies for Wreath Hurwitz Numbers}  \label{okounkov detour}
In this section, we extend the results of \cite{OHur} to wreath Hurwitz numbers.

We first package the disconnected wreath Hurwitz numbers into a convenient generating series.  We will use two separate sets of variables $s^c_i, t^c_i$, $c\in G_*$ to index the ramification conditions over $0$ and $\infty$, respectively.  We will make use of the related set of variables $s^\gamma_i, t^\gamma_i, \gamma\in G^*$.  These variables sets will be related by the inverse of the relations (\ref{cintermsofgamma}) and (\ref{gammaintermsofc}), namely:
$$s_m^\gamma=\sum_{c\in G_*}\gamma(c^{-1})s_m^c,$$
and
$$s_m^c=\sum_{\gamma\in G^*}\zeta_c^{-1}\gamma(c)s_m^\gamma.$$

Then it is clear that
$$\sum_{\gamma\in G^*} s_m^\gamma \alpha^\gamma_m=\sum_{c\in G_*} s_m^c\alpha^c_m,$$ which we will for convenience denote $\alpha^s_m$.

For $\overline{\mu}$ a conjugacy class $\{(\mu_j, c_j)\}$, we will use
$$s_{\overline{\mu}}=\prod_{j=1}^{\ell(\overline{\mu})} s^{c_j}_{\mu_j}$$
and analogously for $t_{\overline{\nu}}$.

Okounkov showed that the generating function for ordinary disconnected Hurwitz numbers 
$$
\tau(s,t,q,\beta)=\sum \Hur^\bullet_g(\mu,\nu)q^d s_\mu t_\nu \frac{\beta^b}{b!}
$$
was a tau function for the 2-Toda hierarchy.  We now review what this means.

Recall that for a matrix $M$, the set of $k\times k$ minors satisfy a set of quadratic relations known as the Pl\"ucker relations.  For $M\in GL(\infty), v,w\in \infwedge V$, the matrix elements $(Mv, w)$ are basically minors of the infinite matrix $M$, and hence satisfy Pl\"ucker relations.  The 2-Toda hierarchy is what results when we package these matrix elements in a generating function $\tau$, and ask what the Pl\"ucker relations say about $\tau$.

Introduce the operator
$$\Omega=\sum_{k\in\Z+1/2}\psi_k\otimes \psi^*_k.$$

The operator $\Omega$ is $GL(\infty)$ invariant: we could replace the bases $\{\psi_k\}, \{\psi^*_k\}$ with any basis of the creation operators and its corresponding dual basis of annihilation operators and still obtain $\Omega$.  As a consequence, we have that
$$[M\otimes M,\Omega]=0$$
for any $M$ in $GL(\infty)$, or more generally for any operator $M$ in the closure of the image of $GL(\infty)$ in $\End(\infwedge V)$.
Hence, for any $v,v',w,w'\in\infwedge V$, we have that
$$([M\otimes M,\Omega] v\otimes v', w\otimes w')=0,$$
which is a compact way to encode the Pl\"ucker relations.

Using the vertex operators
$$\Gamma_{\pm}(t)=\exp\left(\sum_{k>0} t_k\frac{\alpha_{\pm k}}{k}\right)$$
and the translation operators $T$
$$T \bigwedge\underline{e_i}=\bigwedge \underline{e_i+1}$$
we can encode all possible nonzero matrix elements of $M$ in the generating functions
$$\tau^M_n(t,s)=\left\langle T^{-n}\Gamma_+(t)M\Gamma_-(s) T^{n}\right\rangle .$$
The differential equations for the $\tau^M_n(t,s)$ resulting from the Pl\"ucker relations are known as the 2-Toda hierarchy.  For all our $\tau$ functions, the various $\tau_n$ will be rescalings of $\tau_0$, and so will actually satisfy a more restrictive hierarchy.  We say that a function is a $\tau$-function for the 2-Toda hierarchy if it is of the form $\tau^M_0$ for some $M\in GL(\infty)$.

As an example, we derive now the lowest equation of the hierarchy explicitly, following Okounkov.

Let $v_\emptyset$ be the vacuum, and define other vectors by
$$v_{\Box}=\alpha_{-1}v_\emptyset;\quad v_1=Tv_\emptyset; \quad v_{-1}=T^{-1} v_\emptyset.$$

Then we see that
$$\Omega v_\emptyset\otimes v_\Box=v_1\otimes v_{-1},$$
$$\Omega^* v_1\otimes v_{-1}=v_\emptyset\otimes v_\Box-v_\Box\otimes v_\emptyset$$
and so we get the equation
$$(Mv_1, v_1)(Mv_{-1},v_{-1})=(Mv_\emptyset, v_\emptyset)(Mv_\Box, v_\Box)-(Mv_{\emptyset},v_\Box)(Mv_\Box, v_\emptyset)$$
which can be rewritten as
$$
\left\langle T^{-1} M T\right\rangle \left\langle TMT^{-1} \right\rangle
=\left\langle M \right\rangle\left\langle \alpha_1M\alpha_{-1} \right\rangle-
\left\langle \alpha_1 M \right\rangle \left\langle M\alpha_{-1} \right\rangle. $$
In terms of the $\tau$ functions, this is
$$\tau_{n+1}\tau_{n-1}=\tau_n\frac{\partial^2}{\partial t_1\partial s_1}\tau_n -\frac{\partial}{\partial s_1}\tau_n\frac{\partial}{\partial t_1}\tau_n, $$
or, finally,
$$\frac{\partial^2}{\partial t_1\partial s_1} \log \tau_n=\frac{\tau_{n+1}\tau_{n-1}}{\tau_n^2}.$$

We define
$$\tau_G(s,t,q,\beta)=\sum \Hur^\bullet_{G,g}(\overline{\mu},\overline{\nu})q^d\ s_{\overline{\mu}}t_{\overline{\nu}}\frac{\beta^b}{b!}$$
with the usual identification $b=2g-2+\ell(\overline{\mu})+\ell(\overline{\nu})$.  Then we have:

\begin{theorem} \label{wreathtoda}
Expressed in the variables $s_m^\gamma$, the function $\tau_G$ factors into the product of $|G^*|$ functions $\tau_\gamma$.  The function $\tau_\gamma$ depends only on the $s_m^\gamma$, and is a $\tau$ function of the 2-Toda hierarchy.
\end{theorem}

\begin{proof}

We will use the operator
$$H^0=\sum_{\gamma^*} H^\gamma$$
on $\mathcal{Z}_G$.

Following Convention \ref{tensorconvention}, we introduce the vertex operators:
$$\Gamma^\gamma_{\pm}(s)=\exp\left(\sum_{k=1}^\infty \frac{s^\gamma_k}{k} \alpha^\gamma_{\pm k}\right)$$

We similarly define
$$\Gamma^c_{\pm}(s)=\exp\left(\sum_{k=1}^\infty \frac{s^c_k}{k} \alpha^c_{\pm k}\right).$$
and
$$\Gamma^0_{\pm}(s)=\prod_{\gamma\in G^*}\Gamma^\gamma_{\pm}(s)=\prod_{c\in G_*}\Gamma^c_{\pm}(s).$$

Restricting to the case where we consider only $T_0$, we note that
$$\mathcal{F}_2^0=\sum_{\gamma\in G^*}\frac{|G|}{\dim V_{\gamma}}\mathcal{F}^\gamma_2.$$

Using the vertex operators and equation (\ref{hurwitzwedge}),  we immediately have
\begin{eqnarray*}
\tau_G(s,t,q,\beta)&=&
\left\langle \Gamma^0_+(s)  q^{H^0}e^{\beta\mathcal{F}_2^0} \Gamma^0_-(t)\right\rangle \\
&=&\prod_{\gamma\in G^*} \left\langle \Gamma^\gamma_+(s) q^{H^\gamma}
     e^{\beta\frac{|G|}{\dim V_{\gamma}}\mathcal{F}^\gamma_2}
     \Gamma^\gamma_-(t)\right\rangle \\
&=&\prod_{\gamma\in G^*} \tau(s^\gamma, t^\gamma, q, \frac{|G|}{\dim V_{\gamma}}\beta).
\end{eqnarray*}

Thus, we see that expressed in terms of the variables $s^\gamma$, the function $\tau_G$ becomes a product of $|G^*|$ tau functions of the 2-Toda hierarchy, in independent variable sets.
\end{proof}

\chapter{Operator Expressions for Gromov-Witten theory}
This chapter ties together the previous chapters to prove the main result: an operator expression for the equivariant Gromov-Witten invariants of $\mathcal{X}$.  The first step is to obtain an operator expression for the Hurwitz-Hodge integrals $H^\bullet_{\rrr}(z_i, u)$.  Section \ref{orbifold ELSV} briefly reviews \cite{JPT}, joint work with Pandharipande and Tseng, that should be viewed as an extension of the ELSV formula from Hodge integrals to Hurwitz-Hodge integrals.  This result, together with the previous chapter, provides an operator formalism for $H^\bullet_{\rrr}(z_i, u)$ at certain discrete values of the $z_i$.  Section \ref{interpolation} interpolates this expression to an open domain of $\C^n$.  This interpolation requires several technical lemmas whose proofs are relegated to the appendices.  Finally, the operator expression for Hurwitz-Hodge integrals is combined with the results of the localization procedure to produce an operator expression for the Gromov-Witten invariants of $\mathcal{X}$.

\section{Orbifold ELSV Formula} \label{orbifold ELSV}

The ELSV formula \cite{ELSV} relates Hodge integrals and Huritwz numbers:
$$\Hur_g(\mu)=\frac{b!}{\aut(\mu)}\prod_{i=1}^{\ell(\mu)}\frac{\mu_i^{\mu_i}}{\mu_i!}\int_{\mbar_{g,\ell(\mu)}}
\frac{\sum_{i=0}^g (-1)^i\lambda_i}{\prod_{j=1}^{\ell(\mu)}(1-\mu_j\psi_j)}.
$$
It is important to note that the left hand side is only defined when the $\mu_i$ are positive integers, while the right hand side makes sense for arbitrary values of the $\mu_i$.  However, since the right hand side is a rational function of the $\mu_i$, the ELSV formula also determines Hodge integrals in terms of Hurwitz numbers.

The ELSV formula has been extended in \cite{JPT} to determine linear Hurwitz-Hodge integrals of arbitrary abelian groups.  Recall that any irreducible representation $\phi$ of a finite abelian group $R$ is pulled back from the standard representation $U$ of $\Z_r$ as the group of units:
$$0\to K\to R\stackrel{\phi}{\to}\Z_r\to 0.$$

Choose a preimage $x\in R$ of $1\in\Z_r$, and define $\kk=rx\in K$, and define $\overline{r}_{\kk}$ to be the weighted partition
$$\overline{r}_{\kk}=\{\underbrace{(r,-\kk),\dots (r,-\kk)}_{\text{$d/r$ times}}\}.$$  
Since $K$ is abelian a $K_*$-weighted partition is really just a $K$ weighted partition.  For $\overline{\mu}=\{(\mu_i, k^\mu_i)\}$ , define an $\ell$-tuple of elements of $R$ by
$$-\overline{\mu}=\{k^\mu_1-\mu_1x,\dots, k^\mu_\ell-\mu_\ell x\}.$$
Note that while the parts of $\overline{\mu}$ are unordered, an ordering is chosen for $-\overline{\mu}$.

Then Theorem 3 in \cite{JPT} gives the following formula for certain $K_d$ Hurwitz numbers in terms of Hurwitz-Hodge integrals:
\begin{eqnarray} \Hur^\circ_{g, K}(\overline{r}_{\kk},\overline{\mu})
&=&\frac{b!}{\aut(\overline{\mu})}
r^{1-g+\sum \leftover{\frac{\mu_i}{r}}}
\prod_{i=1}^{\ell(\mu)} \frac{\mu_i^{\floor{\frac{\mu_i}{r}}}}{\floor{\frac{\mu_i}{r}}!}
\int_{\mbar_{g,-\overline{\mu}}(\B R)}
\frac{\sum_{i=0}^\infty (-r)^i \lambda_i^\phi}{\prod_{j=1}^{\ell(\mu)}(1-\mu_j\psibar_j)}  \notag \\
&=&\frac{|K|^{\ell(\mu)}b!}{\mathfrak{z}(\overline{\mu})}
r^{1-g+\sum \leftover{\frac{\mu_i}{r}}}
\prod_{i=1}^{\ell(\mu)} \frac{\mu_i^{\floor{\frac{\mu_i}{r}}}}{\floor{\frac{\mu_i}{r}}!}
H^{0,\circ}_{g,-\overline{\mu}}(\mu) \label{JPTtheorem}.
\end{eqnarray}

We will now derive a more convenient version of this formula for our use.  Suppose that $\leftover{\frac{-\mu_i}{r}}=\frac{a_i}{r}, 0\leq a_i<r$.  Then we have that
\begin{eqnarray*}
-\overline{\mu}_i &=& k^\mu_i-\mu_ix \\
&=& k^\mu_i+\left(\floor{\frac{-\mu_i}{r}}r+a_i\right)x \\
&=&k^\mu_i+\floor{\frac{-\mu_i}{r}}\kk +a_i x,
\end{eqnarray*}
giving the monodromy conditions $-\overline{\mu}$ in terms of the cocycle description of $R=\Z_r\times_\beta K$.

For
 $$\rrr=(r_1,\dots, r_{\ell(\rrr)});\qquad r_i=(a_i, k_i)\in\Z_r\times_\beta K=R$$
we introduce the $K$-weighted partition
\begin{equation*}
 \overline{\mu}^\rrr_i=\left(\mu_i, k_i-\floor{\frac{-\mu_i}{r}}\kk\right).
\end{equation*}
Shifting the monodromy conditions of both sides of equation (\ref{JPTtheorem}) by $\floor{\frac{-\mu_i}{r}}$ and summing over genus, we see that for $\mu$ a partition with $-\mu_i\mod r=a_i$, we have:
\begin{equation}  \label{JPTgenerating}
 H^{0,\bullet}_{\rrr}(\mu,u)=
\sum_{g} \left(ur^{1/2}\right)^{2g-2} \frac{r^{-\sum \leftover{\frac{\mu_i}{r}}}}{|K|^{\ell(\mu)}}
\frac{\mathfrak{z}(\overline{\mu}^\rrr)}{b!}
\left(
\prod_{i=1}^{\ell(\mu)}
\frac{\floor{\frac{\mu_i}{r}}!}{\mu_i^{\floor{\frac{\mu_i}{r}}}}
\right)
\Hur^\bullet_{g,K}(\overline{r}_{\kk}, \overline{\mu}^\rrr).
\end{equation}

We now express the function $H^{0,\bullet}_{\rrr}$ in terms of the Fock space formalism, by using the operator formula we derived for the $\Hur^\bullet_{g,K}(\overline{\mu},\overline{\nu})$ we derived in section \ref{wreathfock}.  To keep our formulas compact, for $m\in\Z$ and $k\in K$ we introduce the operator $\alphatw{m}{k}$ defined by:
\begin{equation} \label{alphatw}
\alphatw{m}{k}=\alpha_{m}(k-\floor{\frac{m}{r}}\kk).
\end{equation}

Then by equation (\ref{hurwitzwedge}), we have:
\begin{eqnarray*}
\Hur^\bullet_{g, K}(\overline{r}_{\kk},\overline{\mu}^{\rrr})
&=&
\frac{1}{|K|^{d/r}r^{d/r}(d/r)!}\frac{1}{\mathfrak{z}(\overline{\mu}^\rrr)}
\left\langle \alpha_{r}(-\kk)^{d/r}
\left(\mathcal{F}^0_2\right)^b
\prod_{i=1}^{\ell(\mu)}
\alphatw{-\mu_i}{k_i}
\right\rangle \\
&=&
\frac{1}{\mathfrak{z}(\overline{\mu}^{\rrr})}
\left\langle e^{\frac{\alpha_r(-\kk)}{|R|}}
\left(\mathcal{F}^0_2\right)^b
\prod_{i=1}^{\ell(\mu)}
\alphatw{-\mu_i}{k_i}\right\rangle.
\end{eqnarray*}

Substituting this into equation (\ref{JPTgenerating}) with $u$ replaced with $u/r^{1/2}$, we obtain:
\begin{equation} \label{operatorhodge}
H^{0,\bullet}_{\rrr}\left(\mu,\frac{u}{r^{1/2}}\right)
=u^{-|\mu|/r-\ell(\mu)}
\frac{r^{-\sum\leftover{\frac{\mu_i}{r}}}}{|K|^{\ell(\mu)}}
\left(
\prod_{i=1}^{\ell(\mu)}
\frac{\floor{\frac{\mu_i}{r}}!}{\mu_i^\floor{\frac{\mu_i}{r}}}
\right)
\left\langle e^{\frac{\alpha_r(-\kk)}{|R|}}e^{u\mathcal{F}^0_2}
\prod_{i=1}^{\ell(\mu)}
\alphatw{-\mu_i}{k_i}\right\rangle.
\end{equation}

As was the case with the usual ELSV formula, this formula evaluates $H_\rrr(z)$ only when $z_i$ is an integer.  The monodromy conditions on each point give the further restrction that, if $r_i=(a_i, k_i)\in \Z_r\times_\beta K$, then $z_i\cong -a_i(\mod r)$.  However, $H_\rrr$ is a polynomial (or, in the disconnected case, a rational function), and the orbiofld ELSV formula provides an infinite number of evaluations, and so this formula determines the function $H_\rrr(z_\rrr)$ for all values of $z$.  Our next goal is to refine equation (\ref{operatorhodge}) to an operator formula that does this explicitly, i.e. to interpolate the right hand side so that it makes sense for $\mu_i$ a complex number.

\section{Interpolating the Operator expression} \label{interpolation}

In this section we interpolate the operator expression for Hurwitz-Hodge integrals from the previous section, which is valid only for integers satisfying a congruency condition, to one valid for a certain open domain of $\C^n$.  In addition to interpolating the expression, we adapt it slightly by conjugating by an operator that fixes the vacuum.  This will not change the vacuum expectation, and will give us operators that make the decomposition easier to see.

The strategy is as follows.  We will first replace the operators from the previous section, which only make sense for $\mu_i$ integers, with operators $\mathcal{A}_{\rrr_i}(z_i,u)$ that appear to be formal power series, but that simplify when $z_i=\mu_i$ in the previous section to the operators of the orbifold ELSV formula.  By studying the convergence of the $\mathcal{A}_{\rrr_i}(z_i,u)$, and computing their commutators, we will eventually show that their vacuum expectations are rational functions.   Then, since the orbifold ELSV formula says these rational functions agree with the Hurwitz-Hodge generating functions at infinitely many points, we will be able to deduce that they are in fact equal.

Before we introduce the key operators, we recall some preliminary notation.  Recall the {\em Pochhammer symbol}:
\begin{equation*} (x+1)_n=\frac{(x+k)!}{x!}=\left\{\begin{array}{ll} (x+1)(x+2)\cdots(x+n) &n\geq 0 \\
(x(x-1)\cdots(x+n+1))^{-1} &n\leq 0\end{array} \right. .
\end{equation*}

From the definition, $(x+1)_n$ vanishes for $-n\leq x\leq -1$ an integer, and $1/(x+1)_n$ vanishes for $0\leq x \leq -(n+1)$ an integer.

We will also use the notation
\begin{equation*}
\mathcal{S}(z)=\frac{\varsigma(z)}{z}=\frac{\sinh(z/2)}{z/2}.
\end{equation*}

For $0\leq a\leq r-1$ and $\gamma\in K^*$, we define
\begin{equation} \label{defA}
\mathcal{A}^\gamma_{a/r}(z,u)=\frac{z(\gamma(-\kk)r)^{a/r}}{z+a}\mathcal{S}(|R|uz)^{\frac{z+a}{r}}\sum_{i=-\infty}^\infty \frac{z^i\mathcal{S}(|R|uz)^i}{(1+\frac{z+a}{r})_i}\mathcal{E}^\gamma_{ir+a}\big(|K|uz\big).
 \end{equation}

The operators $\mathcal{A}^\gamma_{a,r}(z,u)$ will play the analog of the operators $\mathcal{A}(x, y)$ in \cite{OP2}, although we have made a few minor changes.  First, we use $z, u$ as the two variables, where \cite{OP2} has $z, zu$.  Additionally, we have modified their operators slightly.   In case $R$ is the trivial group, there is only one such operator, $\mathcal{A}^0_{0/1}$, which simplifies to
$$
\mathcal{A}^0_{0/1}(z,u)=\mathcal{S}(uz)^z\sum_{i=-\infty}^\infty\frac{u^{-i}\varsigma(uz)^i}{(1+z)_i}\mathcal{E}^0_{i}(uz).
$$
This is the same as the operator $\mathcal{A}(z,uz)$ from \cite{OP2} except for the factor of $u^{-i}$.  As $\mathcal{E}_i$ has energy $i$, this change amounts to conjugating by the operator $u^H$.  Since $H$ and its adjoint fix the vacuum, this does not affect operator expectations of products of the $\mathcal{A}$.  Furthermore, this rescaling is in many ways rather natural - it was already used in \cite{OP2} to prove Proposition 9 about convergence.

We will use a related definition: for $\rrr=(a, k)\in \Z_r\times_\beta K=R $, we define
\begin{equation} \label{defAr}
\mathcal{A}_{\rrr}=\sum_{\gamma\in K^*} \gamma(-k)\mathcal{A}^\gamma_{a/r}.
\end{equation}

These definitions are motivated by:

\begin{proposition} \label{ELSVA}
 Let $r_i=(a_i, k_i)\in \Z_r\times_\beta K= R^n$ and $z_i>0, z_i\equiv -a_i\mod r$.  Then
\begin{equation*} H^\bullet_\rrr\left(z_\rrr, \frac{u}{r^{1/2}}\right)=(u|K|)^{-\ell(\rrr)}\left\langle \prod_{i=1}^{\ell(\rrr)} \mathcal{A}_{r_i}(z_i,u)\right\rangle.
\end{equation*}
\end{proposition}

We note that this proposition does not determine the $\mathcal{A}^\gamma_{a/r}$, and our choice of definition is not the one that follows most naturally from the orbifold ELSV formula in the previous section.   However, shortly we will see that our definition is well suited for seeing decomposition.  The relevant properties are visible now: the only dependence on the choice of $K$ and $R$ appear in a global factor of $\gamma(\kk)$, and in the factors of $|K|$, but this second dependence can be absorbed by rescaling $u$ (recall that $|R|=r|K|$).

\begin{proof} We begin by recalling Equation (\ref{operatorhodge}):
\begin{equation*}
H^{0,\bullet}_{\rrr}\left(\mu,\frac{u}{r^{1/2}}\right)
=u^{-|\mu|/r-\ell(\mu)}
\frac{r^{-\sum\leftover{\frac{\mu_i}{r}}}}{|K|^{\ell(\mu)}}
\left(
\prod_{i=1}^{\ell(\mu)}
\frac{\floor{\frac{\mu_i}{r}}!}{\mu_i^\floor{\frac{\mu_i}{r}}}
\right)
\left\langle e^{\frac{\alpha_r(-\kk)}{|R|}}e^{u\mathcal{F}^0_2}
\prod_{i=1}^{\ell(\mu)}
\alphatw{-\mu_i}{k_i}\right\rangle
\end{equation*}
when $\mu_i=-a_i\mod r$.

Since $\alpha_r(-\kk)$ and  $\mathcal{F}^0_2$ , both annihilate the vacuum, the vacuum expectation above (ignoring the prefactors) is equivalent to
$$\left\langle
\prod_{i=1}^{\ell(\mu)}
e^{\frac{\alpha_r(-\kk)}{|R|}}
e^{u\mathcal{F}^0_2}
\alphatw{-\mu_i}{k_i}
e^{-u\mathcal{F}^0_2}
e^{\frac{-\alpha_r(-\kk)}{|R|}}
\right\rangle.
 $$

It will be useful to change our point of view, so that the identification of $\mathcal{Z}_K$ with $\bigotimes \infwedge_0 V^\gamma$ is more visible.   By equation (\ref{cintermsofgamma}) we have:
$$\alpha_r(-\kk)=\sum_{\gamma\in K^*} \gamma(\kk)\alpha^\gamma_r,$$
and similarly, expanding $\alphatw{-\mu_i}{k_i}$ by its definition (\ref{alphatw}), we find that
$$\alphatw{-\mu_i}{k_i}=\sum_{\gamma\in K^*} \gamma(-k_i)\gamma(\kk)^{\floor{\frac{-\mu_i}{r}}}\alpha_{-\mu_i}^\gamma.$$

So, using Equation (\ref{FWabelian}) to expand $\mathcal{F}^0_2$, we see that
\begin{multline*}
e^{\frac{\alpha_r(-\kk)}{|R|}}
e^{u\mathcal{F}^0_2}
\alphatw{-\mu_i}{k_i}
e^{-u\mathcal{F}^0_2}
e^{\frac{-\alpha_r(-\kk)}{|R|}} \\
 =
\sum_{\gamma\in K^*}
\gamma(-k_i)\gamma(\kk)^{\floor{\frac{-\mu_i}{r}}}
 \left(e^{\frac{\alpha^\gamma_r}{|R|}}
e^{u|K|\mathcal{F}^\gamma_2}
\alpha_{-\mu_i}^\gamma
e^{-u|K|\mathcal{F}^\gamma_2}
e^{\frac{-\gamma(\kk)\alpha^\gamma_r}{|R|}}
\right).
\end{multline*}

Since both $H^\gamma$ and $H^{\gamma*}$ annihilate the vacuum, we can further conjugate each operator by $c_\gamma^{H^\gamma}$, for any constants $c_\gamma$, and not change the vacuum expectation.   Note that this has the effect of rescaling by $c^E_\gamma$ all operators on $\infwedge V^\gamma$ that change the energy by $E$ .  We will conjugate each operator by
$$
\prod_{\gamma\in K^*}(u\gamma(\kk))^{\frac{H^\gamma}{r}}.
$$
This will leave $\mathcal{F}^\gamma_2$ fixed, and rescale $\alpha^\gamma_r$ by $(u\gamma(\kk))^{-1}$ and $\alpha^\gamma_{-\mu_i}$ by $(u\gamma(\kk))^{\mu_i/r}$.
Using
$$\frac{m}{r}+\floor{\frac{-m}{r}}=\leftover{\frac{m}{r}}-\delta^\vee_r(m)=-\leftover{\frac{-m}{r}}$$
to simplify the powers of $\gamma(\kk)$ we see that:
\begin{multline*} 
\left\langle\sum_{\gamma\in K^*}
\gamma(-k_i)\gamma(\kk)^{\floor{\frac{-\mu_i}{r}}}
\left(u\gamma(\kk)\right)^{-H^\gamma/r}
e^{\gamma(\kk)\frac{\alpha^\gamma_r}{|R|}}
e^{u|K|\mathcal{F}^\gamma_2}
\alpha_{-\mu_i}^\gamma
e^{-u|K|\mathcal{F}^\gamma_2}
e^{\frac{-\gamma(\kk)\alpha^\gamma_r}{|R|}}
\left(u\gamma(\kk)\right)^{H^\gamma/r}\right\rangle \\
=
\left\langle
\sum_{\gamma\in K^*}
u^{\mu_i/r}\gamma(-k_i)\gamma(\kk)^{-a/r}
e^{\frac{\alpha^\gamma_r}{u|R|}}
e^{u|K|\mathcal{F}^\gamma_2}
\alpha^\gamma_{-\mu_i}
e^{-u|K|\mathcal{F}^\gamma_2}
e^{-\frac{\alpha^\gamma_r}{u|R|}}
\right\rangle
.
\end{multline*}
Canceling the prefactor of $u$ with the factor of $u$ appearing in Equation (\ref{operatorhodge}), we see that to prove the proposition, we must show that for $z=-a\mod r$ we have
\begin{equation} \label{Aconj}
\mathcal{A}^\gamma_{a/r}(z,u)=
r^{-\leftover{\frac{z}{r}}}
\gamma(-\kk)^{a/r}
\frac{\floor{\frac{z}{r}}!}{z^{\floor{\frac{z}{r}}}}
\left(e^{\frac{\alpha^\gamma_r}{u|R|}}
e^{u|K|\mathcal{F}^\gamma_2}
\alpha^\gamma_{-z}
e^{-u|K|\mathcal{F}^\gamma_2}
e^{-\frac{\alpha^\gamma_r}{u|R|}}
\right).
\end{equation}

We will now investigate the term in parentheses, beginning by recalling Equation (2.14) of \cite{OP2}:
$$
e^{u\mathcal{F}_2}\alpha_{-m}e^{-u\mathcal{F}_2}=\mathcal{E}_{-m}(um),
$$
which we will use as
\begin{equation} \label{F2term}
e^{u|K|\mathcal{F}^\gamma_2}
\alpha_{-z}^\gamma
e^{-u|K|\mathcal{F}^\gamma_2}
=\mathcal{E}^\gamma_{-z}(u|K|z).
\end{equation}

Now we consider the effect of the $e^{\alpha_r^{\gamma}/(u|R|)}$ terms.

Since
$$
[\alpha_r,\mathcal{E}_{-m}(w)]=\varsigma(rw)\mathcal{E}_{-m+r}(w),
$$
we see that
\begin{eqnarray*}
e^{\frac{\alpha_r^\gamma}{u|R|}}\mathcal{E}^\gamma_{-m}(w)e^{\frac{-\alpha_r^\gamma}{u|R|}} &=&
\sum_{i=0}^\infty \frac{1}{i!}
\left[\left(\frac{\alpha^\gamma_r}{u|R|}\right)^i,\mathcal{E}^\gamma_{-m}(w)\right]
e^{\frac{-\alpha^\gamma_r}{u|R|}}
+\mathcal{E}^\gamma_{-m}(w)
\\
&=&
\sum_{0\leq j\leq i}
\left(\frac{1}{u|R|}\right)^j
\frac{1}{i!}\binom{i}{j}
\underbrace{[\alpha^\gamma_r,[\dots,[\alpha^\gamma_r}_{\textrm{$j$ times}}
,\mathcal{E}^\gamma_{-m}(w)]]\dots ]
\left(\frac{\alpha^\gamma_r}{u|R|}\right)^{i-j}
e^{\frac{-\alpha^\gamma_r}{u|R|}}
\\
&=&
\sum_{j=0}^\infty
\frac{1}{j!}
\left(\frac{\varsigma(rw)}{u|R|}\right)^j
\mathcal{E}^\gamma_{-m+jr}(w).
\end{eqnarray*}

When $m=z$ and $w=u|K|z$ we see that $\frac{\varsigma(rw)}{u|R|}=z\mathcal{S}(u|R|z)$.
Writing $-z=a-(\frac{z+a}{r})r$, we set $b=\frac{z+a}{r}$, and $h=j-b$, so that the sum becomes:
\begin{multline}  \label{tyty}
\sum_{h=-b}^\infty
\frac{1}{(h+b)!}
\left(z\mathcal{S}(u|R|z)\right)^{h+b}
\mathcal{E}^\gamma_{a-br+jr}(u|K|z) \\
=                                               \frac{1}{b!}
\left(z\mathcal{S}(u|R|z)\right)^{b}
\sum_{h=-b}^\infty
\left(z\mathcal{S}(u|R|z)\right)^h\frac{b!}{(h+b)!}\mathcal{E}^\gamma_{a+hr}(u|K|z).
\end{multline}
Since
 $$\frac{b!}{(h+b)!}=\frac{1}{(1+b)_h}=\frac{1}{(1+\frac{z+a}{r})_h}$$
vanishes if $b\in\Z, b\leq -(h+1)$, extending the sum to all $h\in\Z$ does not change the value when $z=-a\mod r$.
Using $b=(z+a)/r=\floor{\frac{z}{r}}+\delta^\vee_r(a)$, we can rewrite the prefactor in (\ref{tyty}) as
$$\left(\frac{rz}{z+a}\right)^{\delta^\vee_r(a)}\frac{z^{\floor{\frac{z}{r}}}}{\floor{\frac{z}{r}}!}\mathcal{S}(u|R|z)^{\frac{z+a}{r}}.$$
Substituting this in and simplifying, we have shown (\ref{Aconj}), and so are done.
\end{proof}

Like equation (\ref{operatorhodge}), we have proven Proposition \ref{ELSVA} only for restricted values of $z_{\rrr}$.  Our next goal is to show that Proposition \ref{ELSVA} actually holds for all values of $z_\rrr$ in an open domain in $\C^{\ell(\rrr)}$.

The first step is see for what values of $z_\rrr$ the left hand side actually makes sense.  When $z_i\neq -a_i\mod r$, the sum in the definition of $\mathcal{A}^\gamma_{a_i/r}$ is infinite in both dimensions, and so the energy of the operators $\mathcal{A}_{r_i}(z,u)$ will in general be bounded on neither side, and hence we have no reason to suppose that the right hand side of Proposition \ref{ELSVA} makes sense except as a formal power series.

In fact, the right hand side of Proposition \ref{ELSVA} has nice convergence properties.  In particular, we define $\Omega\subset \C^n$ by
$$
\Omega=\bigg\{(z_1,\dots, z_n)\in \C^n\bigg|\forall k,|z_k|>\sum_{i=1}^{k-1}|z_i|\bigg\}.
$$
The operators $\mathcal{A}^\gamma_{a/r}$ have poles at negative integers, but away from these, we have

\begin{proposition} \label{prop3}
Let $K$ be a compact set,
$$K\subset \Omega\cap\{z_i\neq -1,-2,\dots, i=1,\dots, n\}.$$
Then for all $\gamma\in K^*, 0\leq a_i< r,$ and $\mu, \lambda$ partitions, the series
$$\left\langle \mathcal{A}^\gamma_{a_1/r}(z_1,u)\cdots \mathcal{A}^\gamma_{a_n/r}(z_n,u)\mu, \lambda\right\rangle$$
converges uniformly and absolutely for all sufficiently small $u\neq 0$.
\end{proposition}
The proof of Proposition \label{prop3} is presented in Appendix \ref{convergence}, which follows the general argument for the proof of Proposition 3 in \cite{OP2}, with some modification and expanded exposition.

As a consequence of Proposition \ref{prop3}, we see that the vacuum expectation
$$\left\langle \mathcal{A}^\gamma_{a_1/r}(z_1,u)\cdots\mathcal{A}^\gamma_{a_n/r}(z_n,u)\right\rangle$$
 is an analytic function of $(z_1,\dots, z_n,u)$ in a neighborhood of the origin intersect $\Omega\times\C^*$.  Hence, we may expand it as a convergent Laurent series.

It will be important for us to control the terms appearing with negative exponents.  To that end, for any ring $U$, we denote by $U((z))$ the ring of formal Laurent series with coefficients in $U$ and degree bounded below:
$$U((z))=\Big\{\sum_{i\in \Z}u_iz^i\Big| u_i\in U, u_i=0 \forall i<<0\Big\}.$$

Then we have
\begin{lemma} \label{degreebelow} $$
\left\langle \mathcal{A}^\gamma_{a_1/r}(z_1,u)\cdots\mathcal{A}^\gamma_{a_n/r}(z_n,u)\right\rangle
\in \C[u^{\pm 1}]((z_n))((z_{n-1}))\cdots ((z_1)).$$
\end{lemma}

Note that this {\em does not} say that power of $z_n$ appearing in the right hand side is bounded below - in general, it will not be.  Rather, if we fix arbitrary $p_1,\dots, p_{n-1}$, then the powers of $z_n$ appearing as the coefficient of $z_1^{p_1}\cdots z_{n-1}^{p_{n-1}}$ will be bounded below.  However, as the $p_i$ go to infinity, the powers of $z_n$ can go to negative infinity.

\begin{proof}
The key point is that the powers of $z$ appearing in coefficient of $\mathcal{E}^\gamma_m$ in $\mathcal{A}^\gamma_{a/r}$ is bounded below by $\floor{\frac{m}{r}}$.  Indeed, the prefactors of $z/(z+a)$ and $\mathcal{S}$ appearing in the definition (\ref{defA}) contribute only positive powers of $z$.  The $\varsigma^i$ factor of the coefficient of $\mathcal{E}_{ir+a}$ has leading term $z^i$, and the factor $\frac{1}{(1+(z+a)/r)_i}$ also contributes only positive terms.

Since $\mathcal{E}^{\gamma*}_m$ has energy $m$, we see that, apart from the constant term of $1/\varsigma(|K|uz)$ appearing in $\mathcal{A}^\gamma_{0/r}(z,u)$ which will also contribute a $z^{-1}$ term, we have:
\begin{equation} \label{bigo}
\mathcal{E}^{\gamma*}_{a/r}(z)v^\gamma_{\mu}=O(z^{-\floor{\frac{|\mu|}{r}}}),
\end{equation}
where we are studying the growth as $z\to 0$, and so this means that there are no terms appearing with lower exponent.

This immediately implies that the power of $z_1$ appearing will be bounded below by $-1$, coming from the constant term, as there are no vectors of negative energy.

Applying Equation (\ref{bigo}) inductively to each $z_i$ in turn gives the result.

\end{proof}

On $\Omega$, we could also expand $H^\bullet_\rrr$ as a Laurent series, in a similar manner.

Note that changing the order of the $z_i$ changes the definition of the domain $\Omega$, and hence the Laurent expansion.  In particular, the function $\frac{1}{z+w}$ can be expanded as a geometric series in two different ways, depending on  which of $|z|$ and $|w|$ is bigger:
\begin{equation} \label{zw}
\frac{1}{z+w}=\frac{1}{w}-\frac{z}{w^2}+\frac{z^2}{w^3}-\dots, \quad |z|<|w|
\end{equation}
\begin{equation} \label{wz}
\frac{1}{z+w}=\frac{1}{z}-\frac{w}{z^2}+\frac{w^2}{z^3}-\dots, \quad |z|>|w|
\end{equation}

Subtracting equation (\ref{wz}) from equation (\ref{zw}), we get the formal series
\begin{equation*}
\delta(z,-w)=\frac{1}{w}\sum_{i\in\Z}\left(-\frac{z}{w}\right)^n
\end{equation*}
which converges nowhere, but acts as a formal delta function at $z=-w$ because it satisfies satisfies
\begin{equation} \label{deltaequation}
(z+w)\delta(z,-w)=0.
\end{equation}

Since the two point unstable contribution is $\frac{z_iz_j}{|R|(z_i+z_j)}$, and occurs in genus 0, we see that swapping the order of $(z_i,r_i)$ and $(z_{i+1},r_{i+1})$, should change the Laurent expansion of $H^\bullet_{\rrr}(z_\rrr,\frac{u}{r^{1/2}})$ by
\begin{equation*} 
\left(\frac{u}{r^{1/2}}\right)^{-2}\delta_{r_i, \op{r}_{i+1}}\frac{z_iz_{i+1}}{|R|}\delta(z_i,-z_{i+1})
=\delta_{r_i, \op{r}_{i+1}}\frac{z_iz_{i+1}}{u^2|K|}\delta(z_i,-z_{i+1})
\end{equation*}
Comparing this with Proposition \ref{ELSVA} and taking note of the prefactor of $(u|K|)^{-\ell(\rrr)}$ suggests the following formula for the commutators of the $\mathcal{A}_{r_i}(z,u)$:
$$[\mathcal{A}_{r_1}(z,u),\mathcal{A}_{r_2}(w,u)]=\delta_{r_1, -r_2}|K|\delta(z,-w).$$
  We will derive this formula as a corollary of the following commutator formula for the $\mathcal{A}^\gamma_{a/r}(z,u)$, which we will make further use of later:

\begin{lemma} \label{maincommutatorlemma}
$$
 [\mathcal{A}^{\gamma}_{a/r}(z,u),\mathcal{A}^{\gamma^\prime}_{b/r}(w,u)]=
\delta_{\gamma, \gamma^\prime}
\delta_r(a+b)
\gamma(\kk)^{-(a+b)/r}
zw\delta(z, -w).
$$
\end{lemma}

The proof of Lemma \ref{maincommutatorlemma} is rather technical, and we defer its proof until \ref{maincommutatorlemmaproof}.

\begin{corollary} \label{commutatorcor}
$$
[\mathcal{A}_{r_1}(z,u),\mathcal{A}_{r_2}(w,u)]=\delta_{r_1, -r_2}|K|\delta(z,-w).
$$
\end{corollary}
\begin{proof}

Let $r_1=(a,k_1)\in \Z_r\times_\beta R, r_2=(b, k_2)$.  Then, expanding $\mathcal{A}_{r_1}, \mathcal{A}_{r_2}$ by their defintion (\ref{defAr}), we have:
\begin{eqnarray*}
[\mathcal{A}_{r_1}(z,u),\mathcal{A}_{r_2}(w,u)]
&=&\sum_{\gamma,\gamma^\prime\in K^*}
\gamma(-k_1)\gamma^\prime(-k_2)
[\mathcal{A}^\gamma_{a/r}(z,u), \mathcal{A}^{\gamma^\prime}_{b/r}(w,u)] \\
&=&\sum_{\gamma\in K^*}
\gamma(-k_1-k_2-\delta^\vee_r(a)\kk)\delta_{a,\op{b}} zw\delta(z,-w)
\end{eqnarray*}

By character orthogonality, this sum is zero if $k_1+k_2 \neq -\delta^\vee_r(a)\kk$, and $|K|$ otherwise.  From the definition of $R=\Z_r\times_\beta K$, this combines with $\delta_{a,\op{b}}$ to give $|K|\delta_{r_1, -r_2}$.
\end{proof}
As a further corollary of Lemma \ref{maincommutatorlemma}, we see that the left hand side of Proposition \ref{ELSVA} has poles exactly where the right hand side does, and otherwise is a power series:

\begin{corollary} \label{powerseries}
The series:
$$\left(\prod_{\substack{i<j \\ \rrr_i=-\rrr_j}}(z_i+z_j)\right)
\left\langle \mathcal{A}_{\rrr_1}(z_1,u)\cdots \mathcal{A}_{\rrr_n}(z_n,u)\right\rangle$$
is independent of the ordering of the $(z_i, \rrr_i)$, and is an element of
$$\prod_{\{i|\rho_i=0\}}z_i^{-1}\C[u^{\pm 1}][[z_1,\dots, z_n]].$$
\end{corollary}
\begin{proof}
That the series is independent of the ordering is immediate from Corollary \ref{commutatorcor} and Equation (\ref{deltaequation}).  Because the series is independent of ordering, to show that it is a power series except for a factor of $z_i^{-1}$ for $i$ with $\rrr_i=0$, it is enough to do so for $z_1$.  However, this follows immediately from the proof of Proposition \ref{degreebelow} and Equation (\ref{defAr}) expanding $\mathcal{A}_{\rrr_i}$ in terms of $\mathcal{A}^\gamma_{a/r}.$
\end{proof}

We note that the series above is not fully symmetric in the $z_i$, but is under the action of $\Aut(\rrr)\subset S_n$.

\begin{proposition} \label{rational}
The coefficients of powers of $u$ in the right hand side of Proposition \ref{ELSVA},
$$[u^m]\left\langle \mathcal{A}_{\rrr_1}(z_1,u)\cdots \mathcal{A}_{\rrr_n}(z_n,u)\right\rangle, m\in\Z$$
are rational functions in the $z_i$, with at most simple poles along the divisors $z_i+z_j=0$ for $i,j$ with $\rrr_i+\rrr_j=0$, and divisors $z_i$ with $\rrr_i=0$.
\end{proposition}

\begin{proof}
From Corollary \ref{powerseries}, and the fact that expanding $1/(z_i+z_j)$ on $\Omega$ will only introduce negative powers of $z_n$, we see that it is enough to show that the coefficient of $z_n$ is bounded from above.  We will accomplish this by pairing any factor of $z_n^\ell$, with $\ell$ positive, by a factor of $u^{\ell/2}$, and then show that in the remaining terms the powers of $u$ appearing have degree bounded below.

We will consider the expansion of the $\mathcal{A}$ in terms of the $\mathcal{E}$, and hence terms of the form

$$\left\langle \mathcal{E}_{k_1}(u|K|z_1)\cdots \mathcal{E}_{k_n}(u|K|z_n)\right\rangle.$$
These terms vanish unless $\sum k_i=0$ and $k_n\leq 0$.

As in the definition of $\mathcal{A}^\gamma_{1/r}$ (Equation (\ref{defA})) the $\mathcal{E}_{k}$ appear with $k=a+ri$, we see that if $a_n=0$, we must have $i\leq 0$, while if $a_n\neq 0$ we must have $i\leq -1$.  In either case, the pole at $z=-a$ occurring in the prefactor will be canceled, and the vacuum expectation will depend on $z_n$ only through terms of the form
\begin{equation} \label{uz2}
(ur)^{a/r}\mathcal{S}(|R|uz)^{\frac{z+a}{r}}
\end{equation}
from the prefactor,
\begin{equation} \label{fhjdsf}
e^{xu|K|z_n}
\end{equation}
from the definition of $\mathcal{E}$, and
\begin{equation} \label{coeffterm}
\frac{z_n}{r}\left(\frac{z_n+a_n}{r}-1\right)\cdots \left(\frac{z_n+a_n}{r}+i+1\right)\times \left(z_n\mathcal{S}(|R|uz_n)\right)^i
\end{equation}
from the coefficient of $\mathcal{E}_{a_n+ir}$, where the first term in the product is $z_n/r$ instead of $(z_n+a_n)/r$ because we have multiplied it by the prefactor $z_n/(z_n+a_n)$.

Now, it is clear that in term (\ref{fhjdsf}), $z_n^m$ occurs with coefficient $u^m$.  There is a less obvious grouping for the terms of the form (\ref{uz2}) - rewriting $\mathcal{S}$ as $e^{\ln \mathcal{S}}$, and using the Taylor expansion for $\ln (1+x)$, we see that the term $z_n^\ell$ occurs with a coefficient of $u^p$, with $p\geq \ell/2$.  Finally, to handle the $z_n$ appearing in (\ref{coeffterm}), observe that the first product is a polynomial in $z_n$ of degree $-i$, and so we can pair it with the $z_n^i$ appearing, to get all negative powers of $z_n$, except for those paired with $u$.  We have thus shown that all positive appearances of $z_n$ occur with a positive power of $u$ as well.  Furthermore, the only $u$ appearing as a negative power are those coming from the constant term of $\mathcal{E}_0$, and so we are done.
\end{proof}

From Proposition \ref{rational} it follows easily that Proposition \ref{ELSVA} holds on an open set, not just on the integers:

\begin{theorem} \label{HurwitzOperatorFormula}
\begin{equation*} H^\bullet_\rrr(z_\rrr, \frac{u}{r^{1/2}})=(u|K|)^{-\ell(\rrr)}\left\langle \prod_{i=1}^{\ell(\rrr)} \mathcal{A}_{r_i}(z_i,u)\right\rangle.
\end{equation*}
\end{theorem}
\begin{proof}
By Proposition \ref{prop3} the coefficients of $u$ on the right hand side are analytic on $\Omega$, and by Proposition \ref{rational}, they are actually rational.  The same is true of the coefficients of the $u$ on the left hand side, and by Proposition \ref{ELSVA}, the two sides agree when $z_i$ is a positive integer congruent to $-a_i\mod r$.  The set of such $z_i$ in $\Omega$ forms a Zariski dense set, and hence the two sides are equal.
\end{proof}

The precise definition of our operators were chosen so that they would be compatible with decomposition, and we illustrate this now with Theorem \ref{HurwitzOperatorFormula}.

Expanding $$\mathcal{A}_{r_i}(z_i,u)=\sum_{\gamma\in K^*} \gamma(-k_i)\mathcal{A}^\gamma_{a_i/r}(z_i,u)$$

\section{Global operator expression} \label{GlobalOperatorExpression}

Recall that the culmination of our localization calculation was equation (\ref{globalformula}):
$$
G^\bullet_{d,\rrr,\sss}(z_{\rrr},w_{\sss},u)=
\sum_{|\overline{\mu}|=d}
\frac{1}{\mathfrak{z}(\overline{\mu})}\jay_\rrr(z_\rrr,\overline{\mu},u,t)
\jay_\sss(z_\sss,\overline{\mu},u,-t).
$$

Combining the definition of $\jay$ (\ref{jaydef}) with Theorem \ref{HurwitzOperatorFormula} for $H$ gives:
$$ \jay_\rrr(z_\rrr,\overline{\mu},u,t)=
\frac{r^{\sum\leftover{\frac{\mu_i}{r}}-\iota(\rrr)}(|K|u/t)^{\ell(\mu)}}{t^{|\mu|/r+\#(r_i\notin K)+\ell(\rrr)-\iota(\rrr)}}
\left(\prod_{i=1}^{\ell(\overline{\mu})}\frac{\mu_i^\floor{\frac{\mu_i}{r}}}{\floor{\frac{\mu_i}{r}}!}\right) H^{0,\bullet}_{\rrr+\redge(\overline{\mu})}\left(\mu, tz_{\rrr}, \frac{u}{tr^{1/2}}\right)    \\
$$
$$ =\frac{r^{\sum\leftover{\frac{\mu_i}{r}}-\iota(\rrr)}(|K|u/t)^{-\ell(\rrr)}}{t^{|\mu|/r+\#(r_i\notin K)+\ell(\rrr)-\iota(\rrr)}}
\left(\prod_{i=1}^{\ell(\overline{\mu})}\frac{\mu_i^\floor{\frac{\mu_i}{r}}}{\floor{\frac{\mu_i}{r}}!}\right)
\left\langle
\prod_{i=1}^{\ell(\rrr)} \mathcal{A}_{\rrr_i}\left(tz_i,\frac{u}{t}\right)
\prod_{j=1}^{\ell(\overline{\mu})}\mathcal{A}_{\redge(\overline{\mu}_j)}\left(\mu_j, \frac{u}{t}\right)
\right\rangle.
$$

Define the operator $\pvac$ to be projection onto the vacuum vector.  Then, taking the adjoint of the operator definition of $\jay_\sss(z_\sss,\overline{\mu},u,-t)$, we can write $G^\bullet$ as a single vaccuum expectation:
 \begin{multline}  \label{Goperator}
G^\bullet_{d,\rrr,\sss}(z_{\rrr},w_{\sss},u) =  \\
\sum_{|\overline{\mu}|=d}
\frac{1}{\mathfrak{z}(\overline{\mu})}
\frac{r^{\sum\leftover{\frac{\mu_i}{r}}-\iota(\rrr)}\left(\frac{|K|u}{t}\right)^{-\ell(\rrr)}}
{t^{|\mu|/r+\#(r_i\notin K)+\ell(\rrr)-\iota(\rrr)}}
\left(\prod_{i=1}^{\ell(\overline{\mu})}\frac{\mu_i^\floor{\frac{\mu_i}{r}}}{\floor{\frac{\mu_i}{r}}!}\right)
\frac{s^{\sum\leftover{\frac{\mu_i}{s}}-\iota(\sss)}\left(\frac{|K|u}{-t}\right)^{-\ell(\sss)}}
{(-t)^{|\mu|/s+\#(s_i\notin K)+\ell(\sss)-\iota(\sss)}}
\left(\prod_{i=1}^{\ell(\overline{\mu})}\frac{\mu_i^\floor{\frac{\mu_i}{s}}}{\floor{\frac{\mu_i}{s}}!}\right)   \\
\bigg\langle
\prod \mathcal{A}_{\rrr_i}\left(tz_i,\frac{u}{t}\right)
\prod_{j=1}^{\ell(\overline{\mu})}\mathcal{A}_{\redge(\overline{\mu}_j)}\left(\mu_j, \frac{u}{t}\right)
\pvac
\left(\prod_{j=1}^{\ell(\overline{\mu})}\mathcal{A}_{\sedge(\overline{\mu}_j)}\left(\mu_j, \frac{u}{-t}\right)\right)^*
 \prod \mathcal{A}^*_{\sss_i}\left(-tw_i,-\frac{u}{t}\right)
\bigg\rangle .
\end{multline}

We introduce some definitions to simplify Equation \ref{Goperator}.  We first package everything pertaining to $\overline{\mu}$ into one operator:
\begin{multline} \label{Qop}
\Qop_d=\sum_{|\overline{\mu}|=d}
\frac{1}{\mathfrak{z}(\overline{\mu})}
r^{\sum\leftover{\frac{\mu_i}{r}}}t^{-|\mu|/r}
s^{\sum\leftover{\frac{\mu_i}{s}}}(-t)^{-|\mu|/s} \\
\left(\prod_{i=1}^{\ell(\overline{\mu})}\frac{\mu_i^\floor{\frac{\mu_i}{r}}}{\floor{\frac{\mu_i}{r}}!}\right)
\left(\prod_{i=1}^{\ell(\overline{\mu})}\frac{\mu_i^\floor{\frac{\mu_i}{s}}}{\floor{\frac{\mu_i}{s}}!}\right)
\prod_{j=1}^{\ell(\overline{\mu})}\mathcal{A}_{\redge(\overline{\mu}_j)}\left(\mu_j, \frac{u}{t}\right)
\pvac
\left(\prod_{j=1}^{\ell(\overline{\mu})}\mathcal{A}_{\sedge(\overline{\mu}_j)}\left(\mu_j, \frac{u}{-t}\right)\right)^*
.
\end{multline}

In addition, we modify the operators $\mathcal{A}$ to contain the appropriate prefactors.  We define
\begin{eqnarray}
\Atw^\gamma_{a/r}(z)&=& \label{Atw}
\frac{1}{t}\left(\frac{t}{r}\right)^{a/r}
\frac{t^{\delta_r(a)}}{|K|u}
\mathcal{A}^\gamma_{a/r}(tz,u/t) \\
&=& \frac{\left(t\gamma(-\kk)\right)^{a/r}}{t^{\delta^\vee_r(0)}|K|u}
\frac{tz}{(tz+a)}
\mathcal{S}(|R|uz)^{\frac{tz+a}{r}}\sum_{i=-\infty}^\infty \frac{\left(tz\mathcal{S}(|R|uz)\right)^i}{(1+\frac{tz+a}{r})_i}\mathcal{E}^\gamma_{ir+a}(|K|uz), \notag
\end{eqnarray}
and similarly,
$$\Atw_{r_i}(z)=\sum_{\gamma\in K^*} \gamma(-k_i)\Atw^\gamma_{a/r}(z).$$

With these definitions, we see that Equation \ref{Goperator} simplifies to:
\begin{equation*}
G^\bullet_{d,\rrr,\sss}(z_{\rrr},w_{\sss},u) =  \\
\bigg\langle
\prod \Atw_{\rrr_i}(z_i)
\Qop_d
 \prod \Atw^*_{\sss_i}(w_i)
\bigg\rangle.
\end{equation*}

We continue now by investigating the operator $\Qop_d$, and showing that it can simplify vastly.

Note that since the inner produce is Hermitian, and $(\alpha_n^{\gamma})^*=\alpha_{-n}^\gamma$, it follow immediately that $(\alpha_n^k)^*=\alpha_{-n}^{-k}$.

Now, by Equation (\ref{Aconj}), we have that
$$\mathcal{A}_{\redge(\overline{\mu}_j)}\left(\mu_j, \frac{u}{t}\right)
=r^{-\leftover{\frac{\mu_j}{r}}}
\frac{\floor{\frac{\mu_j}{r}}!}{\mu_j^{\floor{\frac{\mu_j}{r}}}}
\sum_{\gamma\in K^*}
\gamma(-\kk_0)^{\leftover{\frac{-\mu_j}{r}}}\gamma(-\redge(\overline{\mu_j}))
e^{\frac{t\alpha^\gamma_r}{u|R|}}
e^{\frac{u|K|}{t}\mathcal{F}^\gamma_2}
\alpha^\gamma_{-\mu_j}
e^{-\frac{u|K|}{t}\mathcal{F}^\gamma_2}
e^{-\frac{t\alpha^\gamma_r}{u|R|}}.
$$
The prefactors here will cancel with some of those in (\ref{Qop}).  Furthermore, recalling the definition of $\redge(\overline{\mu}_j)$:
$$
\redge(\overline{\mu}_j)=\left(-d(\overline{\mu}_j), -k_j-\mu_j\mathbb{L}+\floor{\frac{-\mu_j}{r}}\kk_0\right).
$$
we see
$$
\gamma(-\redge(\overline{\mu_j}))=\gamma(k_j)\gamma(\mathbb{L})^{\mu_j}\gamma(-\kk_0)^{\floor{\frac{-\mu_j}{r}}}.
$$
The last factor here can combine with one in (\ref{Qop}).  Furthermore, as the last two exponentials fix the vacuum vector, and operators with $\gamma\neq \gamma^\prime$ commute,  all of the exponentials will cancel except for an initial appearance of each for each $\gamma$.  Even if no terms corresponding to a given $\gamma$ appear, we can include the factor, as it will simply annihilate the vacuum.   Hence, we can group these exponents together into one factor of:
$$e^{\frac{t\alpha_r(0)}{u|R|}}e^{\frac{u}{t}\mathcal{F}_2^0}$$
  Similar arguments hold for the operators over $\infty$, and so, defining:
$$
\Protw_d=\prod_{j=1}^{\ell(\overline{\mu})}
\left(\sum_{\gamma\in K^*} \gamma(k_j)\gamma(\mathbb{L})^{\mu_j}\gamma(\kk_0)^{\frac{\mu_j}{r}} \alpha^\gamma_{-\mu_j}\right)
\pvac 
\left(
\prod_{j=1}^{\ell(\overline{\mu})}
\left(\sum_{\gamma\in K^*} \gamma(k_j)\gamma(\kk_\infty)^{\frac{\mu_j}{s}} \alpha^\gamma_{-\mu_j}\right)
\right)^*
$$
we have
\begin{equation*} \label{Qopexpansion}
\Qop_d=\sum_{|\overline{\mu}|=d}t^{-|\mu|/r}(-t)^{|\mu|/s}
\frac{1}{\mathfrak{z}(\overline{\mu})}
e^{\frac{t\alpha_r(0)}{u|R|}}e^{\frac{u}{t}\mathcal{F}_2^0}
 \Protw_d
 e^{\frac{u}{-t}\mathcal{F}^0_2}
e^{\frac{-t\alpha_{-s}(0)}{u|R|}} .
\end{equation*}

  The notation $\Protw_d$ stems from the fact that $\Protw_d$ will be a twisted version of $\Pro_d$, projection on to the energy $d$ eigenspace of $\mathcal{Z}_K$:
\begin{eqnarray*}
\Pro_d &=& \sum_{|\overline{\mu}|=d} \frac{1}{\mathfrak{z}(\overline{\mu})}\prod_{j=1}^{\ell(\mu)} \alpha_{-\mu_j}(-k_i) \pvac
\prod_{j=1}^{\ell(\mu)} \alpha_{\mu_j}(k_i) \\
&=&  \sum_{|\overline{\mu}|=d} \frac{1}{\mathfrak{z}(\overline{\mu})} \prod_{j=1}^{\ell(\mu)}
\left(\sum_{\gamma\in K^*} \gamma(k_j)\alpha^\gamma_{-\mu_j}\right) \pvac
\prod_{j=1}^{\ell(\overline{\mu})}
\left(\sum_{\gamma\in K^*} \gamma(-k_j)\alpha^\gamma_{-\mu_j}\right).
\end{eqnarray*}
Indeed, we see that apart from the $\gamma(\mathbb{L}), \gamma(\kk_0)$ and $\gamma(\kk_\infty)$ terms, this is exactly $\Protw_d$; if $\mathbb{L}=\kk_0=\kk_\infty$=0, then $\Protw_d=\Pro_d$.  Since these factors are exactly what capture the gerbe structure of $\mathcal{X}$ if our gerbe were trivial, with the trivial cocycle description, they would all be zero.  So the twisting of our projection operator corresponds to the twisting of the gerbe.

To understand this twisting better, it is convenient to understand the usual projection operator in terms of the decomposition of $\mathcal{Z}_K=\bigotimes \infwedge V^\gamma$:
\begin{eqnarray} \label{projectiondecomposition}
\Pro_d&=&\sum_{\sum d_\gamma=d}\quad \bigotimes_{\gamma\in K^*} \Pro^\gamma_{d_\gamma} \\
&=&\sum_{\sum d_\gamma=d} \quad \bigotimes_{\gamma\in K^*}
\left(\sum_{|\mu^\gamma|=d_\gamma} \prod \alpha^\gamma_{-\mu^\gamma_j}\pvac^\gamma \prod \alpha^\gamma_{\mu^\gamma_j}\right) \notag.
\end{eqnarray}
Now,  since our twisted projection operator differs from $\Pro_d$ by multiplying $\alpha_{-\mu_j}^\gamma$ by $\left(\gamma(\mathbb{L})(\gamma(\kk_0)/t)^{1/r}\right)^{\mu_j}$, and similarly with the operators over infinity, we see from (\ref{projectiondecomposition}) that:
$$
\Protw_d=\sum_{\sum d_\gamma=d}
\bigotimes_{\gamma\in K^*}
\gamma(\kk_0)^{d_\gamma/r}
\gamma(\mathbb{L})^{d_\gamma}
\gamma(\kk_\infty)^{d_\gamma/s}
 \Pro^\gamma_{d_\gamma}.
$$

Since $\Protw_d$ acts diagonally in the $v_{\overline{\lambda}}$ basis, and the operator $\mathcal{F}^0_2$ does as well, they commute.  Thus, expanding the $\alpha_{-r}(\kk_0)$ in terms of $\alpha^\gamma$ in Equation (\ref{Qop}), we have:
$$ \Qop_d =t^{-|\mu|/r}(-t)^{|\mu|/s}
e^{\frac{t\alpha_r(0)}{u|R|}}
 \Protw_d
e^{\frac{-t\alpha_{-s}(0)}{u|S|}}.
$$

Introducing
$$\Htw=\sum_d d\Protw_d$$
and defining
$$G^\bullet_{\rrr,\sss}(z_{\rrr}, w_{\sss}, u,q)=\sum_d G^\bullet_{d,\rrr,\sss}(z_{\rrr}, w_{\sss}, u)q^d,$$
we have that
\begin{equation*}
G^\bullet_{\rrr,\sss}(z_{\rrr},w_{\sss},u,q)=
 \left\langle \prod \Atw_{\rrr_i}(z_i)
e^{\frac{t\alpha_r(0)}{u|R|}}
\left(\frac{q}{t^{1/r}(-t)^{1/s}}\right)^{\Htw}
e^{\frac{-t\alpha_{-s}(0)}{u|S|}}
 \prod\Atw^*_{\sss_i}(w_i)\right\rangle.
\end{equation*}

Recall that $G^\bullet_{\rrr,\sss}$ includes, by definition, unstable contributions, and hence is not the true Gromov-Witten potential.  However, this is easily remedied.  The unstable contributions, defined in Equation (\ref{unstableG}) result from the degree 0, genus 0, one and two point functions, and hence all terms here include a $z_i$ or $w_i$ with a non-positive exponent.  Thus, if we restrict our attention to only positive powers of the variables, we will not include any unstable contributions, and hence recover the usual Gromow-Witten potential.

Denote by $\Atw_{\rrr}[i]=[z^{i+1}]\Atw_{\rrr}(z)$.  Then, we have
\begin{multline*}
\sum_{g\in \Z} \sum_{d\geq 0} u^{2g-2}q^d \left\langle \prod \tau_{k_i}(\zero_{\rrr_i})\prod \tau_{\ell_j}(\infty_{\sss_j})\right\rangle^\bullet_{g,d} \\
=
\left\langle \prod \Atw_{\rrr_i}[k_i]
e^{\frac{t\alpha_r(0)}{u|R|}}
\left(\frac{q}{t^{1/r}(-t)^{1/s}}\right)^{\Htw}
e^{\frac{-t\alpha_{-s}(0)}{u|S|}}
 \prod \Atw^*_{\sss_j}[\ell_j]\right\rangle.
\end{multline*}

Additionally, if we define
$$
\tau(x, x^*,u)=
\sum_{g\in \Z} \sum_{d\geq 0} u^{2g-2}q^d
\left\langle \exp\left(\sum x_i(\rrr)\tau_i(\zero_{\rrr})+\sum x_j^*(\sss) \tau_{j}(\infty_{\sss})\right)\right\rangle^\bullet_{g,d}
$$

then we have
$$
\tau(x, x^*,u)=
\left\langle e^{\sum x_i(\rrr)\Atw_{\rrr}[i]}
e^{\frac{t\alpha_r(0)}{u|R|}}
\left(\frac{q}{t^{1/r}(-t)^{1/s}}\right)^{\Htw}
e^{\frac{-t\alpha_{-s}(0)}{u|S|}}
 e^{\sum x^*_j(\sss)\Atw^*_{\sss}[j]}\right\rangle.
$$ 
\chapter{Decomposition and Integrable Hierarchies}
\section{Decomposition}

We now present a change of variables that expresses the $\tau$ function for $\mathcal{X}$ as a product of $\tau$ functions for $\mathcal{X}_{\text{eff}}=\mathcal{C}_{r,s}$.

Recalling that
$$\Atw_{(a/r,k)}(z)=\sum_{\gamma\in K^*} \gamma(-k)\Atw^\gamma_{a/r}(z),$$
so that
$$\Atw^\gamma_{a/r}(z)=\gamma(\kk_0)^{\delta_r^\vee(a)} $$
we define
$$y_i(a/r,\gamma)=\sum_{k\in K}\gamma(k)x_i(a/r,k),$$
so that
$$\sum_{k\in K} x_i(a/r,k)\Atw_{(a/r,k)}[i]
=\sum_{\gamma\in K^*} y_i(a/r,\gamma)\Atw^\gamma_{a/r}(z)[i].$$

Then, expressed in the $y$ variables, we have that
$$
\tau(y, y^*,u) =
\left\langle e^{\sum y_i(a/r,\gamma)\Atw^\gamma_{a/r}[i]}
\left(\sum_{\gamma\in K^*}e^{\frac{t\alpha_r^\gamma}{u|R|}}\right)
q^{\Htw}
\left(\sum_{\gamma\in K^*}e^{\frac{-t\alpha^\gamma_{-s}}{u|S|}}\right)
 e^{\sum y^*_j(b/s,\gamma)\Atw^{\gamma *}_{b/s}[j]}\right\rangle 
$$
$$
=
\prod_{\gamma\in K^*} \left\langle
e^{\sum y_i(a/r,\gamma)\Atw^\gamma_{a/r}[i]}
e^{\frac{t\alpha_r^\gamma}{u|R|}}
\left(q\gamma(\kk_0)^{1/r}\gamma(\kk_\infty)^{1/s}\gamma(\mathbb{L})\right)^{H_\gamma}
e^{\frac{-t\alpha^\gamma_{-s}}{u|S|}}
 e^{\sum y^*_j(b/s,\gamma)\Atw^{\gamma *}_{b/s}[j]}\right\rangle
$$
 
We can see decomposition on the operator level as follows: each factor in the product above differs only slightly from the operator expression for when $K=0$.  The factor of $q$ on each has been multiplied by $\gamma(\kk_0)^{1/r}\gamma(\kk_\infty)^{1/s}\gamma(\mathbb{L})$, $\Atw^\gamma_{a/r}[i]$ differs from $\Atw_{a/r}[i]$ by a factor of $\gamma(-\kk_0)^{a/r}$, and $u$ has been multiplied by $|K|$.

The first two factors together are exactly turning on discrete torsion, while the third factor is a physically meaningless ``dilaton shift.''

Given the decomposition, for the rest of the section we will work in the effective case.  To that extend, let
$$\Em=
e^{\sum x_i(a/r)\Atw_{a/r}[i]}
e^{\frac{t\alpha_r}{ur}}
q^{H}
e^{\frac{-t\alpha_{-s}}{us}}
 e^{\sum x^*_j(b/s)\Atw^*_{b/s}[j]},$$                         .
so that $\tau=\langle \Em\rangle$.  Then, to show that $\tau$ is a $\tau$ function of the 2-toda hiearchy, we must show that we can conjugate $\Em$ to the form $\Gamma_+(t)M\Gamma_-(s)$, for appropriate $M$.  This conjugation will give a linear change of variables relating the $x_i$ and $x_i^*$ variables of equivariant Gromov-Witten theory to the standard $t_i, s_i$ variables of the 2-Toda hierarchy.

However, before showing that this conjugation exists, we will derive an explicit form of the lowest equation of the hierarchy by hand.  An ingredient we will use in this derivation is the equivariant string equation, which we derive in the next section.

\section{Equivariant string and divisor equations}

The equivariant string equation will follow from the equivariant divisor equation.  Recall that our generating functions include unstable contributions, and so the usual proof would require modifying.  We will derive it from the operator formalism instead.

The equivariant divisor equation describes insertions of the class of a point with no psi insertions.  Suppose that $\rrr$ is an $n$ tuple, and let $\tilde{\rrr}$ be the $n+1$-tuple obtained by adding $0$ in the first position.  Then we have

\begin{proposition}$$
[z_0^1]G^\bullet_{d,\tilde{\rrr},\sss}(z_{\tilde{\rrr}},w_{\sss},u)=\left(d-\frac{1}{24}+t\sum z_i\right) G^\bullet_{d,\rrr,\sss}(z_\rrr,w_\sss,u).$$
\end{proposition}
\begin{proof}
Using the operator formula \ref{operatorG} for $G_d$, we see that:
$$
[z_0^1]G^\bullet_{d,\tilde{\rrr},\sss}(z_{\tilde{\rrr}},w_{\sss},u,q)=
 \left\langle \prod \Atw_0[0]\Atw_{\rrr_i}(z_i)
e^{\frac{t\alpha_r(0)}{ur}}
\Protw_d
e^{\frac{-t\alpha_{-s}(0)}{us}}
 \prod\Atw^*_{\sss_i}(w_i)\right\rangle,
$$

and so our first goal is to understand $\Atw_{0/r}[0]$, the coefficient of $z$ in $\Atw_0$.  By equation \ref{Atw}, we have
$$
 \Atw_{0/r}(z)=
\frac{1}{u}
\mathcal{S}(ruz)^{\frac{tz}{r}}
\sum_{i=-\infty}^\infty
\frac{\left(tz\mathcal{S}(ruz)\right)^i}{(1+\frac{tz}{r})_i}
\mathcal{E}_{ir}(uz).
$$

We see that $\Atw_{0/r}[0]$ will have contributions from $i\leq 1$.  In particular, examining the $i=1$ and $i=0$ terms, since $$\mathcal{E}_r(uz)=\alpha_r+O(z)$$ and
$$\mathcal{E}_0(uz)=\frac{1}{u}z^{-1}+C+(H-\frac{1}{24})uz+O(z^2),$$
we have
$$\Atw_{0/r}[0]= \frac{t}{u}\alpha_r+C-\frac{1}{24}+\dots$$
where the dots are terms that have positive energy, and thus have an adjoint that annihilates the vacuum.

Note that in the case $r=1$, this differs from the expression in \cite{OP2} in that $\alpha_1$ is multiplied by $\frac{t}{u}$, which agrees with the rescaling of the operators we have made.

Since the operators $C$ and $H$ both fix the vacuum, we can replace our operator with $\frac{t}{u}\alpha_r+H-\frac{1}{24}$, and so we have
$$
[z_0^1]G^\bullet_{d,\tilde{\rrr},\sss}(z_{\tilde{\rrr}},w_{\sss},u,q)=
 \left\langle \left(\frac{t}{u}\alpha_r+H-\frac{1}{24}\right) \prod (\Atw_{\rrr_i}(z_i)
e^{\frac{t\alpha_r(0)}{ur}}
\Protw_d
e^{\frac{-t\alpha_{-s}(0)}{us}}
 \prod\Atw^*_{\sss_i}(w_i)\right\rangle.
$$

Now, since
$$[\alpha_r, \mathcal{E}_{a+ir}(uz)]=\varsigma(urz)\mathcal{E}_{a+(i+1)r}(uz)$$
we have
\begin{eqnarray*}\left[\frac{t}{u}\alpha_r, \sum_{i=-\infty}^\infty \frac{\left(tz\mathcal{S}(ruz)\right)^i}{(1+\frac{tz+a}{r})_i}\mathcal{E}^\gamma_{ir+a}(uz)\right]
=\sum_{i=-\infty}^\infty (a+ri+tz)\frac{\left(tz\mathcal{S}(ruz)\right)^i}{(1+\frac{tz+a}{r})_i}\mathcal{E}^\gamma_{ir+a}(uz)
\end{eqnarray*}
where we have used the identity $(1+x+y)(1+x)_{y}=(1+x)_{1+y}$ and reindexed the sum.  Then, since
$$[H, \mathcal{E}_{a+ir}(uz)]=-(a+ir)\mathcal{E}_{a+ir}(uz)$$
it follows immediately that
$$
\left[\frac{t}{u}\alpha_r+H,\Atw_{a/r}(z)\right]=tz\Atw_{a/r}(z).
$$

Then, since $[H,\alpha_r]=-r\alpha_r$ and $H \Pro_d=d\Pro_d$, we have that
$$\left(\frac{t}{u}\alpha_r+H\right)e^{\frac{t\alpha_r}{ur}}=e^{\frac{t\alpha_r}{ur}}H
\Pro_d=de^{\frac{t\alpha_r}{u|R|}}
\Protw_d,$$
and so commuting $\frac{t}{u}\alpha_r+H$ to the center proves the result.
\end{proof}

The equivariant string equation describes insertions of the identity in equivariant cohomology with no psi insertions.  However, due to localization, we can express this in terms of insertions of $\zero(0,0)$ and $\finity(0,0)$:
$$1=\frac{\zero(0)-\finity(0)}{t}.$$

The following differential operator, then, inserts $\tau_0(1)$:
\begin{equation} \label{partial}
\partial=\frac{1}{t}\left(\frac{\partial}{\partial y_0(0)}-\frac{\partial}{\partial y^*_0(0)}\right).
\end{equation}

To obtain an explicit form for the lowest equation of the hierarchy, we will use the string equation in the following form
\begin{proposition} \label{string}
\begin{multline} \notag
 \left\langle e^{\tau_0(1)}\prod \tau_{k_i}(\zero(a_i/r))\prod \tau_{\ell_j}(\finity(b_j/s))\right\rangle^\bullet_{g,d} \\
=
\left[\prod z_i(a_i/r)^{k_i+1}\prod w_j(b_j/s)^{\ell_j+1}\right]e^{\sum z_i(a_i/r)+\sum w_j(b_j/s)}G^\bullet_{g,d}(z,w,u)
\end{multline}
\end{proposition}

\section{Explicit form of the lowest equation}
Recall that one form of the lowest equation in the 2-Toda hiearchy was:
$$
\left\langle T^{-1} M T\right\rangle \left\langle TMT^{-1} \right\rangle
=\left\langle M \right\rangle\left\langle \alpha_1M\alpha_{-1} \right\rangle-
\left\langle \alpha_1 M \right\rangle \left\langle M\alpha_{-1} \right\rangle, $$
and so, to find this equation for Gromov-Witten theory, we need the operator that contain $\alpha_1$.  Using the definition of $\Atw$ (\ref{Atw}) and the same reasoning in the previous section, we see that
$$[z]\Atw_{1/r}(z)=\frac{t^{1/r}}{u}\alpha_1+\dots$$
where the $\dots$ are terms of positive energy, and hence whose adjoint annihilates the vacuum.  We have also assumed that $r>1$, if $r=1$, then there is no $t$, and there is also a constant term.  Since $r=1$ is exactly the case treated by Okounkov and Pandharipande, we will assume from here that $r, s>1$.  In that case, we have, we have that
$$\frac{\partial}{\partial x_0(1/r)}\tau(x,x^*,u)=\langle (\frac{t^{1/r}}{u}\alpha_1) \Em\rangle$$
and
$$\frac{\partial}{\partial x^*_0(1/s)}\tau(x,x^*,u)=\langle  \Em(\frac{t^{1/s}}{u}\alpha_{-1})\rangle$$

and so we have
\begin{equation} \label{asdfasdf}
\tau\frac{\partial^2}{\partial x_0(1/r)\partial x^*_0(1/s)}\tau -\frac{\partial}{\partial x^*_0(1/s)}\tau\frac{\partial}{\partial x_0(1/r)}\tau=\frac{t^{1/r+1/s}}{u^2}\langle T^{-1}\Em T\rangle\langle T\Em T^{-1}\rangle,
\end{equation}
and so we must investigate the conjugation of $\Em$ by powers of $T$.

Now, $T^{-1}\mathcal{E}_r(z) T=e^{z}\mathcal{E}_r(z)$, as follows from the definition, with some care taken in the case of $\mathcal{E}_0$. From this, it is immediate that
$$T^{-1}\Atw_{a/r}(z)T=e^{uz}\Atw_{a/r}(z)$$.
Furthermore, it is clear from the definitions that $T$ commutes with the $\alpha_k$ and that
$$T^{-n}HT^n=H+nC+\frac{n^2}{2}.$$
Since $C$ commutes with the $\Atw$ and the $\alpha_k$ and annihilates the vacuum, its appearance will have no effect.

First, note that the effect of conjugating the $\Atw$ by $T$ will multiply each operator by $e^{zu}$, but, by our form of the string equation, Proposition \ref{string}, doing this is equivalent to applying the operator $e^{u\partial}$.  Furthermore, we can replace $T^{-n} q^{H}T^n$ with $q^{n^2/2}q^H$, and so we have that

\begin{equation} \label{FirstTeq}
\left\langle T^{-n}\Em T^n\right\rangle=q^{n^2/2}e^{nu\partial}\tau
\end{equation}
and so
$$\left\langle T^{-1}\Em T\right\rangle\left\langle T\Em T^{-1}\right\rangle=qe^{u\partial}\tau e^{-u\partial}\tau.$$

Putting this together with equation (\ref{asdfasdf}), and modifying using the same simplification of the 2-Toda equation as before, we get:

\begin{theorem}
Suppose, $r,s>1$.  Then the $\tau$ function satisfies the following 2-Toda equation:
$$\frac{\partial^2}{\partial x_0(1/r)\partial x^*_0(1/s)}\log \tau
=\frac{q t^{1/r+1/s}}{u^2}\frac{e^{u\partial}\tau e^{-u\partial}\tau}{\tau^2}.$$
\end{theorem}

\section{Change of Variables}

In this next two sections we show that $\Em$ can be conjugated to the required form, and show the resulting linear change of variables from the Gromov-Witten times to the standard 2-Toda times.  We will treat all our matrices as operators acting on $V$, not on $\infwedge V$.

More particular, we want to show that there exists an upper triangular matrix $W_r$, called the {\em dressing matrix}, so that

$$W_r^{-1}\exp\left(\sum x_i(a/r)\Atw_{a/r}[i]\right) W_r=\Gamma_+(t)$$
this gives a linear change of variables between the 2-Toda time variables $t$ and the Gromov-Witten variables $x_i(a/r)$.  Taking the adjoint replacing $t$ with $-t$, we set
$$W_r^{\dagger}=W_r^*|_{t\mapsto -t}$$
so that
$$W_r^{\dagger}\exp\left(\sum x^*_i(a/r)\Atw^*_{a/r}[i]\right)\left(W_r^\dagger\right)^{-1}=\Gamma_-(s)$$

Then, we will have that
$$\left\langle \Em\right\rangle = \left\langle W_r\Gamma_+(t)M\Gamma_-(s)W_s^{\dagger}\right\rangle,$$
with
$$M=W_r^{-1}e^{\frac{t\alpha_r}{ur}}
q^{H}
e^{\frac{-t\alpha_{-s}}{us}}\left(W_s^{\dagger}\right)^{-1}.$$

Since $W_r$ and $W_r^\dagger$ are upper triangular, we have
$$
W_r^* v_\emptyset=W_r^\dagger v_{\emptyset}= v_{\emptyset}.
$$
If, additionally, $W_r$ and $W_r^\dagger$ were unitriangular, we would have
\begin{equation} \label{unitriangular}
W_r^* T^n v_\emptyset=W_r^\dagger T^n v_{\emptyset}=T^n v_{\emptyset},
\end{equation}
which would imply
 $$\tau_n=\left\langle T^{-n}\Em T^n\right\rangle=\left\langle T^{-n}W_r\Gamma_+(t)M\Gamma_-(s)W_s^{\dagger}T^n\right\rangle=
\left\langle T^{-n}\Gamma_+(t)M\Gamma_-(s)T^n\right\rangle, $$
and hence that the $\tau_n$ were $\tau$ functions of the 2-Toda hierarchy.  We could, however, choose a $W_r$ that is upper triangular but not unitriangular.  This would leave $\tau_0$, our original Gromov-Witten $\tau$ function, unchanged.  However, this would change the functions $\tau_n$ by multiplying it by some function of $q,u,t$.  Note, though, that a priori the $\tau_n$ had nothing to do with Gromov-Witten theory; we related it to Gromov-Witten theory through equation (\ref{FirstTeq}): $\tau_n=q^{n^2/2}e^{nu\partial}\tau$.

Multiplying $W$ by a diagonal matrix, then, well keep $\tau_0$ unchanged, but give us a different change of variables to the standard 2-Toda times, and multiply $\tau_n$ by some function.

 We will be concerned only the operators $\Atw_{a/r}[k]$ for $k\geq 0$.  These have the form
\begin{equation*}
\Atw_{a/r}[k]=\left\{
\begin{array}{ll}
\frac{ t^{a/r}}{au}\alpha_{a+kr} +\dots & a\neq 0 \\
\frac{t}{u}\alpha_{(k+1)r} +\dots & a= 0
\end{array}
\right.
\end{equation*}
where the dots stand for terms of larger energy.    Hence, there exists an upper triangular matrix $W_r$ so that
$$W_r^{-1}\Atw_{1/r}[0] W=\alpha_1.$$  Note that $W$ is not unique - if we multiply $W_r$ by an element that commutes with $\alpha_1$, the result above would still hold.

Since the $\Atw_{a/r}[k]$ commute by (\ref{maincommutatorlemma}), and have the form above, if we define $$\Ac_{a/r}[k]=W\Atw_{a/r}[k]W^{-1}$$
then we must have
\begin{equation} \label{wewe}
\Ac_{a/r}[k]=\sum_{\ell\leq k+1} c_{a,k,\ell}(u,t)\alpha_{a+\ell r},
\end{equation}
since the $\Ac_{a/r}[k]$ must commute with $\Ac_{1/r}[0]=\alpha_1$.

In fact, we have the following lemma

\begin{lemma} \label{monomiallemma}
The coefficients $c_{a,k,\ell}(u,t)$ are monomials in $u, t$.
\end{lemma}

The proof of Lemma \ref{monomiallemma} is rather technical, and makes use of the hypergeometric function material found in the appendices.  We will postpone its proof until the next subsection.  However, once this lemma is in hand, the change of variables follows easily.  If the $c_{a,k,\ell}(u,t)$ are monomials, they are identical to their asymptotics as $u\to 0$.  So the full change of variables is equivalent to the change of variables in the $u\to 0$ limit.  But by Equation (\ref{Atw}), in the $u\to 0$ limit, we have:
$$\Atw_{a/r}(z)\sim
 \frac{\left(t\right)^{a/r}}{t^{\delta^\vee_r(0)}u}
\frac{tz}{(tz+a)}
\sum_{i=-\infty}^\infty \frac{(tz)^i}{(1+\frac{tz+a}{r})_i}\alpha_{ir+a}
$$
and in the $u\to 0$ limit the operator $W_r$ is diagonal, and we have, for $a\neq 0$:
$$
\sum_{k\geq 0} z^{k+1}\Ac_{a/r}[k]=\frac{t^{a/r}}{ut}\sum_{n\geq 0}\frac{(tz)^{n+1}}{\prod_{i=0}^n(i+\frac{tz+a}{r})}\alpha_{a+nr}
$$
and for $a=0$:
$$
\sum_{k\geq 0} z^{k+1}\Ac_{0/r}[k]=\frac{1}{u}\sum_{n\geq 1} \frac{(tz)^{n}}{\prod_{i=1}^n(i+\frac{tz}{r})}\alpha_{nr}.
$$

\section{Proof of the Monomial lemma}
We now present the proof of Lemma \ref{monomiallemma}, following \cite{OP2}.

When $tz=0\mod r$ and $tw=-a \mod r$, we have that:
$$\Atw_{0/r}(z)\Atw_{a/r}(w)=
\frac{\frac{tz}{r}! (\frac{tw+a-r}{r})!}{(\frac{tz+tw+a-r}{r})!}
\frac{(z+w)^{\frac{tz+tw+a-r}{r}}}{z^\frac{tz}{r}w^{\frac{tz+tw+a-r}{r}}}\Atw_{a/r}(z+w)
$$

From Equation (\ref{Atw}), we see that if we assign the grading
$$\deg u=\deg t=-\deg z=r,$$
then $\Atw_{a/r}$ is homogeneous of degree $a+(\delta_{a,0}-2)r$.  From this, we see that the operator $\Atw_{a/r}[k]$ has degree $a+kr-r(1-\delta_{a,0})$.

The first thing to note is that, since we are viewing our operators as acting on $V$, we can restate equation (\ref{wewe}) as
\begin{equation} \label{powerell}
\Atw_{a/r}[k]=\sum_{\ell\leq k} c_{a,k,\ell}(u,t)\Atw_{a/r}[0]\Atw_{0/r}[0]^\ell,
\end{equation}

and find, since $\Atw_{0/r}[0]$ has degree $0$, that the degree of $c_{a,k,\ell}(u,t)$ is $kr$.
Using this fact, we can begin to show the $c_{a,k,\ell}$ are monomials.
\begin{proposition} For $k\geq 0, \ell>0$ the coefficients $c_{a,k,\ell}$ are monomials.
\end{proposition}
\begin{proof}
Since everything is homogenous, we set $u=1$ for convenience.  By equation \ref{powerell}, it is enough to show that the expansion of $\Atw_{a/r}[k]$ into terms of the form $\Atw_{a/r}[0]\Atw_{0/r}[0]^j$ has monomial coefficients for $j\geq 0$.  Thus, by induction, it is enough to show that the coefficients $b_{a/r, k,\ell}(t)$ in the expansion
$$\Atw_{0/r}[0]\Atw_{a/r}[k]=\sum_{k\leq \ell+1} b_{a/r,k,\ell}(t)\Atw_{a/r}[\ell]$$
are monomials, or equivalently to find the coefficient of $z w^{\ell+1}$ in $\Atw_{0/r}(z)\Atw_{a/r}(w)$.

Expanding this product, we have:
\begin{multline} \Atw_{0/r}(z)\Atw{a/r}(w)=
\frac{t^{a/r}}{t^{\delta^\vee_r(a)}\left(|K|u\right)^2}
\frac{tw}{(tw+a)}
   \mathcal{S}(|R|uz)^{\frac{tz}{r}}\mathcal{S}(|R|uw)^{\frac{tw+a}{r}}\times \\
\sum_{i,j=-\infty}^\infty \frac{\left(tz\mathcal{S}(|R|uz)\right)^i}{(1+\frac{tz}{r})_i} \frac{\left(tz\mathcal{S}(|R|uw)\right)^j}{(1+\frac{tw+a}{r})_j}\mathcal{E}_{ir}(|K|uz)\mathcal{E}_{jr+a}(|K|uw).
\end{multline}

Using $\mathcal{E}_a(z)\mathcal{E}_b(w)=e^{(aw-bz)/2}\mathcal{E}_{a+b}(z+w)$, we can rewrite the second line as
\begin{multline*}
\sum_{m\in \Z} \mathcal{E}_{mr+a}\left(|K|u(z+w)\right)
\\ \times
\sum_{n\in \Z} e^{|K|u((m-n)rw-nrz-az))/2} \frac{z^{m-n}\mathcal{S}(|R|uz)^{m-n}}{(1+\frac{z}{r})_{m-n}}
\frac{w^n\mathcal{S}(|R|uw)^n}{(1+\frac{w+a}{r})_n} 
\end{multline*}
\begin{multline*}
=\sum_{m\in\Z}\frac{(\varsigma(|R|uz)/|R|u)^me^{|K|u(mwr-az)/2}}{(1+\frac{z}{r})_m}\mathcal{E}_{mr+a}\left(|K|u(z+w)\right)
\\ \times \sum_{n\in \Z}
\frac{(-\frac{z}{r}-m)_n}{(1+\frac{w+a}{r})_n}\left(\frac{1-e^{-|R|uw}}{1-e^{|R|uz}}\right)^n
\end{multline*}
where we have used $\frac{1}{(1+\frac{z}{r})_{m-n}}=(-1)^n\frac{(-z/r-m)_n}{(1+\frac{z}{r})_m}$, and expanded the $\mathcal{S}$.  Now, the second sum over $\mathcal{S}$ can be expressed in terms of Gauss's hypergeometric functions:
\begin{equation} \label{brisketchile}
\Gauss{-\frac{z}{r}-m}{1+\frac{w+a}{r}}{\frac{1-e^{-|R|uw}}{1-e^{|R|uz}}}+
\Gauss{-\frac{w+a}{r}}{1+\frac{z}{r}+m}{\frac{1-e^{|R|uz}}{1-e^{-|R|uw}}}-1.
\end{equation}

The hypergeometric series converge when the argument has size less than one, hence the first converges for $|w|<|z|\ll 1$, while the second converges for $|z|<|w|\ll 1$.  Therefore, we can find the coefficient we want as the sum of a contour integrals over two separate domains.  However, if $m>0$, then each function will converge for $|z|=|w|=\epsilon \ll 1$, and we deform both integrals to this common contour.  Then, Lemma \ref{lemma16} applies directly, and we can replace (\ref{brisketchile}) with:

$$\frac{(1-v)^{m+\frac{z+w+a}{r}}}{(-v)^{\frac{w+a}{r}}} \frac{\Gamma(1+\frac{w+a}{r})\Gamma(1+\frac{z}{r})}{\Gamma(1+\frac{z+w+a}{r})}$$
where
$$v=\frac{1-e^{-|R|uw}}{1-e^{|R|uz}}.$$
For cancellation purposes, it is convenient to further rewrite (\ref{brisketchile}) with:
$$e^{|K|u(az-wm)/2}
\frac{\varsigma(|R|u(w+z))^m}{(1+\frac{z+w+a}{r})_m}
\frac{(1+\frac{z}{r})_m}{\varsigma(|R|uz)^m}
\frac{\varsigma(|R|u(w+z))^{\frac{z+w+a}{r}}}{\varsigma(|R|uw)^{\frac{w+a}{r}}\varsigma(|R|uz)^{\frac{z}{r}}}
\frac{\Gamma(1+\frac{w+a}{r})\Gamma(1+\frac{z}{r})}{\Gamma(1+\frac{z+w+a}{r})}
$$
where we have used $\Gamma(1+x+m)=\Gamma(1+x)(1+x)_m$.

Substituting this in, we see that

\begin{multline}
\Atw_{0/r}[0]\Atw_{a/r}[k]=\frac{1}{(2\pi i)^2}\iint_{|z|=|w|=\epsilon}\frac{dzdw}{z^2w^{k+2}}\times
\\
\frac{\left(1+\frac{w}{z}\right)^{\frac{tz+tw+a-r}{r}}}{\left(\frac{w}{z}\right)^{\frac{tw+a-r}{r}}}
\frac{\Gamma(\frac{w+a}{r})\Gamma(1+\frac{z}{r})}{\Gamma(\frac{z+w+a}{r})}\Atw_{a/r}(z+w)+\cdots
\end{multline}
here the terms need some explanation.

The fractional powers of $\frac{w}{z}$ are defined using the cut $w/z\notin (-\infty, 0]$.  Since $w,z$ are small on the contour of integration, the singularitiy at $z=-w$ is still integrable.  Finally, the negative energy terms of $\Atw_{a/r}(z+w)$ are nonsingular at 0, and so have unambiguous extension.  The $\cdots$ represent terms of non-negative energy, since the expansion of $\Atw_{a/r}$ here is ambiguous.

Expanding $(z+w)^{\ell+1}$ by the binomial coefficient, using the definitions, and inserting a few factors of $t$ that cancel we see that this gives:
\begin{multline}
b_{a/r,k,\ell}(t)=\frac{1}{(2\pi i)^2}\sum_{a=0}^{\ell+1}\binom{\ell+1}{a}\iint_{|z|=|w|=\epsilon}\frac{dzdw}{z^{2-a}w^{k+a+1-\ell}}\times
\\
\frac{\left(1+\frac{tw}{tz}\right)^{\frac{tz+tw+a-r}{r}}}{\left(\frac{tw}{tz}\right)^{\frac{tw+a-r}{r}}}
\frac{\Gamma(\frac{tw+a}{r})\Gamma(1+\frac{tz}{r})}{\Gamma(\frac{tz+tw+a}{r})}
\end{multline}
which is indeed a monomial in $t$ of degree $1+k-\ell$.
\end{proof}

We can deduce that all the $c_{a, k,\ell}$ are monomials for all $\ell$ from the case for positive $\ell$ as follows.  Define the operator
$$\D_r=W_r^{-1}\left(\frac{t}{u}\alpha_r+H-\frac{1}{24}\right)W_r.$$

Then, since
$$\left[\frac{t}{u}\alpha_r+H, \Atw_{a/r}(z)\right]=tz\Atw_{a/r}(z),$$
we have that
\begin{equation} \label{drdetermine}
\left[\D_r, \Ac_{a/r}[k]\right]=t\Ac_{a/r}[k-1].
\end{equation}
Furthermore, since $\frac{t}{u}\alpha_r+H-\frac{1}{24}$ is exactly the nonpositive energy part of $\Atw_{0/r}[1]$, we have that $\D_r=\alpha_r$ plus terms of positive energy.

Since $[\D_r, \alpha_1]$ commutes with $\alpha_1$, we see that $\D_r$ must have the form
$$\D_r=\alpha_r+\sum_{n>0} d_{r,n}(u,t)H\alpha_{-n}.$$

The coefficients $d_{r,n}$ are uniquely determined by \ref{drdetermine} from the $c_{a,k,\ell}$ with $\ell>0$.  These, in turn, determine the rest of the $c_{a,k,\ell}$.

\chapter{Proofs of Technical Lemmas}
\label{appendix}

\section{Convergence of $\mathcal{A}$} \label{convergence}
We follow the arguments of the \cite{OP2} closely, but give a complete and self contained presentation.

The main result we will want is that, for all compact $K\subset\Omega$, and all $u$ sufficiently small, the matrix elements
$$(\prod_{i=1}^n\mathcal{A}^{\gamma_i}_{a_i/r}(z_i,u)
v_{\overline{\nu}}, v_{\overline{\lambda}})$$
converge absolutely and uniformly.

First, we note that we can really work with just one $\gamma$ at a time - the operators for different $\gamma_i$ commute, and so the matrix element above will break into
$$\prod_{\gamma\in K^*}(\prod_{i=1}^{n_\gamma}v_{\overline{\nu}}^\gamma, v_{\overline{\lambda}}).$$

\subsection{Review of Lemmas}

In this section we present the proofs of several lemmas we needed to show convergence.  They are essentially directly from \cite{OP2}, but we have included them, with slightly expanded proofs, for completeness.

\begin{lemma} \label{lemma4}
 Let $\nu$ be a partition of $k$.  Then for any $l$, there are at most $k+l+1$ partitions $\lambda$ of $l$ with
$$(\mathcal{A}^\gamma_{a/r}(z)v_\nu, v_\lambda)\neq 0.$$
\end{lemma}
\begin{proof}
The operator $\mathcal{A}^\gamma_{a/r}(z,u)$ is a weighted sum of operators $\mathcal{E}^\gamma_m(z)$, which add or subtract a border strip of size $|m|$.  There are at most $|\nu|$ ways of subtracting a border strip, in case $\nu$ is larger, and $m+\ell(\nu)\leq |\lambda|$  ways of adding a border strip of size $m$, in case $\lambda$ is larger.  The operator $\mathcal{E}^\gamma_0(z)$ acts diagonally in the $v_\nu$ basis, and we have included the $+1$ term to handle the case $\nu$ and $\lambda$ are both empty.
\end{proof}

The following lemma, which appears in \cite{OP2} as Lemma 5, bounds the size of the matrix elements of the operators $\mathcal{E}_r(z)$:

\begin{lemma} \label{lemma5}
For any two partitions $\nu, \lambda$, if $|\nu|\neq|\lambda|$, then
$$
|(\mathcal{E}_{|\nu|-|\lambda|}(z)v_\nu, v_\lambda)|\leq \exp\left(\frac{|\nu|+|\lambda|}{2}|z|\right).
$$
If $|\nu|=|\lambda|$, but $\nu\neq\lambda$, then
$$(\mathcal{E}_0(z)v_\nu,v_\lambda)=0,$$
and if $\nu=\lambda$ we have
$$
\left|(\mathcal{E}_0(z) v_\nu, v_\lambda)-\frac{1}{\varsigma(z)}\right|\leq |\nu|\exp\left(|\nu||z|\right).
$$
\end{lemma}
\begin{proof}
Recall the definition:
$$
\mathcal{E}_r(z)=\sum_{k\in\Z+\frac{1}{2}}e^{z(k-\frac{r}{2})}E_{k-r,k}+\frac{\delta_{r,0}}{\varsigma(z)}.
$$

Recall also the action of $E_{k-r,k}$ in the Maya diagram representation of $\nu$.  When $r\neq 0$, this action is as follows.
If the circle marked $k$ is empty, or the circle marked $k-r$ is full,
then $E_{k-r,k}v_\nu=0$.  Otherwise, $E_{k-r, k}$ sends $v_\nu$ to $\pm v_\mu$, where $\mu$ is the partition whose Maya
diagram is obtained from that of $\nu$ by moving the stone in circle $k$ to circle $k-r$.

For $r=0$, the operator $E_{k,k}$ acts by zero if the $k$th circle of the Maya diagram for $\lambda$ is unchanged from the Maya
diagram for the empty partition, and acts by the sign of $k$ if the $k$th circle has been changed from that of the vacuum.  Thus, we see that $\mathcal{E}_0(z)$ acts diagonally in the $v_\lambda$ basis, and the lemma holds if $|\lambda|=|\nu|$ but $\lambda\neq\nu$.

Setting $r=|\nu|-|\lambda|$, we see that there if $r\neq 0$, there is at most one $k$ such that $E_{k-r, k}$ sends $v_\nu$ to $v_\lambda$.  For this $k$, consider the circles marked $k$ and $k-r$ on the Maya diagrams for $\nu$ and $\lambda$.  At $k$, exactly one of $\mu$ and $\lambda$ must have a stone; at $k-r$, the other of $\nu$, $\lambda$ will have a stone.  Thus, regardless of whether $k, k-r$ are positive or negative, at these two spots one of $\nu$ or $\lambda$ will differ from the vacuum vector.  If $k, k-r$ have the same sign, each partition will agree with the vacuum vector at one spot and differ at the other, while if $k$ and $k-r$ have the same signs, then one of the partitions will differ from the vacuum vector at both spots.

When the Maya diagram for $\mu$ differs from the vacuum vector at spot $k$, in the corresponding Russian diagram of $\mu$ lying above the Maya diagram there is a strip of cells lying directly above the interval from $k$ to $0$ that contains $|k|+1/2$ cells.  So, if $k$, $k-r$ have the same sign then by the preceding paragraph we have $|k|+|k-r|<|\nu|+|\lambda|$, by taking the corresponding strips.  If, on the other hand, $k$ and $k-r$ have the same sign, then inside a single partition we have two strips of sizes $|k|+1/2$ and $|k-r|+1/2$.  These two strips overlap in exactly one square, the one directly above the origin.  So we have $|k|+|k-r|\leq |\nu|+|\lambda|$, with equality occurring when one of $\nu, \lambda$ is a border strip and the other one is empty.

So regardless of the signs of $k$ and $k-r$, we have
 $$\left|k-\frac{r}{2}\right|
\leq \left|\frac{k}{2}\right|+\left|\frac{k-r}{2}\right|\leq \frac{|\lambda|+|\nu|}{2}
$$
and so if $r\neq 0$ the lemma holds.

Finally, in case $\nu=\lambda$, we note that, for any circle $k$ that so that $E_{k,k}v_\nu\neq 0$, there is a box in the Russian diagram above the Maya diagram.  These boxes are distinct, except possibly the box above the origin.  So, unless $\nu=\Box$ is the partition of $1$, we see that the number of $k$ with $E_{k,k}v_\nu\neq 0$ is less than or equal to $|\nu|$, and as for each such $k$ we have $|k|<|\nu|$, we have that the lemmas holds in this case.

Finally, in case $\nu=\Box$, we have that
$$|(\mathcal{E}_0(z)v_\Box, v_\Box)-\frac{1}{\varsigma(z)}|=|e^{z/2}-e^{-z/2}|< e^{|z|}.$$
\end{proof}

\begin{lemma}   \label{sumlemma}
For fixed $k_0, k_n\in\Z$, the series
$$\sum_{k_1,\dots, k_n\geq 0}\prod_{i=1}^n\frac{z_i^{k_i-k_{i-1}}}{(d_i)_{k_i-k_{i-1}}}$$
converges absolutely and uniformly on compact subsets $K\subset\Omega$ for all values of $d_i\notin\Z$.
\end{lemma}

\begin{proof}
If we factor out the terms including $k_1$, what we have is:
\begin{equation} \label{1varterm}
\sum_{k_1\geq 0} \frac{(z_2/z_1)^{k_1}}{(d_1)_{k_1-k_0}(d_2)_{k_2-k_1}}
\end{equation}

Since the $d_i\notin \Z$, there are no poles from the denominator, and since on $\Omega z_2/z_<1$, by the ratio test the factor (\ref{1varterm}) converges absolutely and uniformly on compact sets.  We study this function with respect to $k_2$.

Consider the series
\begin{equation} \label{tempseries}
\sum_{k_1\geq 0} \frac{w^{k_1}}{(1)_{k_1+1}(1)_{k_2-k_1+1}}=
\sum_{k_1=0}^{k_2} \frac{w^{k_1}}{k_1!(k_2-k_1)!}+\sum_{k_1>k_2} \frac{(k_1-k_2)!w^{k_1}}{k_1!}
\end{equation}
 with $w=|z_1/z_2|$.

On the one hand, taking $n$ derivatives with respect to $w$ changes the series to
$$
\sum_{k_1\geq n} \frac{w^{k_1-n}}{(1)_{k_1-n+1}(1)_{k_2-k_1+1}}=\sum_{k_1\geq 0}\frac{w^{k_1}}{(1)_{k_1+1}(1)_{(k_2-n)-k_1+1}},
$$
and so for $n$ large enough series (\ref{tempseries}) bounds series (\ref{1varterm}).

On the other hand, using the right hand side of (\ref{tempseries}), we see by the binomial theorem that the first sum is equal to
$$
\frac{(1+w)^{k_2}}{k_2!}
$$
and that we can bound the second series by
$$
\frac{1}{k_2!}\frac{w^{k_2+1}}{1-w}.
$$
So on compact sets inside $\Omega$, the sum with respect to $k_1$ behaves like
$$\frac{1}{k_2!}\left(\frac{|z_1|+|z_2|}{|z_2|}\right)^{k_2}.$$
Then, the sum of the original series with respect to $z_1$ and $z_2$ behaves like
$$\sum_{k_2\geq 0} \frac{1}{k_2!}\left(\frac{|z_1|+|z_2|}{|z_2|}\right)^{k_2}\frac{(|z_2|/|z_3|)^{k_2}}{(d_3)_{k_3-k_2}}
=
\sum_{k_2\geq 0} \frac{1}{1_{k_2+1}(d_3)_{k_3-k_2}}\left(\frac{|z_1|+|z_2|}{|z_3|}\right)^{k_2}
$$
which is of the form (\ref{1varterm}), and so converges absolutely and uniformly since on $\Omega, |z_1|+|z_2|<|z_3|$.  Iterating this argument, the lemma is proven.
\end{proof}

\subsection{Proof of main convergence result}

First, we expand the sum as a sum over all intermediate partitions, $\lambda=\mu^0, \mu^1,\dots, \mu^n=\nu$ (note that due to a typo in \cite{OP2}, this is the opposite order of the partitions present there) :
\begin{eqnarray*}
\left(\prod_{i=1}^n \mathcal{A}^\gamma_{a_i/r}(z_i,u)v_\nu, v_\lambda\right)&=&\sum_{\nu=\lambda^0, \lambda^1,\dots, \lambda^{n}=\lambda}
\prod_{i=1}^n\left(\mathcal{A}^\gamma_{a_i/r}(z_i,u)v_{\lambda^{i}},v_{\lambda^{i-1}}\right)
\end{eqnarray*}

Let $|\lambda^i|=k_i$, and define $b_i$ by $b_i/r=\leftover{\frac{k_n-\sum_{j>i}a_j}{r}}$.  Then the matrix element above is zero unless $k_i=b_i\mod r$, and so we define $\ell_i$ so that $k_i=b_i+r\ell_i$.  Then we have that
$$\frac{k_i-k_{i-1}}{r}=\ell_i-\ell_{i-1}+\frac{b_i-b_{i-1}}{r}=\ell_i-\ell_{i-1}-\delta(b_i<b_{i-1})+\frac{a_i}{r}.$$

Now, for fixed $\lambda^{i}$ with $|\lambda^{i}|=k_{i}$, we consider the term $\left(\mathcal{A}^\gamma_{a_i/r}(z_i,u)v_{\lambda^{i}},v_{\lambda^{i-1}}\right)$ appearing in the product above.  Fixing a $k_{i-1}$, we expand the definition of $\mathcal{A}$, taking the one relevant $\mathcal{E}$ term.  Bounding the number of possible $\lambda^{i-1}$ using Lemma \ref{lemma4}, and the matrix element of $\mathcal{E}$ by Lemma \ref{lemma5}, we have:
\begin{multline}
\left|\sum_{|\lambda^{i-1}|=k_{i-1}}\left(\mathcal{A}^\gamma_{a_i/r}(z_i,u)v_{\lambda^{i}}, v_{\lambda^{i-1}}\right)\right|
\leq \\
 (k_i+k_{i-1}+1)\left|
\frac{z_i(ur)^{a_i/r}}{z_i+a}\mathcal{S}(|R|uz_i)^{\frac{z_i+a_i}{r}}
\frac{\left(\frac{\gamma(\kk)}{|R|}\varsigma(|R|uz_i)\right)^{p_i}}{(1+\frac{z_i+a_i}{r})_{p_i}
}\exp\left(\frac{k_i+k_{i-1}}{2}|K|u|z_i|\right)\right|
\end{multline}
with $p_i=\frac{k_{i}-k_{i-1}-a_i}{r}=\ell_i-\ell_{i-1}-\delta_i$, and $\delta_i=\delta(b_i<b_{i-1})$.

We have ignored the $\frac{1}{\varsigma(z)}$ terms appearing in $\mathcal{E}_0(z)$, as they are analytic for $u\neq 0$ and will not affect convergence.  Similarly, from here on we will ignore the $\frac{z_i}{z_i+a}$ and $\mathcal{S}$ terms, for as $u\to 0, \mathcal{S}(|R|uz_i)\to 1$, and so for small $u, \mathcal{S}(|R|uz_i)^{\frac{z_i+a_i}{r}}$ will be single valued and analytic, and so will also have no effect on convergence.  Furthermore, we can rewrite the above in terms of the $\ell_i$.

As a result, we see must consider sums of the form
$$
\sum_{\ell_i\geq 0}p(\ell_i)(ur)^{a_i}e^{\frac{|K|z_iu}{2}(k_{i-1}+k_i)}
\frac{\left(\varsigma(|R|uz_i)/|R|\right)^{p_i}}{(1+\frac{z_i+a_i}{r})_{p_i}}
$$

The term raised to the power $k_i$ is
$$
\exp\left(\frac{|K|u(|z_i|+|z_{i+1}|)}{2}\right)\left(\frac{\varsigma(|R|uz_{i})}{\varsigma(|R|uz_{i+1})}\right)^{1/r},
$$
and, as $u\to 0$, we see that this goes to $(z_{i-1}/z_i)^{1/r}$.  Thus, eliminating the irrelevant polynomial factor, it is enough to consider the sum:

$$\sum_{\ell_i\geq 0} \frac{\left(\frac{z_{i-1}}{z_i}\right)^{\ell_i-\ell_{i-1}-\delta_i}}{(1+\frac{z_i+a_i}{r})_{\ell_i-\ell_{i-1}-\delta_i}}.$$

The $\delta_i$ terms can be absorbed into the $\ell_i$ to give a sum of the form of Lemma \ref{sumlemma}, with some terms missing, taking $d_i=1+\frac{z_i+a_i}{r}$.

\section{Commutators of $\mathcal{A}$} \label{maincommutatorlemmaproof}

This section presents the proof of Lemma \ref{maincommutatorlemma}:
$$
[\mathcal{A}^{\gamma}_{a/r}(z,u),\mathcal{A}^{\gamma^\prime}_{b/r}(w,u)]=
\delta_{\gamma, \gamma^\prime}
\delta_{a/r, \op{b}/r}
\gamma(\kk)^{-\delta^\vee_r(a)}
\frac{zw}{r}
\delta(z, -w).
$$

\subsection{The commutators as a hypergeometric function}
From the definition of $\mathcal{A}^\gamma_{a/r}$, it is clear that $[\mathcal{A}^{\gamma}_{a/r}(z,u),\mathcal{A}^{\gamma^\prime}_{b/r}(w,u)]$=0 if $\gamma\neq \gamma^\prime$, so we assume $\gamma^\prime=\gamma$.   Expanding the commutator using the definition of $\mathcal{A}$ (\ref{defA}) gives:
\begin{multline} \label{comAexpand}
[\mathcal{A}^\gamma_{a/r}(z,u),\mathcal{A}^\gamma_{b/r}(w,u)]
=\left(r\gamma(-\kk)\right)^{a+b/r}
\frac{z}{z+a}\frac{w}{w+b}
\mathcal{S}(|R|uz)^\frac{z+a}{r}
\mathcal{S}(|R|uw)^\frac{w+b}{r} \\
\sum_{i,j\in\Z}
\frac{(z\mathcal{S}(|R|uz))^i}{(1+\frac{z+a}{r})_i}
\frac{(w\mathcal{S}(|R|uw))^j}{(1+\frac{w+b}{r})_j}
[\mathcal{E}^\gamma_{ir+a}(|K|uz),\mathcal{E}^\gamma_{jr+b}(|K|uw)]
\end{multline}

By (\cite{comE}) we know that
\begin{equation}
[\mathcal{E}^\gamma_{ir+a}(|K|uz),\mathcal{E}^\gamma_{jr+b}(|K|uw)]=
\varsigma\big(u|K|((ir+a)w-(jr+b)z)\big)
\mathcal{E}^\gamma_{(i+j)r+a+b}(u|K|(z+w))
\end{equation}

so reindexing with $i+j=n$, we see that the second line of (\ref{comAexpand}) can be written
\begin{equation} \label{cintroduction} 
\sum_{n\in\Z}
c_{n,a,b}(z,w)
\mathcal{E}^\gamma_{nr+a+b}(u|K|(z+w))
\end{equation}

with
\begin{equation}  \label{cdef}    
c_{n,a,b}(z,w)=\sum_{i\in\Z}
\frac{(z\mathcal{S}(|R|uz))^i}{(1+\frac{z+a}{r})_i}
\frac{(w\mathcal{S}(|R|uw))^{n-i}}{(1+\frac{w+b}{r})_{n-i}}
\varsigma(u|K|((ir+a)w-((n-i)r+b)z)).
\end{equation}
Note that $c_{n,a,b}(z,w)$ is independent of $\kk$, so that we have separated out the dependence on the group extension that forms $R$.

We assume that $n=2m$ is even; an analogous argument works when $n-2m-1$ is odd.  Temporarily suppressing the last $\varsigma$ factor, we can rewrite Equation \ref{cdef} as:
\begin{equation}
\frac{(z\mathcal{S}(|R|uz))^m}{(1+\frac{z+a}{r})_m}
\frac{(w\mathcal{S}(|R|uw))^m}{(1+\frac{w+b}{r})_m}
\sum_{i\in\Z}
\frac{(1+\frac{z+a}{r})_m}{(1+\frac{z+a}{r})_i}
\frac{(1+\frac{w+b}{r})_m}{(1+\frac{w+b}{r})_{n-i}}
\left(\frac{z\mathcal{S}(|R|uz)}{w\mathcal{S}(|R|uw)}\right)^{i-m}
\end{equation}

By reindexing $j=i-m$ and using
\begin{equation}
\frac{(1+x)_d}{(1+x)_e}=\frac{1}{(1+x+d)_{e-d}}=(-1)^{e-d}(-x-d)_{d-e}
\end{equation}

 the sum becomes
\begin{equation}
\sum_{j\in\Z}
\frac{(-\frac{w+b}{r}-m)_j}{(1+\frac{z+a}{r}+m)_j}
(-1)^j
\left(\frac{\varsigma(|R|uz)}{\varsigma(|R|uw)}\right)^j,
\end{equation}
where we have also converted the $\mathcal{S}$ to $\varsigma$.

Factoring back in the $\varsigma$ and expanding $\varsigma(x)=e^{x/2}-e^{-x/2}$ we get two sums, and
since
$$u|K|((ir+a)w-((n-i)r+b)z)=u|K|(rj(w+z)+aw-bz+rmw-rmz),$$
 we can factor the $j$ terms into the sum, so that we have
\begin{multline}
e^{\frac{u|K|}{2}((rm+a)w-(rm+b)z)}\sum_{j\in\Z}
\frac{(-\frac{w+b}{r}-m)_j}{(1+\frac{z+a}{r}+m)_j}
\left(\frac{1-e^{|R|uz}}{1-e^{-|R|uw}}\right)^j \\
-e^{\frac{u|K|}{2}((rm+b)z-(rm+a)w)}
\sum_{j\in\Z}
\frac{(-\frac{w+b}{r}-m)_j}{(1+\frac{z+a}{r}+m)_j}
\left(\frac{1-e^{-|R|uz}}{1-e^{|R|uw}}\right)^j
\end{multline}

Recalling the definition of Gauss's hypergeometric function
\begin{equation}
\Gauss{x}{y}{z}=\sum_{i=0}^\infty\frac{(x)_i}{(y)_i}z^i
\end{equation}

we see that the $j\geq 0$ part of either sum gives us us a hypergeometric function.  Taking the $j\leq 0$ sum of the first line, we get, by substituting $i=-j$,
\begin{multline}
e^{\frac{u|K|}{2}((rm+a)w-(rm+b)z)}\sum_{j\leq0}
\frac{(-\frac{w+b}{r}-m)_j}{(1+\frac{z+a}{r}+m)_j}
\left(\frac{1-e^{|R|uz}}{1-e^{-|R|uw}}\right)^j \\
=e^{\frac{u|K|}{2}((rm+a)w-(rm+b)z)}\sum_{i\geq 0}
\frac{(-\frac{z+a}{r}-m)_i}{(1+\frac{w+b}{r}+m)_i}
\left(\frac{1-e^{-|R|uw}}{1-e^{|R|uz}}\right)^i
\end{multline}

A similar expression holds for the first line.  We're using the $j=0$ contribution twice, so we also need to subtract off that, getting
\begin{eqnarray} \label{line1}
e^{\frac{u|K|}{2}((rm+a)w-(rm+b)z)}&\cdot\bigg[&
\Gauss{-\frac{w+b}{r}-m}{1+\frac{z+a}{r}+m}{\frac{1-e^{|R|uz}}{1-e^{-|R|uw}}}\\ \label{line2}
&+&\Gauss{-\frac{z+a}{r}-m}{1+\frac{w+b}{r}+m}{\frac{1-e^{-|R|uw}}{1-e^{|R|uz}}} \bigg] \\ \label{line3}
-e^{\frac{u|K|}{2}((rm+b)z-(rm+a)w)}&\cdot\bigg[&
\Gauss{-\frac{w+b}{r}-m}{1+\frac{z+a}{r}+m}{\frac{1-e^{-|R|uz}}{1-e^{|R|uw}}}\\ \label{line4}
&+&\Gauss{-\frac{z+a}{r}-m}{1+\frac{w+b}{r}+m}{\frac{1-e^{|R|uw}}{1-e^{-|R|uz}}} \bigg] \\
&-&\varsigma(u|K|((rm+a)w-(rm+b)z)) \label{line5}
\end{eqnarray}

We repackage this slightly by noting that interchanging $(z, a)$ and $(w,b)$ swaps (\ref{line1}) with (\ref{line4}) and (\ref{line2}) with (\ref{line3}), and multiplies (\ref{line5}) by $-1$.  This leads us to define:
\begin{multline}\label{fdef}
f_{m,u}(z,k,a,b)=e^{\frac{u|K|}{2}((rm+a)w-(rm+b)z)}
\Gauss{-\frac{w+b}{r}-m}{1+\frac{z+a}{r}+m}{\frac{1-e^{|R|uz}}{1-e^{-|R|uw}}} \\
-e^{\frac{u|K|}{2}((rm+b)z-(rm+a)w)}
\Gauss{-\frac{w+b}{r}-m}{1+\frac{z+a}{r}+m}{\frac{1-e^{-|R|uz}}{1-e^{|R|uw}}} \\
-\frac{\varsigma(u|K|((rm+a)w-(rm+b)z))}{2}
\end{multline}

Putting it all together, we have shown that when $n=2m$ we have
\begin{equation} \label{evenprefactors}
c_{n,a,b}(z,w)=
\frac{(\varsigma(|R|uz)/|R|)^m}{(1+\frac{z+a}{r})_m}
\frac{(\varsigma(|R|uw)/|R|)^m}{(1+\frac{w+b}{r})_m}
[f_{m,u}(z,a,w,b)-f_{m,u}(w,b,z,a)]
\end{equation}

A similar argument shows that when $n=2m-1$ and we define
\begin{multline} \label{gdef}
g_{m,u}(z,a,w,b)= \\
u|K|\frac{mr+b+w}{\varsigma(|R|uw)}
\bigg(e^{\frac{u|K|}{2}((rm+a)w-((m-1)r+b)z)}
\Gauss{1-\frac{w+b}{r}-m}{1+\frac{z+a}{r}+m}{\frac{1-e^{|R|uz}}{1-e^{-|R|uw}}} \\
-e^{\frac{u|K|}{2}(((m-1)r+b)z-(rm+a)w)}
\Gauss{1-\frac{w+b}{r}-m}{1+\frac{z+a}{r}+m}{\frac{1-e^{-|R|uz}}{1-e^{|R|uw}}}\bigg)
\end{multline}

then we have
\begin{equation} \label{oddprefactors}
c_{n,a,b}(z,w)=
\frac{(z\mathcal{S}(|R|uz))^m}{(1+\frac{z+a}{r})_m}
\frac{(w\mathcal{S}(|R|uw))^m}{(1+\frac{w+b}{r})_m}
[g_{m,u}(z,a,w,b)-g_{m,u}(w,b,z,a)]
\end{equation}

\subsection{Hypergeometric Function Identities}

In this section we prove two lemmas about hypergeometric functions.  These are exactly lemmas 16 and 17 from \cite{OP2}.
For completeness, we present a slightly extended discussion here.

Note that the definition we have used of Gauss's hypergeometric function,
\begin{equation}
\Gaussfull{a}{b}{c}{z}=\sum_{i=0}^\infty\frac{(a)_i(b)_i}{(y)_ii!}z^i, |z|<1,
\end{equation}
is only valid for $|z|<1$.  The Euler integral gives an extension of this definition, valid on $\C$ with a cut along $[1,\infty)$:
\begin{equation} \label{eulerint}
\Gaussfull{a}{b}{c}{z}
=
\frac{\Gamma(c)}{\Gamma(b)\Gamma(c-b)}
\int_0^1 t^{b-1}(1-t)^{c-b-1}(1-tz)^{-a}dt.
\end{equation}

In the definitions of $f$ and $g$ above, the only use of $\Gaussfull{a}{b}{c}{z}$ used is with $b=1$.
In this case, the hypergeometric function is degenerate.  The hypergeometric differential equation is:

and the function

Because of this degeneration, when we analytically continue \ref{eulerint} through $[1,\infty)$, we only add elementary terms to our solution.
Lemmas \ref{lemma16} and \ref{lemma17} captures this behavior.

We will need the $\beta$-function integral:
\begin{equation} \label{betafunction}
\int_0^1t^{p-1}(1-t)^{q-1}dt=\frac{\Gamma(p)\Gamma(q)}{\Gamma(p+q)}
\end{equation}

\begin{lemma} \label{lemma16}
For $z\notin [0,\infty)$ we have:
\begin{multline}
\Gauss{-x}{y+1}{z} \notag= \\
1-\Gauss{-y}{x+1}{\frac{1}{z}}
+\frac{(1-z)^{x+y}}{(-z)^y}
\frac{\Gamma(y+1)\Gamma(x+1)}{\Gamma(x+y+1)}
\end{multline}
\end{lemma}
\begin{proof}
Since one of our $b=1$, the Euler integral (\ref{eulerint}) simplifies, and then integrating by parts and substituting $v=zt$ gives:
\begin{eqnarray} \label{firstline}
\Gauss{-x}{y+1}{z} &=& y
\int_0^1 (1-t)^{y-1}(1-tz)^xdt \\
&=& 1-x\int_0^1(1-t)^y(1-tz)^{x-1}zdt \notag\\
&=& 1-x\int_0^1 (1-\frac{v}{z})^y(1-v)^{x-1}dv \label{secondlastline} \\
&&+ x\int_z^1 (1-\frac{v}{z})^y(1-v)^{x-1}dv  \label{lastline}
\end{eqnarray}
In (\ref{secondlastline}), we recognize the first integral (\ref{firstline}) with $x$ and $y$ interchanged and $1/z$ replacing $z$, giving the first term of (\ref{lemma16}).  We transform (\ref{lastline}) into the beta integral (\ref{betafunction}) by making the substitution $v=z+(1-z)t$:
\begin{eqnarray*}
x\int_z^1 (1-\frac{v}{z})^y(1-v)^{x-1}dv
&=&\frac{(1-z)^{x+y}}{(-z)^y} x\int_0^1 t^y(1-t)^{x-1}dt \\
&=&\frac{(1-z)^{x+y}}{(-z)^y} \frac{\Gamma(y+1)\Gamma(x+1)}{\Gamma(x+y+1)}
\end{eqnarray*}
\end{proof}

A similar argument shows
\begin{lemma} \label{lemma17}  For $z\notin [0,\infty)$ we have:
\begin{multline}
x\Gauss{-x+1}{y+1}{z}= \notag \\
\frac{y}{z}\Gauss{-y+1}{x+1}{\frac{1}{z}}
+\frac{(1-z)^{x+y-1}}{(-z)^y}
\frac{\Gamma(y+1)\Gamma(x+1)}{\Gamma(x+y)}
\end{multline}
\end{lemma}
\begin{proof} We again use the Euler integral (\ref{eulerint}), this time making the substitution $v=tz$ immediately:
\begin{eqnarray*}
x\Gauss{-x+1}{y+1}{z} &=& xy\int_0^1(1-t)^{y-1}(1-tz)^{x-1}dt \\
&=&\frac{xy}{z}\bigg(\int_0^1\left(1-\frac{v}{z}\right)^{y-1}(1-v)^{x-1}dv \\
&&-\int_z^1\left(1-\frac{v}{z}\right)^{y-1}(1-v)^{x-1}dv \bigg)\\
 &=&\frac{y}{z}\Gauss{-y+1}{x+1}{\frac{1}{z}}\\
&&+\frac{(1-z)^{x+y-1}}{(-z)^y}
\frac{\Gamma(y+1)\Gamma(x+1)}{\Gamma(x+y)}
\end{eqnarray*}
\end{proof}

\subsection{Symmetry }

\begin{lemma} \label{lemma18} The functions $f_{m, u}(z,a,w,b)$ and $g_{m,u}(z,a,w,b)$ are analytic in a neighborhood of the origin and symmetric under interchanging $(z,a)$ with $(w,b)$.
\end{lemma}
\begin{proof}
We begin by applying (\ref{lemma16}) to each of the hypergeometric terms in the definition of $f_{m, u}(z,a,w,b)$ (\ref{fdef}).  The resulting hypergeometric terms are exactly those appearing in $f_{m,u}(w,b,z,a)$.  Furthermore, the terms coming from the 1 in (\ref{lemma16}) will exactly cancel the $\varsigma$ terms in the definition of $f_{m,u}(z,a,w,b)$.  So it remains to show that the two $\Gamma$ terms will cancel.  the actual factors of $\Gamma$ will be the same for each, and so we must show
\begin{equation*}
e^{\frac{u|K|}{2}((rm+a)w-(rm+b)z)}
\frac{(1-v)^{x+y}}{(-v)^y}
=e^{\frac{-u|K|}{2}((rm+a)w-(rm+b)z)}
\frac{(1-v^\prime)^{x+y}}{(-v^\prime)^y}
\end{equation*}
where
\begin{gather*}
x=m+\frac{w+b}{r},\quad y=m+\frac{z+a}{r} \\
 v=\frac{1-e^{|R|uz}}{1-e^{-|R|uw}}\quad \textrm{ and } \quad v^\prime=\frac{1-e^{-|R|uz}}{1-e^{|R|uw}}
 \end{gather*}

Since $v=e^{|R|u(z+w)}v^\prime$ and $(1-v)=e^{|R|uz}(1-v^\prime)$ we have
\begin{equation*}
\frac{(1-v)^{x+y}}{(-v)^y}=e^{u|R|((x+y)z-y(z+w)}\frac{(1-v^\prime)^{x+y}}{(-v^\prime)^y}
\end{equation*}

Then we note that
\begin{eqnarray*}
|R|u[(x+y)z-y(z+w)]&=&|R|u(xz-yw) \\
&=&u|K|((mr+b)z-(mr+a)w)
\end{eqnarray*}

So we've see that $f_{m,u}$ has the desired symmetry where its definition makes sense and Lemma \ref{lemma16} is applicable.

Note that if $m=-j-b/r$, for some positive integer $j$, then the Pochhammer symbols in the denominator will eventual have a $z$ term, and we will apparently have a simple pole there.  We will assume for now that $m$ does not have this form; then the only possibility singularities of $f$ come from the singularities of Gauss's hypergeometric function at 1 and $\infty$.  Because of the Euler integral, the hypergeometric function is well defined and analytic away from the cut along $[1,\infty)$, and so $f$ will be defined unless the arguments of the hypergeometric function lie on that cut.

So $f_{m,u}(z,a,w,b)$ is well defined unless one of
\begin{equation*}
\frac{1-e^{-|R|uz}}{1-e^{|R|uw}}, \frac{1-e^{|R|uz}}{1-e^{-|R|uw}}\approx -\frac{z}{w}
\end{equation*}

fall on $[1,\infty)$, and $f_{m,u}(w,b,z,a)$ is well defined unless one of
\begin{equation*}
\frac{1-e^{-|R|uw}}{1-e^{|R|uz}}, \frac{1-e^{|R|uw}}{1-e^{-|R|uz}} \approx -\frac{w}{z}
\end{equation*}

fall on $[1,\infty)$.  So the only place where $f_{m,u}(z,a,w,b)$ might not be analytic around the origin is the divisor $z+w=0$.  However, we can calculate explicitly that $f_{m,u}$ is in fact analytic here, and hence in a neighborhood of the origin.

The apparent singularities for certain values of $m$ can be seen to not happen, as the they don't appear in the $w$.

When $z=-w$, the arugment of the hypergeometric function becomes one, and then applying
\begin{equation} \label{Gaussevaluation}
\Gaussfull{a}{b}{c}{1}=\frac{\Gamma(c)\Gamma(c-a-b)}{\Gamma(c-a)\Gamma(c-b)}
\end{equation}
 to the hypergeometric function appearing in $f_{m,u}(z,a,-z,b)$, we see that
\begin{equation}
\Gauss{\frac{z-b}{r}-m}{1+\frac{z+a}{r}+m}{1}=
\frac{\Gamma(1+\frac{z+a}{r}+m)\Gamma(2m+\frac{a+b}{r})}
{\Gamma(2m+1+\frac{a+b}{r})\Gamma(\frac{z+a}{r}+m)} =
\frac{\frac{z+a}{r}+m}{2m+\frac{a+b}{r}}
\end{equation}

We see that
\begin{eqnarray*}
f_{m,u}(z,a,-z,b)&=&
\varsigma(-u|K|(2rm+a+b)z)
\left(\frac{\frac{z+a}{r}+m}{2m+\frac{a+b}{r}}-\frac{1}{2}\right) \\
&=&\varsigma(-u|K|(2rm+a+b)z)\frac{z+\frac{a-b}{2}}{2mr+a+b} \\
&=&-\frac{z+\frac{a-b}{2}}{mr+\frac{a+b}{2}}\sinh(u|K|(rm+\frac{a+b}{2})z)
\end{eqnarray*}

In particular, we will need the case when $m=a=b=0$.  This is easily seen to be $-u|K|z^2$.

Similarly, if we apply (\ref{Gaussevaluation}) to the hypergeometric functions appearing in the definition of $g$, we see that
\begin{equation}
\Gauss{1+\frac{z-b}{r}-m}{1+\frac{z+a}{r}+m}{1}=
\frac{\Gamma(1+\frac{z+a}{r}+m)\Gamma(2m+\frac{a+b}{r}-1)}
{\Gamma(2m+\frac{a+b}{r})\Gamma(\frac{z+a}{r}+m)} =
\frac{\frac{z+a}{r}+m}{2m+\frac{a+b}{r}-1}
\end{equation}

Plugging this in, we see
\begin{eqnarray*}
g_{m,u}(z,a,-z,b)&=&u|K|\frac{mr+b-z}{\varsigma(-|R|uz)}
\frac{\frac{z+a}{r}+m}{2m+\frac{a+b}{r}-1}
\varsigma(-u|K|(2rm-r+a+b)z)
\end{eqnarray*}
\end{proof}

Here, we will need the case that $n=-1$, so $m=0$ and  $a+b=r$.  Then, this simplifies to
\begin{equation} \label{g0}
g_{0,u}(z,a,-z,r-a)=-\frac{u^2|K|^2z(z+a)(z-b)}{\varsigma(|R|uz)}
\end{equation}

We see everything commutes except possibly the constant term appearing in $\mathcal{E}_0(u|K|(z+w))$.  First, note that $\mathcal{E}_0$ appears in two cases: $n=a=b=0$ and $n=1, a+b=r$.  In each case, we must keep track of the prefactors appearing in (\ref{comAexpand}), (\ref{oddprefactors}) and (\ref{evenprefactors}).  In both cases, we have $m=0$, and so these last two prefactors are identically 1.

The prefactor from (\ref{comAexpand}) is
$$(r\gamma(-\kk))^{(a+b)/r}
\frac{z}{z+a}\frac{w}{w+b}
 \mathcal{S}(|R|uz)^\frac{z+a}{r}
\mathcal{S}(|R|uw)^\frac{w+b}{r}$$.

Since we only care about the value of the factor along the singularity $z+w=0$, we note that in the first case, since $a=b=0$, this factor is $1$, while in the second case $a+b=r$ means that this factor is
\begin{equation} \label{gprefactor}
r\gamma(-\kk)\frac{z}{z+a}\frac{z}{z-b}\mathcal{S}(|R|uz).
\end{equation}

Putting it all together, we see that in the $f$ case ($n$ even), the prefactors are identically one when $z+w=0$, and so the commutator is really
\begin{equation}
\frac{f_{0,u}(z, 0, w, 0)}{\varsigma(|K|u(z+w))}-\frac{f_{0,u}(w,0,z,0)}{\varsigma(|K|u(z+w))}=zw\delta(z,-w)
\end{equation}

Similarly, in the $g$ case, note that the prefactors from Equation \ref{gprefactor} together with Equation \ref{g0} gives:
$$-\gamma(-\kk)ru^2|K|^2z^3\frac{\mathcal{S}(|R|uz)}{\varsigma(|R|uz)}=-\gamma(-\kk)u|K|z^2,$$
and so the commutator gives us
$$
-\gamma(-\kk)u|K|z^2\left(\frac{1}{\varsigma(u|K|(z+w))}-\frac{1}{\varsigma(u|K|(w+z))}\right)=\gamma(-\kk)zw\delta(z,-w).$$

\bibliography{biblio}

\newcommand{\etalchar}[1]{$^{#1}$}
\begin{thebibliography}{HHP{\etalchar{+}}07}

\bibitem[AB84]{AB}
M.F. Atiyah and R.~Bott.
\newblock The moment map and equivariant cohomology.
\newblock {\em Topology}, 23(1):1--28, 1984.

\bibitem[AGV]{AGV}
Dan Abramovich, Tom Graber, and Angelo Vistoli.
\newblock {G}romov-{W}itten theory of {D}eligne-{M}umford stacks.

\bibitem[AJTa]{AJT2}
Elena Andreini, Yunfeng Jiang, and Hsian-Hua Tseng.
\newblock {G}romov-{W}itten theory of \'etale gerbes {I}: root gerbes.
\newblock {\em Forthcoming}.

\bibitem[AJTb]{AJT1}
Elena Andreini, Yunfeng Jiang, and Hsian-Hua Tseng.
\newblock On {G}romov-{W}itten theory of root gerbes.
\newblock {\em arXiv:0812.4477}.

\bibitem[ALR07]{ALR}
Alejandro Adem, Johann Leida, and Yongbin Ruan.
\newblock {\em Orbifolds and stringy topology}.
\newblock Cambridge University Press, 2007.

\bibitem[BCS05]{BCS}
Lev Borisov, Linda Chen, and Gregory Smith.
\newblock The orbifold {C}how ring of toric {D}eligne-{M}umford stacks.
\newblock {\em J. Amer. Math. Soc.}, 18(1):193–--215, 2005.

\bibitem[Bre]{Breen2}
Lawrence Breen.
\newblock Notes on $1$- and $2$- gerbes.
\newblock {\em math.CT/0611317}.

\bibitem[Bre94]{Breen1}
Lawrence Breen.
\newblock On the classification of $2$-gerbes and $2$-stacks.
\newblock {\em AstŽrisque}, (225), 1994.

\bibitem[CC]{CC}
Charles Cadman and Renzo Cavalieri.
\newblock Gerby localization, $\mathbb{Z}_3$ {H}odge integrals and the {GW}
  theory of $\bold{C}^3/\mathbb{Z}_3$.
\newblock {\em arXiv:0705.2158}.

\bibitem[CR02]{CRGW}
Weimin Chen and Yongbin Ruan.
\newblock Orbifold {G}romov-{W}itten theory.
\newblock {\em Orbifolds in mathematics and physics (Madison, WI, 2001)}, 2002.

\bibitem[CR04]{CR}
Weimin Chen and Yongbin Ruan.
\newblock A new cohomology theory of orbifold.
\newblock {\em Comm. Math. Phys.}, 248(1):1--31, 2004.

\bibitem[EHY95]{EHY}
Tohru Eguchi, Kentaro Hori, and Sung-Kil Yang.
\newblock Topological $\sigma$ models and large-$n$ matrix integral.
\newblock {\em Internat. J. Modern Phys. A}, 10(29):4203--4224, 1995.

\bibitem[ELSV01]{ELSV}
Torsten Ekedahl, Sergei Lando, Michael Shapiro, and Alek Vainshtein.
\newblock {H}urwitz numbers and intersections on moduli spaces of curves.
\newblock {\em Invent. Math.}, 146(2):297--327, 2001.

\bibitem[EY94]{EY}
Tohru Eguchi and Sung-Kil Yang.
\newblock The topological $\bold{C}\mathbb{P}^1$ model and the large-$n$ matrix
  integral.
\newblock {\em Modern Phys. Lett. A}, 9(31):2893--2902, 1994.

\bibitem[FMN]{FMN}
Barbara Fantechi, Etienne Mann, and Fabio Nironi.
\newblock Smooth toric {DM} stacks.
\newblock {\em arXiv:0708.1254}.

\bibitem[FW01]{FW}
Igor Frenkel and Weiqiang Wang.
\newblock {V}irasoro algebra and wreath product convolution.
\newblock {\em J. Alg.}, 242:656--671, 2001.

\bibitem[Gir71]{Giraud}
Jean Giraud.
\newblock {\em Cohomologie non abŽlienne}.
\newblock Springer-Verlag, 1971.

\bibitem[GP]{GP}
Tom Graber and Rahul Pandharipande.
\newblock Localization of virtual classes.
\newblock {\em Invent. Math.}, 135(2):487--518.

\bibitem[HHP{\etalchar{+}}07]{HHPSA}
Simeon Hellerman, Andr\'e Henriques, Tony Pantev, Eric Sharpe, and Matt Ando.
\newblock Cluster decomposition, {T}-duality, and gerby {CFT}s.
\newblock {\em Adv. Theor. Math. Phys.}, 11(5):751--818, 2007.

\bibitem[JK02]{JK}
Tyler Jarvis and Takashi Kimura.
\newblock Orbifold quantum cohomology of the classifying space of a finite
  group.
\newblock In {\em Orbifolds in mathematics and physics (Madison, WI, 2001)},
  volume 310 of {\em Contemp. Math.}, pages 123–--134. Amer. Math. Soc.,
  Providence, RI, 2002.

\bibitem[JPT]{JPT}
Paul Johnson, Rahul Pandharipande, and Hsian-Hua Tseng.
\newblock Abelian {H}urwitz-{H}odge integrals.
\newblock {\em arXiv:0803.0499}.

\bibitem[JT]{JT}
Yunfeng Jiang and Hsian-Hua Tseng.
\newblock On {V}irasoro constraints for orbifold {G}romov-{W}itten theory.
\newblock {\em arXiv:0704.2009}.

\bibitem[Mac95]{M}
I.G. Macdonald.
\newblock {\em Symmetric Functions and {H}all Polynomials}.
\newblock Oxford University Press, second edition, 1995.

\bibitem[MJD]{MJD}
T.~Miwa, M.~Jimbo, and E.~Date.
\newblock {\em Solitons: Differential Equations, Symmetries and Infinite
  Dimensional Algebras}.
\newblock Cambirdge University Press.

\bibitem[MT]{MTequivariant}
Todor Milanov and Hsian-Hua Tseng.
\newblock Equivariant orbifold structures on the projective line and integrable
  hierarchies.
\newblock {\em arXiv:0707.3172}.

\bibitem[Oko00]{OHur}
Andrei Okounkov.
\newblock {T}oda equations for {H}urwitz numbers.
\newblock {\em Math. Res. Lett.}, 7(4):447--453, 2000.

\bibitem[OP06a]{OP2}
Andrei Okounkov and Rahul Pandharipande.
\newblock The equivariant {G}romow-{W}itten theory of $\proj^1$.
\newblock {\em Ann. of Math. (2)}, 163(2):561--605, 2006.

\bibitem[OP06b]{OP1}
Andrei Okounkov and Rahul Pandharipande.
\newblock {G}romov-{W}itten theory, {H}urwitz theory, and completed cycles.
\newblock {\em Ann. of Math. (2)}, 163(2):517--560, 2006.

\bibitem[OP06c]{OP3}
Andrei Okounkov and Rahul Pandharipande.
\newblock Virasoro constraints for target curves.
\newblock {\em Invent. Math.}, 163(1):47--108, 2006.

\bibitem[Pan00]{PToda}
Rahul Pandharipande.
\newblock The {T}oda equations and the {G}romov-{W}itten theory of the
  {R}iemann sphere.
\newblock {\em Lett. Math. Phys.}, 53(1):59--74, 2000.

\bibitem[PRY]{PRY}
Jianzhong Pan, Yongbin Ruan, and Xiaoqin Yin.
\newblock Gerbes and twisted orbifold quantum cohomology.
\newblock {\em Sci. China Ser. A}, 51(6):995--1016.

\bibitem[QW]{QW}
Zhenbo Qin and Weiqiang Wang.
\newblock Hilbert schemes of points on the minimal resolution and soliton
  equations.

\bibitem[Rosa]{Rossi1}
Paolo Rossi.
\newblock {G}romov-{W}itten invariants of target curves via symplectic field
  theory.
\newblock {\em arXiv:0709.2860}.

\bibitem[Rosb]{Rossi2}
Paolo Rossi.
\newblock {G}romov-{W}itten theory of orbicurves, the space of tri-polynomials
  and symplectic field theory of seifert fibrations.
\newblock {\em arXiv:0808.2626}.

\bibitem[Rua]{RuanDiscreteTorsion}
Yongbin Ruan.
\newblock Discrete torsion and twisted orbifold cohomology.
\newblock {\em Journal of Symplectic Geometry}, 2(1):1--24.

\bibitem[Sha]{S1}
Eric Sharpe.
\newblock Discrete torsion.
\newblock {\em Phys. Rev. D}, 68(12).

\bibitem[Vaf]{V}
Cumrun Vafa.
\newblock Modular invariance and discrete torsion on orbifolds.
\newblock {\em Nuclear Phys. B}, 273(3-4):592--606.

\end{thebibliography}
\end{document}